\documentclass[aap,press]{apt}
\usepackage[colorlinks = true,linkcolor = blue,urlcolor  = blue,citecolor = blue,anchorcolor = blue]{hyperref}
\usepackage{graphicx}
\usepackage{subcaption}
\usepackage{amsmath}

\authornames{C.~DURANSTANTI AND S.~KHAN AND A.~OLENKO} 
\shorttitle{Anisotropic Spherical Random Field} 

\begin{document}

\title{Exploring the Structure of Anisotropic Random Fields on the Sphere} 

\authorone[Sapienza Universit\'a di Roma]{Claudio Durastanti}
\authortwo[La Trobe University]{Shahid Khan}
\authorthree[La Trobe University]{Andriy Olenko}

\addressone{Sapienza Universit\'a di Roma, Roma, 00185, Italy} 
\emailone{claudio.durastanti@uniroma1.it} 
\addresstwo{La Trobe University, Melbourne, 3086, Victoria, Australia} 
\emailtwo{shahid.khan@latrobe.edu.au} 
\addressthree{La Trobe University, Melbourne, 3086, Victoria, Australia} 
\emailthree{a.olenko@latrobe.edu.au} 

\begin{abstract}
This paper addresses fundamental questions about the existence and structure of isotropic and anisotropic spherical random fields and the properties of composite transformations of isotropic fields. It introduces a hierarchical structure comprising several classes of anisotropic random fields, using geometric and spectral approaches, and provides a geometric--spectral classification of second-order anisotropy on $\mathbb{S}^2$. We characterize feasible isotropic and anisotropic scenarios and identify those that, in principle, cannot occur, in contrast to the case of random fields in Euclidean space. Furthermore, we establish connections between geometric invariance properties and the structure of the spectral covariance matrix of spherical random fields. The obtained results lay the foundations for applications in statistics and data analysis, providing guidance on feasible models, their spectral representations, and appropriate sampling schemes for parameter estimation or anisotropy hypothesis testing for spherical data.
\end{abstract}



\section{Introduction} \label{s1} 

Spherical random fields are one of the main stochastic models and tools used to study spherical and directional data in applications such as cosmology, geophysics, and environmental science, see \cite{Christakos, Duque, Hristopulos, Malyarenko2, marpecbook, planck18}. One of their most notable recent applications is the analysis of the Cosmic Microwave Background radiation (CMB), where spherical random fields are used to model temperature and polarization fluctuations observed across the celestial sphere. Properties of such fields are essential for testing cosmological theories concerning the early universe and the nature of dark matter and dark energy, see, for example, \cite{planck18}. While classical CMB analysis relies heavily on the assumption of isotropy, increasing empirical evidence suggests the presence of anisotropic features, motivating the development of stochastic models that go beyond rotation-invariant frameworks. However, despite their practical relevance, anisotropic spherical random fields have remained much less studied than their isotropic counterparts.

Properties of spherical random fields are typically characterized through their covariance functions, corresponding spectral decompositions, and angular power spectra, which provide insight into both spatial dependencies and stationarity. There is an extensive body of recent literature on isotropic models, beginning with several highly cited publications \cite{ Gneiting, Kuriki, Lang, Leonenko99, marpecbook, MarinucciVadlamani, Spodarev, Yadrenko}. By contrast, real spherical data often display non-stationary or anisotropic behavior that cannot be captured within rotation-invariant models. Several recent contributions proposed specific anisotropic constructions \cite{Allard, Blake, BouDurMarTod24, Caponera, Duque, Jeong}. Despite it, a unified theoretical framework describing anisotropy in terms of geometric invariance properties and the structure of the full spectral covariance matrix is still missing.
In particular, while isotropy admits a complete spectral characterization
through the diagonalization property \eqref{iso}, no analogous classification has been available for anisotropic fields that
systematically links geometric symmetry breaking with the precise
pattern of cross-correlations among harmonic coefficients. The paper fills this gap by providing a geometric--spectral
classification of second-order anisotropy classes on the sphere~$\mathbb{S}^2$.

Since non-stationary random fields can have different properties at different locations, they are a challenging object both for mathematical investigations and statistical estimation. To deal with it, the most widely used non-stationary model in Euclidean space, elliptic anisotropy, assumes that for all $\mathbf{x}_1, \mathbf{x}_2 \in \mathbb{R}^{3}$ the covariance function of a random field $T(\cdot),$ which measures the statistical dependency between values of the field at different points, has the following form
\begin{equation} \label{aa}
    C(\mathbf{x}_1, \mathbf{x}_2) = \mathbf{E}\left((T(\mathbf{x}_1)-\mathbf{E}(T(\mathbf{x}_1)) ({T(\mathbf{x}_2)}-\mathbf{E}(T(\mathbf{x}_2)))^*\right) = B\left(\|A(\mathbf{x}_1 - \mathbf{x}_2)\|\right),
\end{equation}
where the symbol $*$ denotes the complex conjugation,  $A$ is a $3 \times 3$ invertible matrix, and $B(\cdot)$ is a real-valued function. This construction can be interpreted as a global deformation of an underlying isotropic field, obtained by uniformly stretching or squeezing the ambient space in prescribed directions. Its advantage is that a single transformation governs anisotropic behavior throughout the entire domain and can be described through a small number of parameters.
However, it crucially relies on global linear deformations in Euclidean space and therefore lacks a direct analogue on the sphere, where no such uniform stretching or squeezing can be consistently defined. However, $\mathbb{S}^2$ possesses a natural transformation group of rotations, $SO(3),$ and other transformations can be used to alter its geometry. This raises the question of whether one can construct spherical random fields that are anisotropic yet remain homogeneous in a generalized sense, namely, statistically invariant under a prescribed class of transformations while breaking full rotational symmetry.

Another related problem is constructing anisotropic spherical random fields with complex dependence structures by using composite models based on simpler isotropic components. Multiplicative, convolutional, and deformed random fields, for instance, provide flexible mechanisms for introducing directional effects and spatial heterogeneity while retaining analytical tractability~\cite{Berzin, Perrin,  rm21,  srk12, Sampson}.

Spectral decompositions based on spherical harmonics provide one of the main analytical tools for studying isotropic random fields on the sphere, allowing covariance structure and scale-dependent behavior to be expressed in a precise and tractable form, see \cite{marpecbook} and references therein. In applications such as geophysics and cosmology, these representations form the basis of data analysis. In particular, in cosmological studies only a single realization of the CMB sky is available, making spectral coefficients the primary objects for statistical inference.
The angular power spectrum offers insight into the covariance function of the field across various angular scales, revealing the amount of power (``energy'') associated with each spherical harmonic degree, $\ell.$ Lower values of $\ell$ correspond to broad global patterns, while higher values capture finer, localized details. In the presence of anisotropy, deviations from the isotropic angular power spectrum reveal directional or spatially varying features. Therefore, understanding how anisotropy affects the structure and dependence of these coefficients is essential for both modeling and inference on the sphere.

This paper addresses several general questions about the construction and characterization of anisotropic spherical fields. First, it introduces hierarchical classes of anisotropic fields and clarifies their existence and interrelations based on geometric properties. Then, it presents a general framework for constructing anisotropic fields via composite transformations of isotropic models. Finally, it derives the corresponding spectral representations and identifies how specific geometric forms of anisotropy or composite transformations translate into structural constraints on harmonic coefficients.
These results provide a rigorous theoretical foundation for understanding spherical anisotropy and offer tools for future methodological developments and applications.

\section{Definitions and assumptions} \label{s2}
This section introduces the main notations, definitions, and assumptions used in this work. It starts by establishing geometric and probabilistic preliminaries to formalize the notions of isotropy and anisotropy for spherical random fields. Then, it introduces the harmonic analysis framework for studying the spectral representation of these fields.

Let $\mathbb{N}_0:=\mathbb{N} \cup \{0\}.$ $\mathbb{S}^2:= \{\mathbf{x} \in \mathbb{R}^3,\, ||\mathbf{x}||=1 \}$ will denote the unit sphere in the real three-dimensional Euclidean space $\mathbb{R}^3,$ where $||\cdot||$ stands for the Euclidean norm in $\mathbb{R}^3.$ For two points $\mathbf{x}_1$ and $\mathbf{x}_2\in \mathbb{S}^2,$ the notation
$\gamma (\mathbf{x}_1,\mathbf{x}_2)$ will be used to denote the great-circle distance, which is the angle between two rays from the origin to these points. The notations $O(\mathbf{x}_0, r):= \{\mathbf{x} \in \mathbb{S}^2:\, \gamma(\mathbf{x}_0,\mathbf{x})=r \}$  and $D(\mathbf{x}_0, r):= \{\mathbf{x} \in \mathbb{S}^2:\, \gamma(\mathbf{x}_0,\mathbf{x})<r\},$ $r \in (0, \pi),$ will respectively represent a circle and an open disk on $\mathbb{S}^2$ with the center at $\mathbf{x}_0$ and the angular radius~$r.$  Let $SO(3)$ denote the group of rotations on $\mathbb{R}^3.$ A rotation $\rho \in SO(3)$ maps points on $\mathbb{S}^2$ to other points on $\mathbb{S}^2$. It can be identified by three angular coordinates, known as Euler angles, $(\alpha,\beta,\psi)$, where $\alpha \in [0, 2\pi),$ $\beta \in \left[0,\pi\right]$, and $\psi \in  [0, 2\pi),$ see, for example,~\cite{vmk}. Then $(-\psi, -\beta, -\alpha)$ are Euler angles of $\rho^{-1} $ obtained by reversing both the order of rotations and the signs of the angles. All random variables will be defined on the same probability space denoted by $(\Omega, \mathcal{F}, P).$

\subsection{Spherical random fields and anisotropy}
This subsection introduces the notion of spherical random fields and develops a hierarchy of isotropy and anisotropy classes.

\begin{definition}
A measurable function $T(\omega,\mathbf{x}):\Omega \times \mathbb{S}^2 \rightarrow \mathbb{C} $ is called a spherical second-order random field, if for all $\mathbf{x}\in \mathbb{S}^2$ the second moments of the random variables $T(\cdot,\mathbf{x})$ are finite, i.e. $\mathbf{E} |T(\cdot, \mathbf{x}) |^2  < \infty.$
\end{definition}
Henceforth, we will refer to $T(\cdot,\mathbf{x})$ simply as a random field and denote it by $T(\mathbf{x}),$ $\mathbf{x} \in \mathbb{S}^2.$ Without loss of generality, it will be assumed that $\mathbf{E}T(\mathbf{x})=0,$ so that the covariance function fully characterizes second-order properties and will serve as the main tool for distinguishing isotropic and anisotropic behavior.

In what follows, we will also represent points of~$\mathbb{S}^2$
using their spherical coordinates $(\theta,\varphi)$, where $\theta \in [0,\pi]$
is the colatitude and $\varphi \in [0,2\pi)$ is the longitude. When using these coordinates, we denote the random field and the spherical harmonics by the same symbols, $T(\cdot)$ and $Y_{\ell m}(\cdot),$ as in their Euclidean-coordinate representations.

From now on, unless otherwise stated, we restrict our attention to random fields with first- and second-order moments satisfying the following assumptions.
\begin{condition}\label{cond1}
 $T(\mathbf{x})$ is a spherical real-valued second-order random field with zero mean and a continuous covariance function on $\mathbb{S}^2\times \mathbb{S}^2,$ defined by
 $C(\mathbf{x}_1,\mathbf{x}_2):=\mathbf{E}(T(\mathbf{x}_1) {T(\mathbf{x}_2)}).$
\end{condition}
\begin{condition}\label{cond2}  $C(\mathbf{x}_1,\mathbf{x}_2)$   is not an identical constant for all points on $\mathbb{S}^2.$
\end{condition}

The last assumption eliminates cases of degenerate fields that consist of the same constants or a single random variable replicated at all locations.
The real-valued case is considered for simplicity and its popularity in most of applications, but many of the following results are also valid for complex-valued random fields.

\begin{definition}\label{isotropy}
The random field $T(\mathbf{x})$ is called 2-weakly isotropic, if for all $\mathbf{x}_{1}, \mathbf{x}_{2} \in \mathbb{S}^2$ and all rotations $\rho \in SO(3),$ it holds \[\mathbf{E}(T(\rho\mathbf{x_1}))=\mathbf{E}(T(\mathbf{x_1})), \quad C(\mathbf{x}_1,\mathbf{x}_2)=C(\rho \mathbf{x}_1,\rho \mathbf{x}_2).\] \end{definition}
Isotropy formalizes the idea that the probabilistic structure of the field does not privilege any specific direction on the sphere. In the sequel, a 2-weakly isotropic field will simply be referred to as isotropic.

The random field $T(\mathbf{x})$ is Gaussian if for all $k \in \mathbb{N}$ and any collections of points $ \mathbf{x}_1, \mathbf{x}_2, \ldots, \mathbf{x}_k \in \mathbb{S}^2$ the random vector $(T(\mathbf{x}_1),\ldots, T(\mathbf{x}_k))$ has a multivariate Gaussian distribution.

This work focuses on the second-order properties of spherical random fields. Therefore, all results can be directly applied to Gaussian random fields. Moreover, any rejection of isotropy based on second-order considerations remains valid for non-Gaussian fields, as such negative conclusions do not rely on higher-order distributional properties.


\begin{remark}\label{rem2}
A random field $T(\mathbf{x}),$ $ \mathbf{x} \in \mathbb{S}^2,$ is isotropic if and only if its covariance function depends only on the geodesic distance between locations. More precisely, there exists a real-valued function $B:[0,\pi]\mapsto\mathbb{R},$ such that for all $\mathbf{x}_1, \mathbf{x}_2 \in \mathbb{S}^2$ it holds
\begin{equation} \label{eq1}
    C(\mathbf{x}_1, \mathbf{x}_2)=B(\gamma (\mathbf{x}_1,\mathbf{x}_2)).
\end{equation}
This is a fundamental characterization that leads to the diagonal structure of the covariance operator in the spherical harmonic basis, which underpins the spectral representation in Section~\ref{subs}.
\end{remark}

\begin{definition}
A random field $T(\mathbf{x}),$ $\mathbf{x} \in \mathbb{S}^2,$ is anisotropic, if its covariance function $C(\mathbf{x}_1, \mathbf{x}_2)$ does not satisfy \eqref{eq1}.
\end{definition}
Thus, the covariance functions of an anisotropic random field depend not only on the angular distance between points but also on their specific locations. Unlike the isotropic case, anisotropic covariance structures do not admit reductions to a one-dimensional kernel, making their analysis substantially more complex both theoretically and statistically.

\begin{definition}
A random field $T(\mathbf{x}),$ $ \mathbf{x} \in \mathbb{S}^2,$ is called isotropic at a point $\mathbf{x}_0 \in \mathbb{S}^2,$ if there exists a real-valued function $B_{\mathbf{x}_0}:[0,\pi]\mapsto\mathbb{R}$ such that, for all $\mathbf{x} \in \mathbb{S}^2$
\begin{equation} \label{eq 2}
    C(\mathbf{x}_0, \mathbf{x})=B_{\mathbf{x}_0}(\gamma (\mathbf{x}_0,\mathbf{x})).
\end{equation}
\end{definition}

\begin{definition}\label{def_sani}
A random field $T(\mathbf{x}),$ $\mathbf{x} \in \mathbb{S}^2,$ is strictly anisotropic, if there are no points $\mathbf{x}_0 \in \mathbb{S}^2,$ where it is isotropic in the sense of \eqref{eq 2}.
\end{definition}

\begin{definition}\label{defca}
A random field $T(\mathbf{x}),$ $ \mathbf{x} \in \mathbb{S}^2,$ is called isotropic on a circle $O(\mathbf{x}_0,r)$ with fixed center $\mathbf{x}_0$ and radius $r \in (0,\pi),$ if $C(\mathbf{x}_0, \mathbf{x})$ is constant for all $\mathbf{x}\in O(\mathbf{x}_0,r).$
\end{definition}

\begin{remark}
Global isotropy implies isotropy at every point $\mathbf{x}_0 \in \mathbb{S}^2$. Only if a field is isotropic at all points with the same covariance function $B(\cdot)$ that does not depend on $\mathbf{x}_0$, then it is (globally) isotropic. One may ask whether a field can be isotropic at each point separately while failing to be (globally) isotropic, i.e.  with different covariance functions $B(\cdot)$ at two locations.

Strictly anisotropic fields exhibit no local rotational symmetry at any point on the sphere, so their covariance structure cannot be expressed as a function of angular distance anywhere on $\mathbb{S}^2$. Isotropy on a circle is a weaker notion than pointwise isotropy. The pointwise isotropy at the point $\mathbf{x}_0\in \mathbb{S}^2$ requires a covariance function to be constant along each circle~$O(\mathbf{x}_0,r)$, $r \in (0,\pi).$

\end{remark}

\begin{definition}
A random field $T(\mathbf{x}),$ $\mathbf{x} \in \mathbb{S}^2,$ is super anisotropic, if there exists no circle $O(\mathbf{x}_0,r) \subset \mathbb{S}^2$ on which it is isotropic in the sense of Definition~{\rm \ref{defca}}.
\end{definition}

\begin{remark}
Let us note that the aforementioned isotropy and anisotropy concepts can also be equivalently characterized by the level sets $L_\mu(\mathbf{x}_1)
:=
\left\{
\mathbf{x}_2 \in \mathbb{S}^2
\,\middle|\,
C(\mathbf{x}_1,\mathbf{x}_2) = \mu
\right\},
$ $\mu\in\mathbb{R},$ of a covariance function $C(\mathbf{x}_1, \mathbf{x}_2),$ when the first argument is fixed and the function is considered as a function of the second argument $\mathbf{x}_2$. For example, a random field is isotropic if, for each $\mathbf{x}_1$ and $\mu$, the corresponding level sets $L_\mu(\mathbf{x}_1)$ are either empty or unions of circles and the same for any two points of the sphere. A random field is isotropic at a point $\mathbf{x}_0 \in \mathbb{S}^2,$ if all level sets $L_\mu(\mathbf{x}_0)$ are either empty or unions of circles. A random field is super-anisotropic if none of the level sets of its covariance function includes a circle.
    \end{remark}

In recent literature, considerable attention has been given to axially symmetric spherical random fields and their applications, see, for example, \cite{Alegría, Buhmann, Hitczenko, Jones, Stein} and the references therein. The statistical properties of these fields are invariant with respect to rotations around a fixed axis.
\begin{definition}\label{def_axial}
A random field is axially symmetric around the axis passing through the point $\mathbf{x}_0$ and its antipode if, for all $\mathbf{x}_1, \mathbf{x}_2 \in \mathbb{S}^2$ and $\alpha \in [0, 2\pi)$, its covariance function satisfies the relation
\[ C(\rho_{\mathbf{x}_0, \alpha} \mathbf{x}_1, \rho_{\mathbf{x}_0, \alpha} \mathbf{x}_2) = C(\mathbf{x}_1, \mathbf{x}_2),
    \]
where $\rho_{\mathbf{x}_0, \alpha}$ denotes a rotation around that axis through $\mathbf{x}_0$ by the angle $\alpha$.
\end{definition}

\begin{definition}\label{def_axial_lon} An axially symmetric random field around the axis passing through the point $\mathbf{x}_0$ and its antipode is said to be longitudinally reversible around that axis if, for all $\mathbf{x}_1, \mathbf{x}_2 \in \mathbb{S}^2,$ it holds
\[   C(\mathbf{x}_1, \mathbf{x}_2) =
    C(\mathbf{x}_1, \rho_{\mathbf{x}_0, -2\alpha(\mathbf{x}_1, \mathbf{x}_2; \mathbf{x}_0)} \mathbf{x}_2),
    \]
    where $\alpha(\mathbf{x}_1, \mathbf{x}_2; \mathbf{x}_0)$ denotes the angle between two geodesic arcs connecting $\mathbf{x}_0$ with $\mathbf{x}_1$ and $\mathbf{x}_2,$ respectively.
\end{definition}

This property implies that the covariance structure at the point $\mathbf{x}_1$ remains unchanged if the field is reflected over the great circle determined by the point $\mathbf{x}_1$ and the axis through the point $\mathbf{x}_0,$ capturing a type of mirror symmetry around this axis.

\begin{remark}The available literature focuses only on axially symmetric fields around the axis passing through the pole $\mathbf{0}.$ For such fields, their covariance functions can be expressed in terms of functions $\tilde{C}(\cdot)$ defined on $[0,\pi]^2\times(-2\pi,2\pi)$ as
\begin{equation}\label{axiall}
C(\mathbf{x}_1, \mathbf{x}_2)=\tilde{C}(\theta_1,\theta_2,\varphi_2-\varphi_1),
\end{equation}
where the points $\mathbf{x}_1$ and $\mathbf{x}_2\in \mathbb{S}^2$ have spherical coordinates $(\theta_1,\varphi_1)$ and $(\theta_2, \varphi_2),$ respectively. In the case of longitudinally reversible random fields around the axis passing through the pole $\mathbf{0}$, in addition to (\ref{axiall}), it is required that
\begin{equation}\label{reversib}\tilde{C}(\theta_1,\theta_2,\varphi_2-\varphi_1)=\tilde{C}(\theta_1,\theta_2,\varphi_1-\varphi_2),
\end{equation}
which implies that the covariance function in (\ref{axiall}) is an even function of its third argument, i.e.,
\[C(\mathbf{x}_1, \mathbf{x}_2)=\tilde{C}(\theta_1,\theta_2,|\varphi_2-\varphi_1|).\]
\end{remark}

The introduced notions define progressively weaker forms of rotational invariance. The following hierarchy of isotropy classes holds, with the corresponding inverse inclusions for anisotropy classes:
\begin{align}
\textsc{Isotropic} & \,\subset \,\textsc{Longitudinally Reversible}
\,\subset \,\textsc{Axially Symmetric} \nonumber\\
&\subset\, \textsc{Point-isotropic}
\, \subset\, \textsc{Circle-isotropic.}\label{isotr}
\end{align}

The subsequent sections will prove that each inclusion is strict and will analyze in detail the properties of the corresponding anisotropic classes.

\subsection{Spectral representation and harmonic analysis}\label{subs}
This subsection introduces the spectral representation of spherical random fields using spherical harmonics. Isotropy corresponds to a complete decorrelation of harmonic coefficients across different degrees and orders, whereas anisotropy manifests itself through structured cross-correlations in the spectral domain.

The considered random fields can be expanded in the mean-square sense as the following Laplace series, see \cite[p. 123]{marpecbook} or \cite[p. 73]{Yadrenko}:
    \begin{equation}  \label{spec}
        {T}(\theta ,\varphi )=\sum_{\ell\in  \mathbb{N}_0}\sum_{m=-\ell}^{\ell}a_{\ell m} {Y}_{\ell m}(\theta
        ,\varphi ),
    \end{equation}
where $\{ {Y}_{\ell m}(\theta,\varphi )\}$ are the complex spherical harmonics and $\{ a_{\ell m}\}$ are the associated random harmonic coefficients.

The spectral representation~(\ref{spec}) converges in the Hilbert space $L_{2}(\Omega \times
    \mathbb{S}^2,\sin \theta d\theta d\varphi ),$ that is,%
    \[
    \lim_{L\rightarrow \infty }\mathbf{E}\left( \int\limits_{\mathbb{S}^2}\left(
    {T}(\theta ,\varphi )-\sum_{\ell=0}^{L}\sum_{m=-\ell}^{\ell} a_{\ell m} {Y}_{\ell m}(\theta ,\varphi
    )\right) ^{2}\sin \theta d\theta d\varphi \right) =0.
    \]%
The expansion (\ref{spec}) also converges in the Hilbert space $L_{2}(\Omega ),$ that is,
    \[
    \lim_{L\rightarrow \infty }\mathbf{E}\left( {T}(\mathbf{x})-\sum_{\ell=0}^{L}%
    \sum_{m=-\ell}^{\ell}a_{\ell m}{Y}_{\ell m}(\mathbf{x})\right) ^{2}=0.
    \]

For $\ell\in  \mathbb{N}_0$ and $0\leq m\leq \ell,$ the complex spherical harmonics $Y_{\ell m}(\theta,\varphi)$ are defined as
    \begin{equation} \label{ylm}
    {Y}_{\ell m}(\theta ,\varphi ):= e^{im\varphi}
    \left[\frac{(2\ell+1)(\ell-m)!}{4\pi(\ell+m)!}\right] ^{1/2}
    P_{\ell}^{m}(\cos\theta),
    \end{equation}
where $P_{\ell}^{m}(\cdot)$ denotes the associated Legendre polynomial with the indices $\ell$ and $m.$  For negative orders, we adopt the convention
\[
P_{\ell}^{-m}(x)
=
(-1)^m
\frac{(\ell-m)!}{(\ell+m)!}
P_{\ell}^{m}(x),
\qquad m>0.
\]

The integer number $\ell \in \mathbb{N}_0$ is the multipole (or degree) number, determining the total number of nodal lines in the polar direction and controlling the overall angular scale of variation, while $m \in \{-\ell, \dots, \ell\}$ is the azimuthal (or order) number, specifying the number of nodal lines in the longitudinal direction.

The spherical harmonics have the following properties
\begin{equation}\label{Ylms}
    \int_{0}^{\pi }\int_{0}^{2\pi } {Y}_{\ell m}(\theta ,\varphi ) {Y}^{\ast }_{\ell^{\prime}m^{\prime }}(\theta ,\varphi )\sin \theta d\varphi d\theta  =\delta
    _{\ell}^{\ell^{\prime }}\delta _{m}^{m^{\prime }}, \qquad \text{(Orthonormality)}
    \end{equation}
    \begin{equation}\label{Ylll}
         {Y}_{\ell m}^{\ast }(\theta ,\varphi ) =(-1)^{m} {Y}_{\ell(-m)}(\theta ,\varphi ),\quad |m|\leq\ell,\qquad \text{(Conjugation relation)}
    \end{equation}
    \begin{equation*}\label{Ylmst}
     {Y}_{\ell m}(\pi -\theta ,\varphi +\pi ) = (-1)^{\ell} {Y}_{\ell m}(\theta ,\varphi ), \qquad\qquad\qquad \text{(Parity)} \qquad
     \end{equation*}
    \begin{equation}\label{Yl0}
        {Y}_{\ell m}(\mathbf{0})=\delta^{0}_m {Y}_{\ell0}(\mathbf{0})=\sqrt{\frac{2\ell+1}{4\pi}}P_{\ell}(1)=\sqrt{\frac{2\ell+1}{4\pi}}, \ \ \ \text{(Value at the north pole)}
    \end{equation}
where $\delta _{\ell}^{\ell^{\prime }}$ is the Kronecker delta function and $\mathbf{0}$ corresponds to the pole with $\theta=0.$

The random harmonic coefficients $a_{\ell m}$ in the Laplace series~(\ref{spec}) are given by the $L_2$-projection of $T$ onto the spherical harmonics, namely,%
    \begin{equation} \label{almss}
        a_{\ell m}:=\int_{0}^{\pi }\int_{0}^{2\pi } {T}(\theta ,\varphi) {Y}_{\ell m}^{\ast
        }(\theta ,\varphi )\sin \theta d\theta d\varphi = \int\limits_{\mathbb{S}^2}T(\mathbf{x}) Y^{*}_{\ell m}(\mathbf{x}) d\mathbf{s}(\mathbf{x}),
        \end{equation}
where $d \mathbf{s} (\mathbf{x})=\sin \theta d\theta d\varphi$ is the differential area element.
For real-valued spherical random fields, the conjugation symmetry \eqref{Ylll} yields $a_{\ell (-m)} = (-1)^m \, a_{\ell m}^{\ast}.$

A spherical random field $T$ is isotropic if and only if its harmonic coefficients are uncorrelated across different degrees and orders and satisfy \cite{Baldi}
    \begin{equation}\label{iso}
    \mathbf{E}a_{\ell m}a_{\ell^{\prime }m^{\prime }}^{\ast }=\delta _{\ell}^{\ell^{\prime
    }}\delta_{m}^{m^{\prime }}C_{\ell},\quad -\ell\leq m\leq \ell,\quad -\ell^{\prime }\leq
    m^{\prime }\leq \ell^{\prime },\quad \ell,\ell^{\prime}\in \mathbb{N}_0.
    \end{equation}
   In particular, isotropy implies that the second moments of the coefficients depend only on the multipole index $\ell$, that is, $C_{\ell}:=\mathbf{E}|a_{\ell m}|^{2},$ $ -\ell\leq m\leq \ell.$
   The sequence $\left\{ C_{\ell}\right\}$ is called the angular power spectrum of the isotropic field $T(\cdot).$

We provide simple examples that illustrate several of the considered models by plotting realizations of random fields and selected covariance functions. The examples demonstrate that generating anisotropic fields from the considered models is a computationally feasible process.

We begin with an example that generates an isotropic random field, which serves as the basis for all subsequent examples. These examples illustrate how the realization and covariance structure of composite random fields change after applying the studied transformations to the base field.

\begin{example}\label{ex1} Figure~\ref{fig1a} displays a realization of a real-valued spherical isotropic random field~$T(\cdot)$, see technical details in Appendix~\ref{appA}.
The simulations are based on the angular power spectrum
$C_\ell = {1}/(\ell(\ell + 1))$ for\ $0<\ell \le \ell_{max},$ $\ell_{max}=1535,$ and $C_\ell = 0$ otherwise.

\begin{figure}[!hb]
  \centering
  \begin{subfigure}{0.53\textwidth}
    \centering
    \includegraphics[width=\textwidth, trim={0 0 0 0cm}, clip]{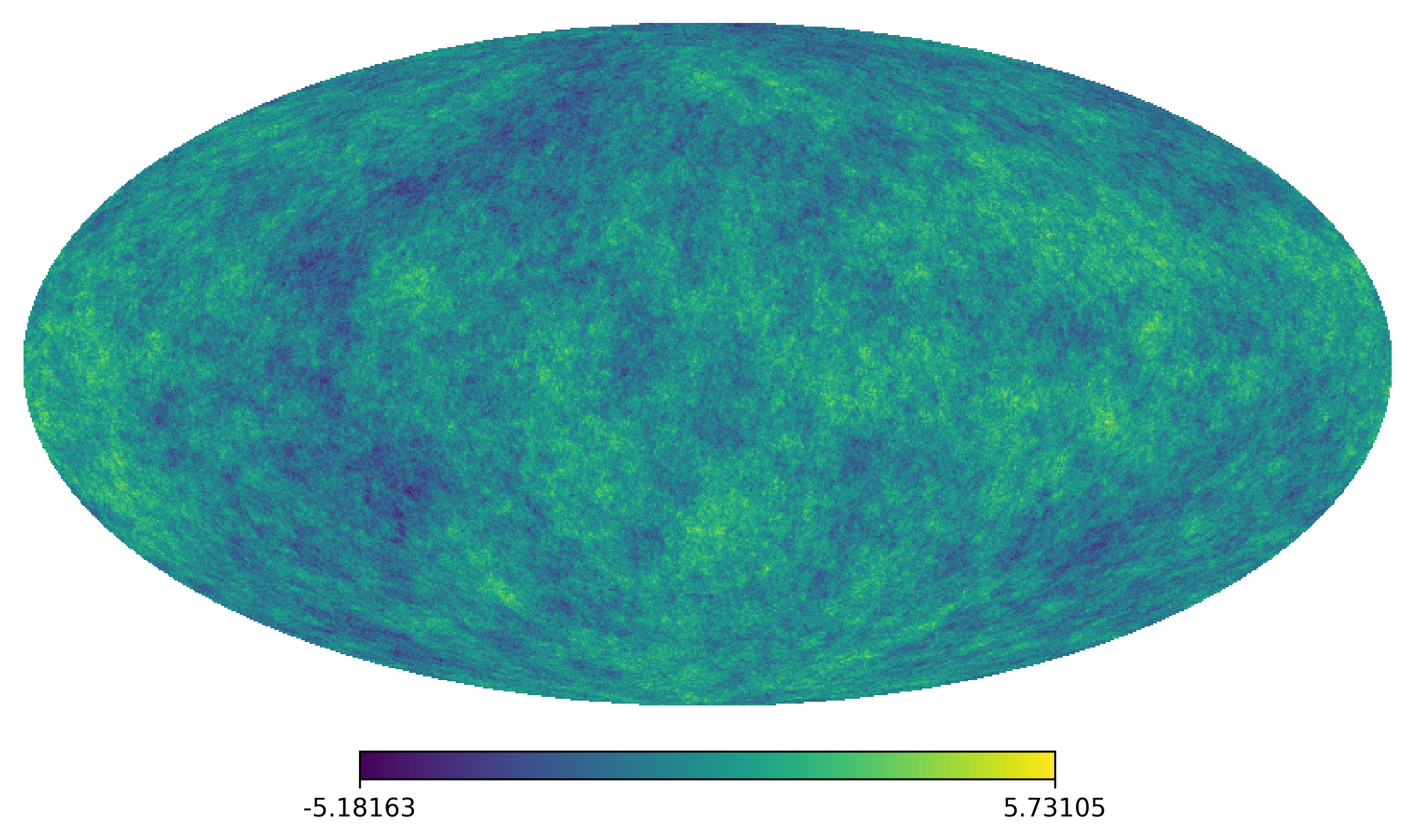}\\[0.2cm]
    \caption{\it Realization of the field}
    \label{fig1a}
  \end{subfigure}%
  \begin{subfigure}{0.61\textwidth}
    \centering
    \includegraphics[width=\textwidth, trim={4cm 0 0 5.5cm}, clip]{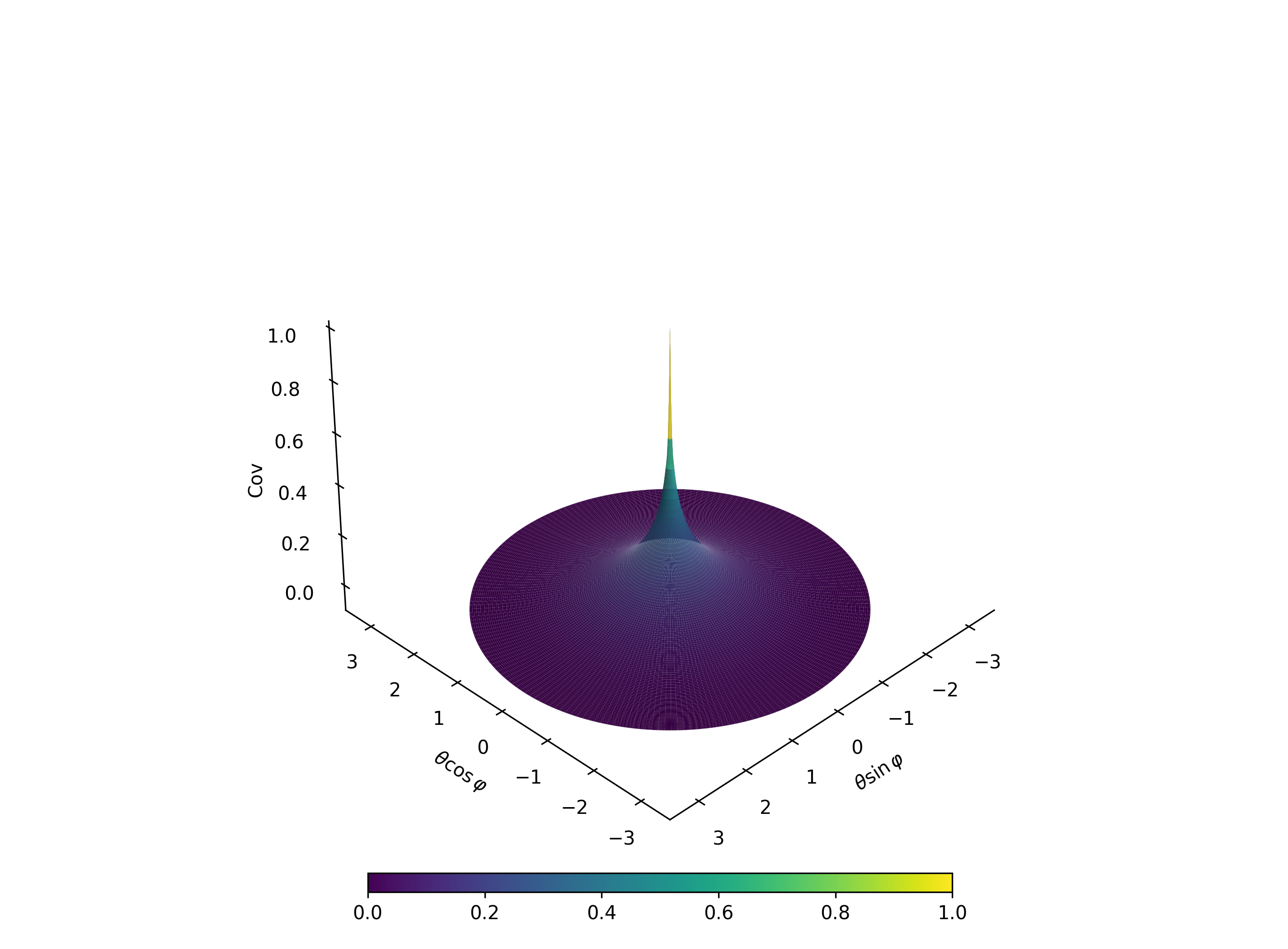}
    \caption{\it Covariance function}
      \label{fig1b}
  \end{subfigure}
\caption{\it Isotropic random field $T(\mathbf{x})$ with $C_\ell = {1}/(\ell(\ell + 1))$.}
  \label{fig1}
\end{figure}
 Figure~\ref{fig1b} plots the corresponding covariance function $C(\mathbf{0}, \mathbf{x})$ between the pole $\mathbf{0}$ and locations $\mathbf{x}$ with spherical coordinates $(\theta, \varphi).$ By isotropy, the covariance function depends solely on the great-circle distance between $\mathbf{0}$ and $\mathbf{x}$, and hence is a function of the angular separation $\theta\in[0,\pi]$. In those cases where the realisation of the random field is given, the realisations of harmonic coefficients $a_{\ell m}$ can be obtained via the projection formula~(\ref{almss}).
\end{example}

\section{Exploration of transformation-invariant anisotropy}

This section examines random fields with covariance functions $C(\mathbf{x}_0, \mathbf{x}),$ $ \mathbf{x} \in \mathbb{S}^2,$ that are invariant under specified transformations, for all reference points $\mathbf{x}_0 \in \mathbb{S}^2.$

Let us begin with a broad meta-definition.

\begin{definition}
A statistic of a spherical random field is rotationally invariant if, for any points $\mathbf{x}_0,\mathbf{x}_0^{\prime} \in \mathbb{S}^2,$ there exists a rotation $\rho \in SO(3)$ that maps $\mathbf{x}_0$ into $\mathbf{x}_0^{\prime}$ and  the statistics at $\mathbf{x}_0$ are equal to the corresponding statistics at $\mathbf{x}_0^{\prime}$ after such mapping.
\end{definition}

For example, if a statistic $S(\mathbf{x}_0)=\mathbf{E}T(\mathbf{x}_0)$ represents the mean of the field $T(\cdot)$, then it is rotationally invariant if and only if the random field has a constant mean. A further example, if $S(\mathbf{x}_0,\mathbf{x})$ denotes a statistic in $\mathbf{x}_0$, rotational invariance requires that, for each $\mathbf{x}_0,\mathbf{x}_0^{\prime}\in\mathbb{S}^2$, there exists $\rho\in SO(3)$ such that, for all~$\mathbf{x}\in\mathbb{S}^2$,
\[S(\mathbf{x}_0,\mathbf{x})=S(\rho\mathbf{x}_0,\rho\mathbf{x})= S(\mathbf{x}_0^{\prime}, \mathbf{x}'),\]
where $\mathbf{x}'=\rho\mathbf{x}$.
Thus, the means and covariance functions of all isotropic random fields are rotationally invariant.

The following result demonstrates that anisotropic random fields on the sphere exhibit fundamentally different behavior from both isotropic spherical fields and elliptically anisotropic fields in Euclidean space, such as those defined by~(\ref{aa}).

\begin{theorem} \label{t1}
If $T(\mathbf{x}),$ $\mathbf{x} \in \mathbb{S}^2,$ is an anisotropic field, then its covariance function is not rotationally invariant.
\end{theorem}

\begin{corollary} \label{CC1}
If $T(\mathbf{x}),$ $\mathbf{x} \in \mathbb{S}^2,$ is not isotropic on some circle $O(\mathbf{x}_0,r),$ then its covariance function is not rotationally invariant when restricted to circles of radius $r$.
\end{corollary}

As the isotropic behavior at different points may be different, the next result clarifies the structure of fields that are isotropic at every point on $\mathbb{S}^2.$

\begin{theorem}\label{th2}
If $T(\mathbf{x}),$ $\mathbf{x} \in \mathbb{S}^2,$ is an isotropic random field at every point of $\mathbb{S}^2,$ then it is (globally) isotropic.
\end{theorem}

\begin{remark}
The conclusion of Theorem~\ref{th2} is specific to the spherical setting. On~$\mathbb{S}^2$, isotropy at every point enforces global isotropy through the transitive action of $SO(3)$. In more general geometries or under other symmetry assumptions, locally isotropic behavior does not necessarily results in global isotropy.
\end{remark}

Theorems~\ref{t1} and~\ref{th2} can be extended to a broader class of transformations on the sphere that go beyond simple rotations.

\begin{definition}
A circular scaling function is a continuous bijection
$
\lambda : [0,\pi] \to [0,\pi]
$
satisfying $\lambda(0)=0$ and $\lambda(\pi)=\pi$.
\end{definition}

\begin{definition}
A circular scaling transformation of the sphere at a point $\mathbf{x}_0 \in \mathbb{S}^2$ is a mapping of $\mathbb{S}^2$ onto itself such that every point $\mathbf{x}$ lying on a great circle through $\mathbf{x}_0$ is mapped to a point $\mathbf{x}^{\prime}$ on the same semicircle, according to
$\gamma(\mathbf{x}^{\prime},\mathbf{x}_0)
=
\lambda\left(\gamma(\mathbf{x},\mathbf{x}_0)\right),
$
where $\lambda(\cdot)$ is a circular scaling function.
\end{definition}

\begin{definition}\label{vbn}
A statistical characteristic of a spherical random field is said to be rotationally and circular-scaling invariant if, for any
$\mathbf{x}_0,\mathbf{x}_0^{\prime} \in \mathbb{S}^2$, there exists a composition of rotation and circular scaling transformations that maps $\mathbf{x}_0$ and the statistical characteristic at $\mathbf{x}_0$ onto $\mathbf{x}_0^{\prime}$ and the corresponding characteristic.
\end{definition}

\begin{theorem}\label{th3}
If $T(\mathbf{x})$, $\mathbf{x}\in\mathbb{S}^2$, is anisotropic, then its covariance function is not invariant under combined rotations and circular scaling transformations.
\end{theorem}

\begin{remark}
This result shows that isotropic random fields are the only spherical random fields whose covariance functions remain invariant under the class of transformations generated by rotations and circular scalings.
\end{remark}

Spherical random fields play a central role in cosmology, where they are used to model the Cosmic Microwave Background (CMB). A fundamental cosmological question concerns testing Einstein's cosmological principle, which asserts that the Universe is statistically identical in all directions and at all locations. The above results indicate that, if the cosmological principle holds, the Universe cannot exhibit anisotropic statistical behavior.

We next examine the existence of the anisotropic field types defined in Section~\ref{s2}.

\begin{definition}\label{TV}
Let $V(\mathbf{x})$, $\mathbf{x} \in \mathbb{S}^2$, be a continuous deterministic real-valued function, and let $T(\mathbf{x})$, $\mathbf{x} \in \mathbb{S}^2$, be an isotropic second-order random field with zero mean. The $V$-multiplicative transformation of $T$ is defined by
\[
T_V(\mathbf{x}) := V(\mathbf{x})\, T(\mathbf{x}), \qquad \mathbf{x} \in \mathbb{S}^2 .
\]
\end{definition}

\begin{proposition}\label{prop1}
Let $T(\mathbf{x})$, $\mathbf{x} \in \mathbb{S}^2$, be an isotropic random field with a covariance function given by \eqref{eq1} for some function $B(\cdot)$. Let $V(\mathbf{x})$, $\mathbf{x} \in \mathbb{S}^2$, be a deterministic real-valued function that is non-vanishing on $\mathbb{S}^2$ and such that, for every $\mathbf{x}_0 \in \mathbb{S}^2$, there exists a circle $O(\mathbf{x}_0,r)$ (with a radius $r$ possibly depending on $\mathbf{x}_0$) for which $B(r) \neq 0$ and $V(\mathbf{x})$ is not constant on that circle.
Then, the $V$-multiplicatively transformed field $T_V(\mathbf{x})$ is strictly anisotropic.
\end{proposition}

Intuitively, the multiplicative transformation $T_V(\mathbf{x})$ destroys local rotational symmetry whenever $V$ varies along some circle around a point $\mathbf{x}_0$, where the isotropic field $T$ has nonzero covariance. Since this occurs at every point on the sphere, no location retains isotropy, and thus $T_V$ is strictly anisotropic.

A stronger version of Proposition~\ref{prop1} can be obtained under additional assumptions on the function $V(\mathbf{x})$:

\begin{proposition}\label{prop2}
Let $T(\mathbf{x})$, $\mathbf{x} \in \mathbb{S}^2$, be an isotropic random field with a covariance function given by \eqref{eq1} for a strictly positive function $B(\cdot)$. Let $V(\mathbf{x})$, $\mathbf{x} \in \mathbb{S}^2$, be a deterministic non-vanishing real-valued function, which is not constant on every circle $O(\mathbf{x}_0,r)$, for all $\mathbf{x}_0 \in \mathbb{S}^2$ and $r \in (0,\pi)$.
Then, the $V$-multiplicatively transformed field $T_V(\mathbf{x})$ is a super anisotropic random field.
\end{proposition}

For instance, the power exponential covariance function~\cite{Gneiting} can serve as $B(\cdot)$ in Propositions~\ref{prop1} and \ref{prop2}, since it is strictly positive on $\mathbb{S}^2$.

The following example gives a function $V(\cdot)$ satisfying both Propositions~\ref{prop1} and~\ref{prop2}.
\begin{example} \label{ex2}
The function
\begin{equation}\label{V}
  V(\theta, \varphi) = \theta (\pi - \theta)\, \varphi (2 \pi - \varphi) + 1, \quad \theta \in [0, \pi],\ \varphi \in [0, 2\pi),
\end{equation}
is strictly positive and non-constant on every circle $O(\mathbf{x}_0,r),$ $r \in (0,\pi)$ of $\mathbb{S}^2$. Therefore, it satisfies the conditions of both Propositions~\ref{prop1} and \ref{prop2} and can be used to produce strict and super anisotropic random fields.
\end{example}

While Propositions~\ref{prop1} and \ref{prop2} produce classes of fields lacking isotropy at each point or circle, one can also construct anisotropic fields that maintain isotropy on specific geometric structures, such as circles of a fixed radius.

\begin{proposition} \label{prop3}
For any fixed $r_0 \in (0,\pi)$, there exists an anisotropic random field that is not isotropic at any point of $\mathbb{S}^2$, but is isotropic on all circles $O(\mathbf{x},r_0)$, $\mathbf{x} \in \mathbb{S}^2$.
\end{proposition}

Finally, if a random field is axially symmetric around the axis passing through a point $\mathbf{x}_0$ and its antipode, then, by Definition~\ref{def_axial}, there exists a function $B(\cdot)$ such that for all $\mathbf{x}_2 \in \mathbb{S}^2$,
$
C(\mathbf{x}_0, \mathbf{x}_2) = C(\mathbf{x}_0, \rho_{\mathbf{x}_0, \alpha} \mathbf{x}_2) = B(\gamma(\mathbf{x}_0, \mathbf{x}_2)).
$
Thus, axially symmetric fields form a subclass of fields that are isotropic at a point. The following result clarifies the variance structure of axially symmetric random fields. It illustrates a weaker form of isotropy at a point that holds for second-order moments, even when the field may be anisotropic. It immediately follows from Definition~\ref{def_axial}, since, for all $\mathbf{x} \in \mathbb{S}^2$ and $\alpha \in [0, 2\pi)$, it holds that
\[
\mathrm{Var}(T(\mathbf{x}))
= C(\mathbf{x}, \mathbf{x})
= C(\rho_{\mathbf{x}_0,\alpha}\mathbf{x},\, \rho_{\mathbf{x}_0,\alpha}\mathbf{x})
= \mathrm{Var}(T(\rho_{\mathbf{x}_0,\alpha}\mathbf{x})).
\]

\begin{proposition} \label{pp}
The variance of an axially symmetric random field, with an axis of symmetry passing through $\mathbf{x}_0$, is constant on each circle $O(\mathbf{x}_0, r)$, for all $r \in (0, \pi)$.
\end{proposition}

\section{On admissible isotropic locations}
If a random field is isotropic, then it is isotropic at every point of $\mathbb{S}^2$ and invariant under all rotations. By contrast, the previous section demonstrated the existence of strictly anisotropic random fields that admit no points of isotropy. Consequently, a natural question arises about intermediate cases. Specifically, whether for any set $J\subset\mathbb{S}^2$ there exists a random field $T(\mathbf{x})$ that is isotropic at each point of $J$ and anisotropic on its complement, $\mathbb{S}^2 \setminus J.$ In particular, for isotropic random fields one has $J = \mathbb{S}^2$, while for strictly anisotropic fields $J = \varnothing$.

Since our attention is restricted to random fields with continuous covariance functions, for simplicity it is assumed throughout that the set $J$ is compact and connected. For a set $A \subset \mathbb{S}^2$, define the function $d(\mathbf{x},A)$, $\mathbf{x} \in \mathbb{S}^2$, as the shortest angular distance from $\mathbf{x}$ to $A$,
\[
d(\mathbf{x}, A) := \inf_{\mathbf{x}^{\prime} \in A} \gamma(\mathbf{x}, \mathbf{x}^{\prime}).
\]
If $A$ is compact, this function is continuous and  $\inf$ may be replaced by $\min$.

The next auxiliary result is also of independent interest, as it provides necessary and sufficient conditions for pointwise isotropy of the following composite random field.
\begin{lemma}\label{op} Let us consider the field
\begin{equation}\label{sumT}
T(\mathbf{x}) := T_1(\mathbf{x}) + d(\mathbf{x}, J) \cdot T_2(\mathbf{x}), \quad\mathbf{x} \in \mathbb{S}^2, \end{equation}
where $T_{1}(\cdot)$ and $T_{2}(\cdot)$ are two independent isotropic random fields, such that $T_2$ has a strictly positive covariance function.
Then the field $T(\mathbf{x})$ is isotropic at a point $\mathbf{x}_0 \in \mathbb{S}^2 \setminus J$ if and only if the set $\mathbb{S}^2 \setminus J$ (and thus $J$ itself) is rotationally invariant about the axis determined by the antipodal points $\mathbf{x}_0$ and $\mathbf{x}_0^{\prime}$.
\end{lemma}
Since $J$ is compact, the construction in~\eqref{sumT} yields a continuous covariance function. Lemma~\ref{op} shows that pointwise isotropy at $\mathbf{x}_0$ is entirely determined by the geometry of $J$ as viewed from $\mathbf{x}_0$.

Using the construction of Lemma~\ref{op}, one can explicitly construct a field that is isotropic precisely on a prescribed compact set $J$.
\begin{theorem} \label{prop4}
For any compact set $J \subset \mathbb{S}^2,$ there exists a random field $T(\mathbf{x}),$ that is isotropic at each point of $J$ and anisotropic on $\mathbb{S}^2 \setminus J.$
\end{theorem}
This result shows that pointwise isotropy imposes no global structural constraints on the field: any compact set $J$ may be chosen. As an isotropic random field is not only pointwise isotropic but also axially symmetric with respect to each axis passing through a point and its antipode, a similar question can be posed regarding axial symmetry. In contrast to Theorem~\ref{prop4}, requiring axial symmetry imposes stronger geometric restrictions,  as shown by the following result.
\begin{theorem} \label{yu}
An axially symmetric random field that is not isotropic has a single axis of symmetry.
\end{theorem}

\section{On composite random fields}\label{sec:composite}

This section introduces composite random fields, defined as random fields obtained from an isotropic field via deterministic convolution, pointwise multiplication, or space deformation approaches.

\subsection{Multiplicative models}
First, we investigate $V$-multiplicative transformations of isotropic random fields. Recall that these transformations involve the pointwise multiplication $T_V(\mathbf{x}) = V(\mathbf{x}) T(\mathbf{x})$ of an isotropic random field $T(\cdot)$ and a deterministic function $V(\cdot)$, referred to as a multiplicative filter.

The following result provides a characterization of the isotropy and anisotropy properties of the multiplicative field $T_V(\cdot)$.

\begin{theorem}\label{th4}
Let $T(\mathbf{x}),$ $\mathbf{x} \in \mathbb{S}^2,$ be an isotropic random field with a covariance function given by~\eqref{eq1} with a strictly positive function $B(\cdot).$
Then,
\begin{enumerate}
\item[\rm(1)] if $V(\mathbf{x})$ is constant on $\mathbb{S}^2,$ then $T_V(\cdot)$ is isotropic;
\item[\rm(2)] if $V(\mathbf{x}_0)=0,$ then $T_V$ is isotropic at the point $\mathbf{x}_0;$
\item[\rm(3)] if there exists a point $\mathbf{x}_0$  such that
$
V(\mathbf{x})\neq 0,$ for
all $\mathbf{x}\in\mathbb{S}^2\setminus\{\mathbf{x}_0\},
$
$V(\mathbf{x})$ is constant on every circle $O(\mathbf{x}_0,r)$, $r\in(0,\pi)$, and there are two distinct radii
$r_0^{\prime}\neq r_0^{\prime\prime}$ for which these constants are different, then $T_V(\cdot)$ is isotropic precisely at the antipodal points $\mathbf{x}_0$ and $\mathbf{x}_0^{\prime}$;

\item[\rm(4)] if $V(\mathbf{x})\neq 0$ for all $\mathbf{x}\in\mathbb{S}^2$, and no pair of antipodal points satisfies the conditions in {\rm(3)}, then $T_V(\cdot)$ is strictly anisotropic.
\end{enumerate}
\end{theorem}

\begin{remark}
The assumption of strict positivity of $B(\cdot)$ in Theorem~\ref{th4} ensures that the isotropic base field $T(\cdot)$ contributes nontrivially to the covariance function of the $V$-multiplicative field $T_V(\cdot)$. This assumption can be relaxed: the same conclusions hold if $B(\cdot)$ vanishes only at a finite number of points or on a measure-zero set.
\end{remark}

\begin{example} \label{e3}
Consider the deterministic function $V(\mathbf{x})=V(\theta,\varphi) = \theta(\pi - \theta)+1$ which depends only on $\theta$. Figure~\ref{fig2a} shows a realization of $T_V(\mathbf{x})$ constructed using the isotropic random field $T(\mathbf{x})$ from Example~\ref{ex1}.
\begin{figure}[!hbt]
  \centering
  \begin{subfigure}{0.53\textwidth}
    \centering
    \includegraphics[width=\textwidth, trim={0 0 0 0}, clip]{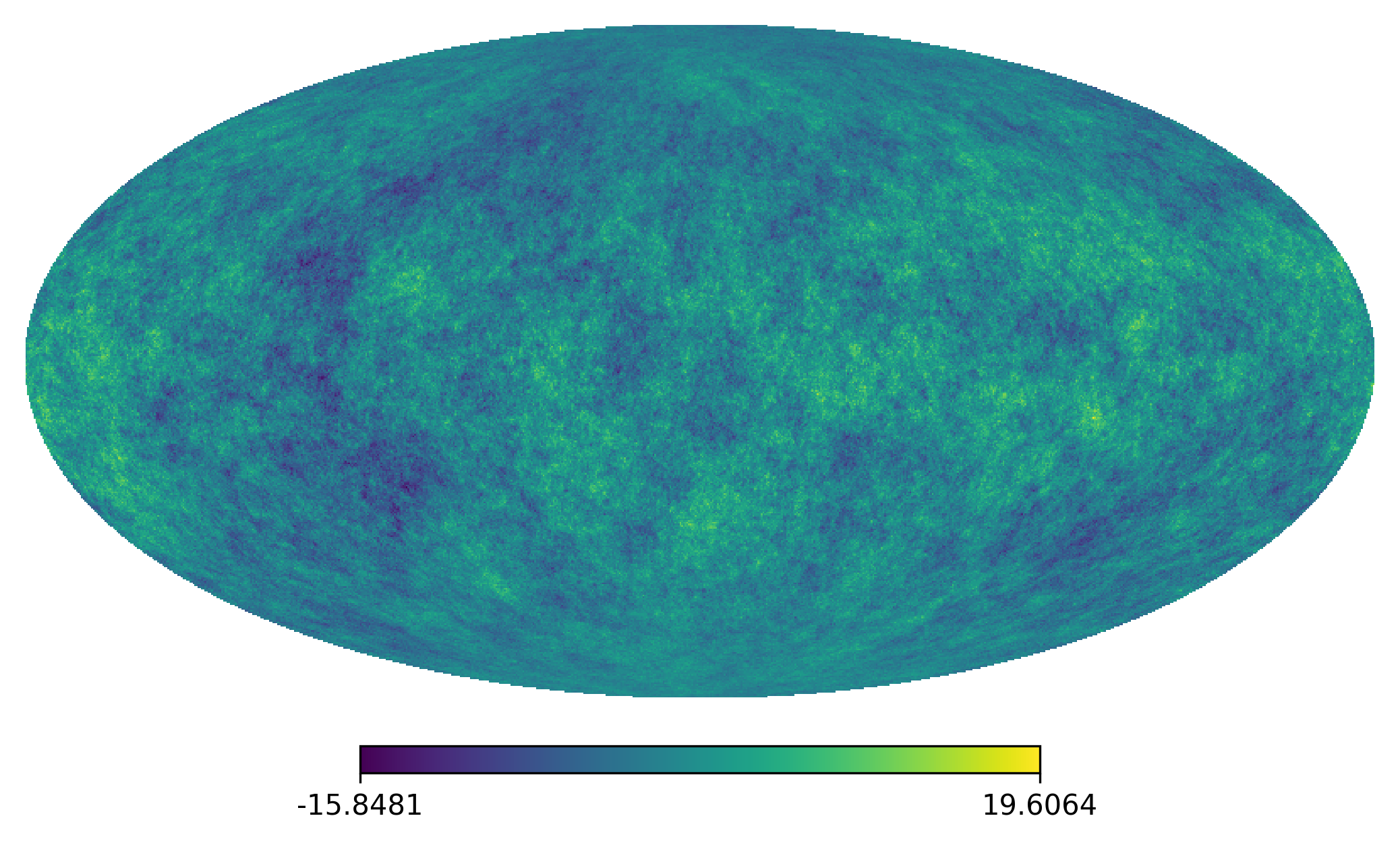}\\[0.2cm]
    \caption{\it Realization of  the field}
    \label{fig2a}
  \end{subfigure}%
  \begin{subfigure}{0.61\textwidth}
    \centering
    \includegraphics[width=\textwidth, trim={4cm 0 0 5.5cm}, clip]{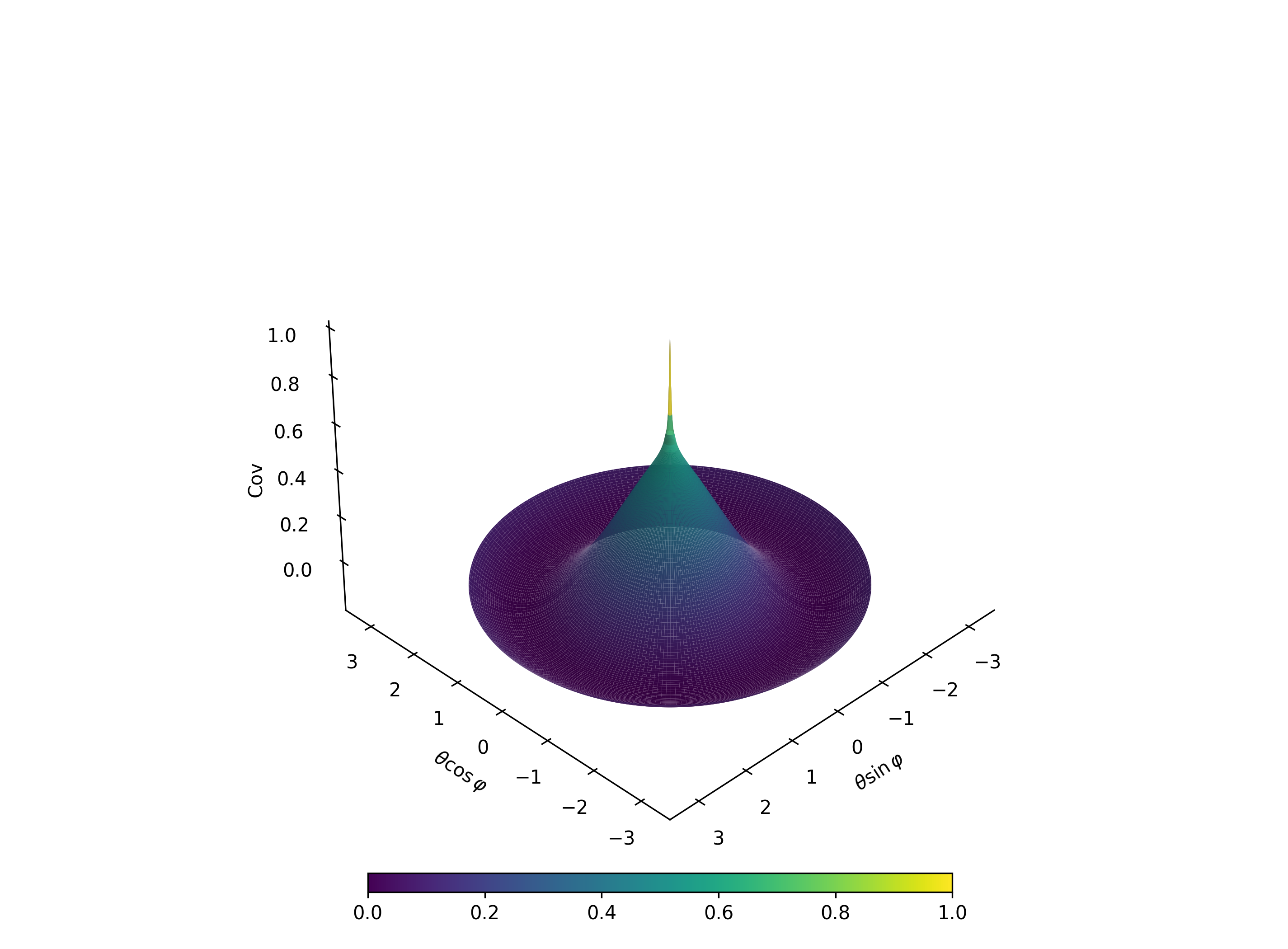}
    \captionsetup{justification=raggedright, singlelinecheck=false, margin=0cm}
    \caption{\it Covariance function  $\text{Cov}(T_V(\mathbf{0}),T_V(\mathbf{x}))$}
    \label{fig2b}
  \end{subfigure}
  \caption{\it Anisotropic field $T_V(\mathbf{x})$ which is isotropic at $\mathbf{0},$ $V(\theta,\varphi) = \theta(\pi - \theta)+1.$ }
  \label{fig2}\end{figure}

As expected, since $V(\mathbf{x}) \ge 1,$ due to multiplication by non-negative values, the realizations in Figure~\ref{fig1a} and~\ref{fig2a} exhibit an identical pattern of "cold" and "hot" regions. However, the ranges of fields values differ, as $ \max_{\theta \in [0,\pi]} V({\theta}) = 1 + \pi^2/4 \approx 3.467.$ By Theorem~\ref{th4} (Case~{\rm(3)}), since $V(\mathbf{x})>0,$ takes constant values on each circle $O(\mathbf{0}, r),$ $r \in (0, \pi),$ and these values differ for some radii, then $T_V(\mathbf{x})$ is isotropic only at two antipodal points, the poles $\mathbf{0}$ and $\mathbf{0}^{\prime}$.  Figure~\ref{fig2b} illustrates isotropy at $\mathbf{0}$: the covariance function $\text{Cov}(T_V(\mathbf{0}),T_V(\mathbf{x}))$ is axially symmetric around the vertical axis. The figure demonstrates that the plotted covariance function of the transformed field $T_V(\mathbf{x})$ takes higher values than the corresponding covariance function of the base field $T(\mathbf{x})$ in Figure~\ref{fig1b}. This is expected because $V(\mathbf{0})=1$, and, therefore, two covariance functions differ only by the multiplier $V(\mathbf{x})\ge 1$.
\end{example}

\begin{example} \label{e44}
Consider a deterministic multiplication function $V(\mathbf{x})=V(\theta,\varphi) = \theta(\pi - \theta)\varphi(2\pi - \varphi) + 1$ from Example~\ref{ex2} and the isotropic random field $T(\mathbf{x})$ from Example~\ref{ex1}. Figure~\ref{fig3a} shows a realization of $T_V(\mathbf{x}).$ In the central part, the values of $\varphi$ are close to 0, and the multiplication function $V(\mathbf{x})$ is close to the constant value 1, regardless of the values of $\theta.$ In areas far from the center, $\varphi$ increases and $\max_{x} V(\mathbf{x}) = \pi^4/4 + 1 \approx 25.352.$ Therefore, the pattern of "cold" and "hot" regions is the same as in Figure~\ref{fig1a}, but the color of the central part in Figure~\ref{fig3a} is "blurred" because it now has substantially smaller values, which results in smaller color gradient changes.

Since $V(\mathbf{x})>0$ and it is not constant on any circles, Case~{\rm(4)} of Theorem~\ref{th4} implies that $T_V(\mathbf{x})$ is strictly anisotropic. Figure~\ref{fig3b} visualizes this by plotting the covariance function $\text{Cov}(T_V(\mathbf{0}),T_V(\mathbf{x}))$, which does not display rotational symmetry.

\begin{figure}[!htb]
  \centering
  \begin{subfigure}{0.53\textwidth}
    \centering
    \includegraphics[width=\textwidth, trim={1 0 0 0}, clip]{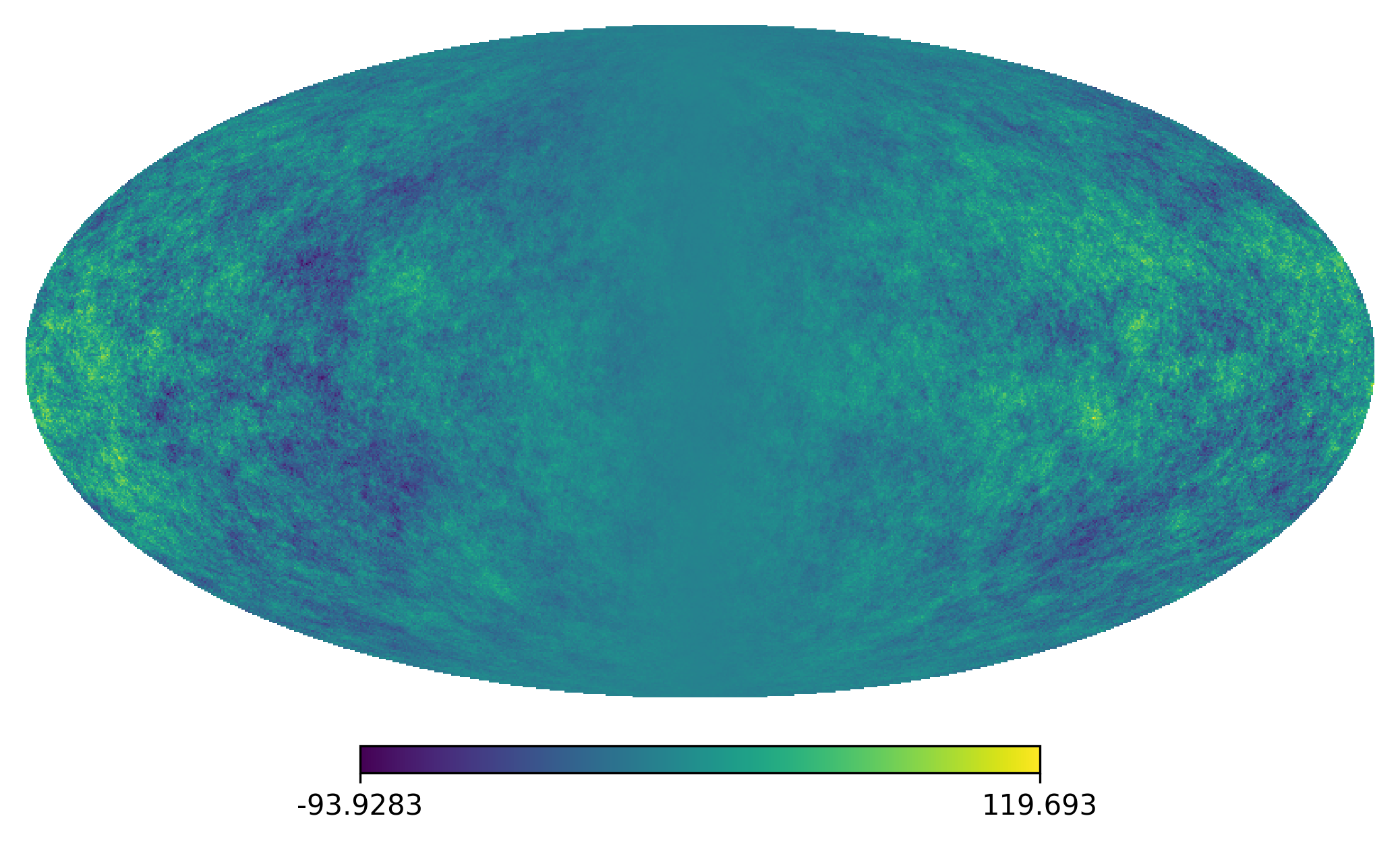}\\[0.2cm]
    \caption{\it Realization of the field}\label{fig3a}
  \end{subfigure}%
  \begin{subfigure}{0.61\textwidth}
    \centering
    \includegraphics[width=\textwidth, trim={4cm 0 0 5.5cm}, clip]{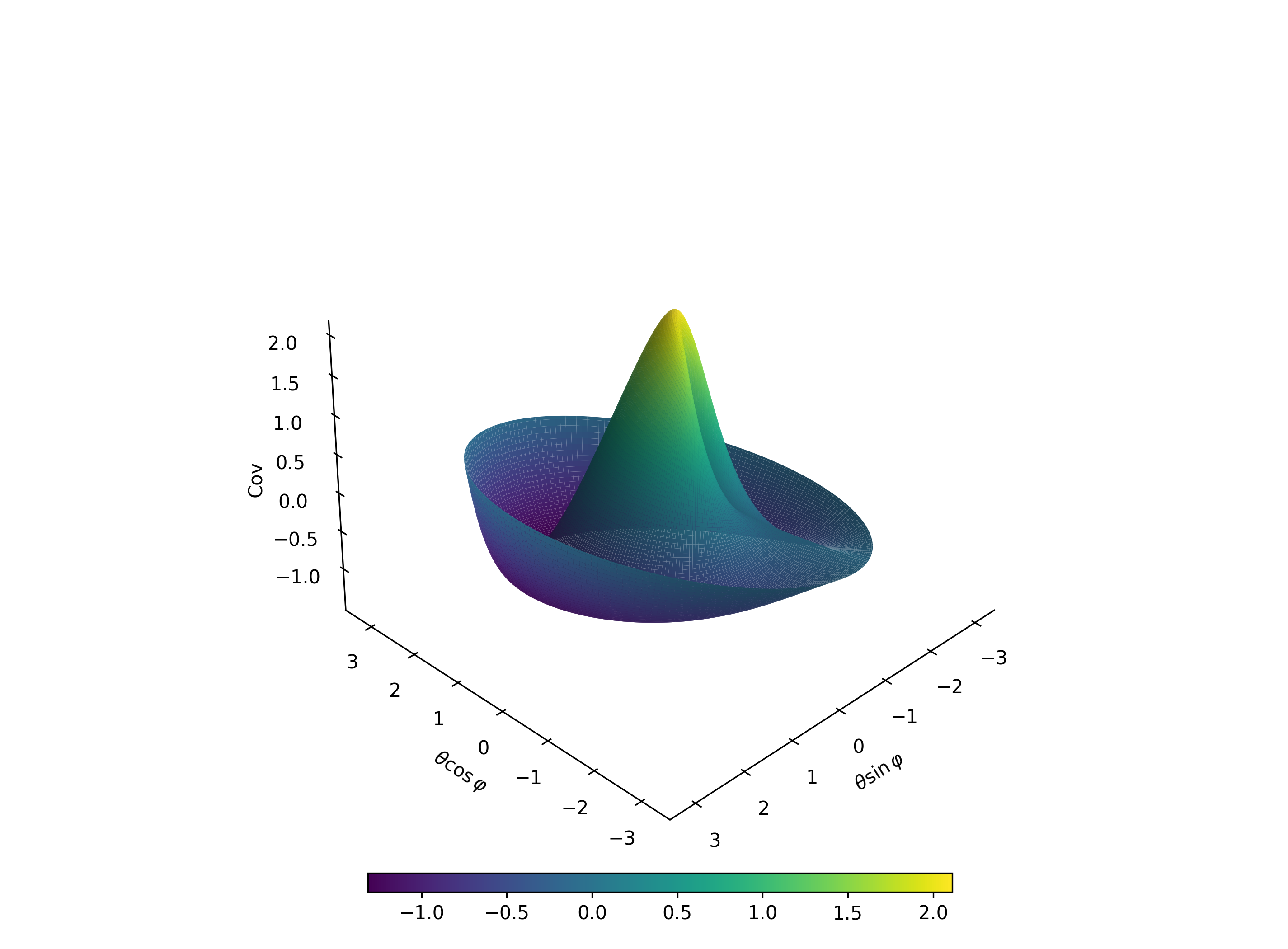}
        \captionsetup{justification=raggedright, singlelinecheck=false, margin=0cm}
    \caption{\it Covariance function  $\text{Cov}(T_V(\mathbf{0}),T_V(\mathbf{x}))$}
    \label{fig3b}
  \end{subfigure}
  \caption{\it Super anisotropic field $T_V(\mathbf{x}),$  $V(\theta,\varphi) = \theta(\pi - \theta)\varphi(2\pi - \varphi) + 1.$}
  \label{fig3}
\end{figure}
\end{example}

\subsection{Convolutional models}\label{sub:conv}
 Convolutional random fields are constructed by convolving an underlying (typically isotropic) field with a deterministic function, referred to as a convolution kernel. This operation is particularly useful in applications that require smoothing or localized modifications of a base field.

Formally, in the general setting, such fields are defined on the rotation group $SO(3),$ as presented in the following definition.

\begin{definition} \label{lmn}
Let $T(\mathbf{x}),$ $\mathbf{x} \in \mathbb{S}^2,$ be a random field, and let $ H(\mathbf{\rho}, \mathbf{x^{\prime}}),$ $\mathbf{\rho} \in SO(3),$ $\mathbf{x^{\prime}} \in \mathbb{S}^2, $ be a deterministic real-valued convolution kernel. Suppose that for all $\mathbf{\rho} \in SO(3)$, the kernel satisfies the integrability condition
\begin{equation*}
\int_{\mathbb{S}^2} \int_{\mathbb{S}^2}
\left| C(\mathbf{x}_1, \mathbf{x}_2) \, H(\mathbf{\rho}, \mathbf{x}_1) H(\mathbf{\rho}, \mathbf{x}_2) \right| \, d\mathbf{s}(\mathbf{x}_1) \, d\mathbf{s}(\mathbf{x}_2) < +\infty,
\end{equation*}
where $C(\mathbf{x}_1, \mathbf{x}_2)$ denotes the covariance function of $T(\cdot).$\\
The corresponding convolutional random field $T_{*H}(\mathbf{\rho})$ is then defined as
\begin{equation}\label{conv}
T_{*H}(\mathbf{\rho}) := (T*H)(\mathbf{\rho}) = \int_{\mathbb{S}^2} T(\mathbf{x^{\prime}}) \, H(\mathbf{\rho}, \mathbf{x^{\prime}}) \, d\mathbf{s}(\mathbf{x^{\prime}}), \quad \mathbf{\rho} \in SO(3).
\end{equation}
    \end{definition}

 The convolution kernel $ H(\mathbf{\rho}, \mathbf{x^{\prime}})$  may vary with the direction $\mathbf{\rho},$ thereby introducing anisotropy. The convolutional framework can be interpreted as a filtering of (isotropic) signals on spherical domains, resulting in a structured anisotropic representation introduced by $ H(\cdot).$  The literature contains various definitions of spherical convolutions. For comprehensive reviews, see \cite{rm21, srk12} and the references therein. Here, we focus on random fields convolved with directional kernels, considering isotropic, commutative anisotropic, and sifting convolutions.

 A spherical isotropic convolution employs an isotropic kernel, invariant under rotations around any axis passing through the center of the sphere. This ensures spherical symmetry, producing outputs that are rotationally invariant.

More specifically, there are two approaches to employ the construction in \eqref{conv} (defined on $SO(3)$) to obtain convolutional random fields on $\mathbb{S}^2.$

In the first approach, for a point $\mathbf{x} \in \mathbb{S}^2$ with spherical coordinates $(\theta, \varphi)$, one can use a convolution kernel $H$ that depends only on the location $\mathbf{x}$ on the sphere and not on the third Euler angle $\psi$ of the rotation $\rho \in SO(3)$:
\begin{equation*}
H(\rho, \mathbf{x^{\prime}}) = H((\varphi, \theta, \psi), \mathbf{x^{\prime}}) = H((\theta, \varphi), \mathbf{x^{\prime}}) = H(\mathbf{x}, \mathbf{x^{\prime}}).
\end{equation*}
The next two classes of fields are examples of this construction.

\begin{definition}\label{secondapp} Let $T(\mathbf{x}),$ $\mathbf{x} \in \mathbb{S}^2,$ be an isotropic random field with covariance function $B(\gamma(\mathbf{x}_1,\mathbf{x}_2))$, and let $ H(\mathbf{x}, \mathbf{x^{\prime}}),$ $\mathbf{x},\mathbf{x^{\prime}} \in \mathbb{S}^2,$ be a deterministic real-valued  convolution kernel. Assume that for all $\mathbf{x} \in \mathbb{S}^2$ the kernel satisfies the integrability condition
\begin{equation*}
\int_{\mathbb{S}^2} \int_{\mathbb{S}^2}
\left| B(\gamma(\mathbf{x}_1, \mathbf{x}_2))  H(\mathbf{x}, \mathbf{x}_1) H(\mathbf{x}, \mathbf{x}_2) \right|  d\mathbf{s}(\mathbf{x}_1)  d\mathbf{s}(\mathbf{x}_2) < +\infty.
\end{equation*}
Then, a convolutional random field associated with $T$ and $H$ is defined by
\begin{equation}\label{conv1}
T_{*H}(\mathbf{x}) := (T*H)(\mathbf{x}) =\int_{\mathbb{S}^2} T(\mathbf{x^{\prime}}) H(\mathbf{x},\mathbf{x^{\prime}}) d\mathbf{s}(\mathbf{x^{\prime}}).
\end{equation}
\end{definition}

The covariance function of the convolutional field $T_{*H}(\mathbf{x}),$ $\mathbf{x} \in \mathbb{S}^2,$ is given by
\begin{equation} \label{covs}
\text{Cov}\left(T_{*H}(\mathbf{x}_1), T_{*H}(\mathbf{x}_2)\right) =
\int_{\mathbb{S}^2} \int_{\mathbb{S}^2}
B\left(\gamma(\mathbf{x}_1^{\prime}, \mathbf{x}_2^{\prime})\right)  H(\mathbf{x}_1, \mathbf{x}_1^{\prime})  H(\mathbf{x}_2, \mathbf{x}_2^{\prime})  d\mathbf{s}(\mathbf{x}_1^{\prime})  d\mathbf{s}(\mathbf{x}_2^{\prime}).
\end{equation}

\begin{remark}
$V$-multiplicative fields can be viewed as a limiting case of convolutional fields, when the integral in {\rm (\ref{conv1})} is interpreted in the sense of Dirac generalized function distribution theory.

Specifically, if
$H(\mathbf{x}, \mathbf{x^{\prime}}) = \delta_{\mathbf{x}}(\mathbf{x^{\prime}})  V(\mathbf{x^{\prime}}),
$
with $\delta_{\mathbf{x}}(\mathbf{x^{\prime}})$ denoting the spherical Dirac delta function mapping $\mathbf{x^{\prime}}$ to $\mathbf{x}$, then the convolution in {\rm (\ref{conv1})} reduces to pointwise multiplication. Thus, the anisotropy of $T_{*H}$ arises entirely from the spatial structure of $V$, making $V$-multiplicative fields a simple yet illustrative example of how the kernel shape controls the anisotropic properties of the resulting field.
\end{remark}

    \begin{definition}
	Let a convolution kernel depend only on the great-circle distance
		between points on the sphere, i.e, $H(\mathbf{x},\mathbf{x^{\prime}})=H(\gamma(\mathbf{x},\mathbf{x^{\prime}})).$ The spherical isotropic convolution is then defined as
        \begin{equation}\label{eq:ic}
        T_{\odot H}(\mathbf{x}) := \left(T \odot H\right)(\mathbf{x}) = \int_{\mathbb{S}^2} T\left(\mathbf{x}^{\prime}\right) H(\gamma(\mathbf{x}, \mathbf{x}^{\prime})) \, d\mathbf{s}(\mathbf{x}^{\prime}).
        \end{equation}
	\end{definition}

By (\ref{covs}) and the isotropy of $T$, for any rotation $\rho \in SO(3)$, it follows that
\begin{align*}
&\text{Cov}\left(T_{\odot H}(\rho \mathbf{x}_1), T_{\odot H}(\rho \mathbf{x}_2)\right)\\
&\quad = \int_{\mathbb{S}^2} \int_{\mathbb{S}^2}
B\left(\gamma(\rho \mathbf{x}_1^{\prime}, \rho \mathbf{x}_2^{\prime})\right)
H(\gamma(\rho \mathbf{x}_1, \mathbf{x}_1^{\prime})) H(\gamma(\rho \mathbf{x}_2, \mathbf{x}_2^{\prime}))\,
d\mathbf{s}(\mathbf{x}_1^{\prime})\, d\mathbf{s}(\mathbf{x}_2^{\prime}) \notag \\
&\quad = \int_{\mathbb{S}^2} \!\! \int_{\mathbb{S}^2}
B\left(\gamma(\rho^{-1} \mathbf{x}_1^{\prime}, \rho^{-1} \mathbf{x}_2^{\prime})\right)
H(\gamma(\mathbf{x}_1, \rho^{-1} \mathbf{x}_1^{\prime})) H(\gamma(\mathbf{x}_2, \rho^{-1} \mathbf{x}_2^{\prime}))\,
d\mathbf{s}(\rho^{-1} \mathbf{x}_1^{\prime})\, d\mathbf{s}(\rho^{-1} \mathbf{x}_2^{\prime}) \notag \\
&\quad = \text{Cov}\left(T_{\odot H}(\mathbf{x}_1), T_{\odot H}(\mathbf{x}_2)\right),
\end{align*}
where the second equality follows from the rotation-invariance of both the isotropic covariance $B(\cdot)$ and the measure $\mathbf{s}(\cdot)$ on $\mathbb{S}^2$. Hence, the spherical isotropic convolution $T_{\odot H}(\mathbf{x})$ defines an isotropic field.

The second approach defines $T_{*H}^{\prime}(\theta, \varphi)$ by selecting a particular section of the field $T_{*H}(\mathbf{\rho})$ along a function $u(\theta,\varphi)$, i.e.,
\[
T_{*H}^{\prime}(\theta,\varphi) := \left. T_{*H}(\mathbf{\rho})\right|_{\psi = u(\theta,\varphi)} = T_{*H}(\theta,\varphi,u(\theta,\varphi)).
\]
This procedure allows defining specific classes of convolutional fields on $\mathbb{S}^2$ as restrictions of the full $SO(3)$ convolution.

	\begin{definition} Let $\mathbf{x} \in \mathbb{S}^2$ have spherical coordinates $(\theta,\varphi)$ and the rotation $\tilde{\rho} \in SO(3)$ be defined by its Euler angles $(\varphi,\theta,u(\theta,\varphi))$, with $u(\theta,\varphi) = \pi - \theta$. Assume that the convolution kernel satisfies $H(\tilde{\rho},\mathbf{x}^{\prime}) = H(\tilde{\rho}^{-1}\mathbf{x}^{\prime})$. Then, the commutative anisotropic convolution is defined by
        \begin{equation}\label{eq:cac}
        T_{\oplus H}(\mathbf{x}) = \left(T \oplus H \right)(\mathbf{x}) = T_{*H}(\theta, \varphi, \pi - \theta) = \int_{\mathbb{S}^2} T(\mathbf{x}^{\prime}) H(\tilde{\rho}^{-1} \mathbf{x}^{\prime}) \, d\mathbf{s}(\mathbf{x}^{\prime}).
        \end{equation}
	\end{definition}
 The choice of the third Euler angle, $\psi \equiv \pi-\theta,$ guarantees the commutativity of the convolution. This specific instance of a directional convolution was introduced in \cite{srk12}. By commutativity, the same transformation can equivalently be expressed as
    \begin{equation*}
 T_{\oplus H}(\mathbf{x}) = \int_{\mathbb{S}^2} T(\tilde{\rho}\mathbf{x}^{\prime}) H(\mathbf{x}^{\prime}) \, d\mathbf{s}(\mathbf{x}^{\prime}).
 \end{equation*}

	\begin{definition} For a real-valued kernel function $H(\mathbf{x}), \mathbf{x} \in \mathbb{S}^2,$ such that $H \in L^2(\mathbb{S}^2),$ the sifting convolution is defined by
		\begin{equation}\label{eq:sc}
			T_{\circledast H} (\mathbf{x})=	\left(T \circledast H\right)(\mathbf{x}) = \int_{\mathbb{S}^2} (\mathcal{L}_{\mathbf{x}} T)(\mathbf{x^\prime}) H(\mathbf{x^\prime}) d \mathbf{s} \, (\mathbf{x^\prime}),
		\end{equation}
		where $\mathcal{L}_{\mathbf{x}}$ is the spherical translation operator acting on a random field expressed via the Laplace series \eqref{spec} as
    \begin{equation}\label{oper}
    (\mathcal{L}_{\mathbf{x}} \, T)(\mathbf{x}^{\prime}) = \sum_{\ell \in \mathbb{N}_0} \sum_{m=-\ell}^{\ell} a_{\ell m} Y_{\ell m}(\mathbf{x}) Y_{\ell m}(\mathbf{x}^{\prime}), \quad \mathbf{x}^{\prime} \in \mathbb{S}^2.
    \end{equation}
\end{definition}
The operator $\mathcal{L}_{\mathbf{x}}$ can equivalently be interpreted as a sifting convolution of a function $f \in L^2\left(\mathbb{S}^2\right)$ with the spherical Dirac delta function, i.e.,
$$\left( \mathcal{L}_{\mathbf{x}} f\right)(\mathbf{x^\prime})=(f\circledast \delta_{\mathbf{x}})(\mathbf{x^\prime}),$$
see \cite{rm21} for details.

\subsection{Deformation models}
Finally, we consider the space deformation method for spherical isotropic random fields. The elliptic anisotropy discussed earlier serves as an example in~$\mathbb{R}^3.$ In Euclidean space, this approach has been studied in several publications and applied to the estimation of non-stationary correlation structures, see, for example, \cite{Perrin, Porcu} and the references therein.  The publication~\cite{Perrin} provided necessary and sufficient conditions under which non-stationary correlation functions can be reduced to stationary ones via the space deformation approach in~$\mathbb{R}^n$. For the case of spherical random fields, such a characterization is a much more difficult problem, but some progress using alternative approaches was obtained in~\cite{Porcu}. We also refer the reader to this work for additional references on available theoretical results and statistical applications.

Consider a continuous bijection (homeomorphism)   $U: \mathbb{S}^2 \to \mathbb{S}^2$ of the sphere onto itself. No smoothness assumptions are imposed unless explicitly stated. The formal definition of $U$-deformed fields is as follows.
    \begin{definition}\label{defor}
        Let $T(\mathbf{x})$, $\mathbf{x} \in \mathbb{S}^2$, be an isotropic second-order random field with zero mean. The space deformed random field $T^U(\cdot)$ is given by
        \[
        T^U(\mathbf{x}) := T(U(\mathbf{x})), \quad \mathbf{x} \in \mathbb{S}^2.
        \]
    \end{definition}
 Unlike multiplicative and convolutional constructions, deformation models can induce anisotropy by modifying the underlying geometry of the sphere rather than the values of the base field.

 \begin{remark}
$U$-deformed fields can be considered as a limiting case of convolutional fields {\rm (\ref{conv1})} when
$H(\mathbf{x},\mathbf{x}')=\delta_{U(\mathbf{x})}(\mathbf{x}')$,
where the integral in {\rm (\ref{conv1})} is interpreted in the sense of Dirac distribution theory.
 \end{remark}

    The following results provide characterizations of isotropy, point isotropy, and some related geometric properties of such space-deformed fields.
    \begin{theorem}\label{thU}
A field $T^U(\cdot)$ is isotropic for every isotropic random field $T(\cdot)$
if and only if $U(\cdot)$ is an isometry of the sphere.
    \end{theorem}
    \begin{remark}\label{isom}
        Note that the group  $O(3)$ of isometries of the sphere is generated by rotations and reflections in great circles, and, thus, every isometry can be written as their composition. This is crucial in the present context, since any deformation not belonging to this group necessarily distorts geodesic distances and thus
breaks isotropy.
    \end{remark}
The following result directly follows from Theorem~\ref{thU} and Remark~\ref{isom}, since it  considers a transformation that is not an isometry of the sphere.
\begin{corollary}\label{cor1}
Let $U(\cdot)$ be the composition of a non-trivial(non-identical) circular scaling transformation of the sphere at the point $\mathbf{x}_0 \in \mathbb{S}^2$, a rotation about the axis through $\mathbf{x}_0$ and its antipodal point, and a reflection of the sphere across the great circle through this axis. Then, the field $T^U(\cdot)$ is isotropic at the point $\mathbf{x}_0$, and there is an isotropic random field $T(\cdot)$ such that $T^U(\cdot)$ is anisotropic.
    \end{corollary}
The isotropy at the point $\mathbf{x}_0$ follows as the circular
scaling preserves all circles centered at $\mathbf{x}_0.$
The anisotropy statement follows because the transformation does not preserve at least one circle, and hence
breaks global isotropy for suitable isotropic covariance functions.

Theorem~{\rm\ref{thU}} is valid for any isotropic field $T(\cdot)$, but the last statement of Corollary~\ref{cor1} only shows the existence of a field $T(\cdot)$ with the given property.
However, under the non-degeneracy assumption in Condition~\ref{cond2}, the considered isotropic covariance functions are non-constant. Since any non-trivial circular scaling necessarily modifies circles away from the point $\mathbf{x}_0$, the following result generalizes Corollary~\ref{cor1} and establishes anisotropy for every isotropic random field.

\begin{corollary}\label{cor2}
Let $U(\cdot)$ be the composition of a non-trivial circular scaling transformation of the sphere at the point $\mathbf{x}_0 \in \mathbb{S}^2$, a rotation about the axis through $\mathbf{x}_0$ and its antipodal point, and a reflection of the sphere across the great circle through this axis. Then, for any choice of an isotropic random field $T(\cdot)$, the field $T^U(\cdot)$ is anisotropic, but remains isotropic at the point $\mathbf{x}_0$.
    \end{corollary}

\begin{example} \label{e444}
The deformed random field $T^U(\cdot)$ in Figure~\ref{fig4a} is obtained using a circular scaling transformation of the sphere at the pole $\mathbf{0} \in \mathbb{S}^2$ with the circular scaling function $\lambda(\theta) = \pi (\theta/\pi)^p$, $\theta \in [0,\pi]$. A value of $p=0.5$ was used. The realization in Figure~\ref{fig4a} shows that the field values are the same as in Figure~\ref{fig1a}, but they are shifted from the pole along the meridians. Therefore, the covariance function of $T^U(\cdot)$ in Figure~\ref{fig4b} takes higher values than the one shown in Figure~\ref{fig1b}.
\begin{figure}[!htb]
  \centering
  \begin{subfigure}{0.53\textwidth}
    \centering
    \includegraphics[width=\textwidth, trim={1 0 0 0}, clip]{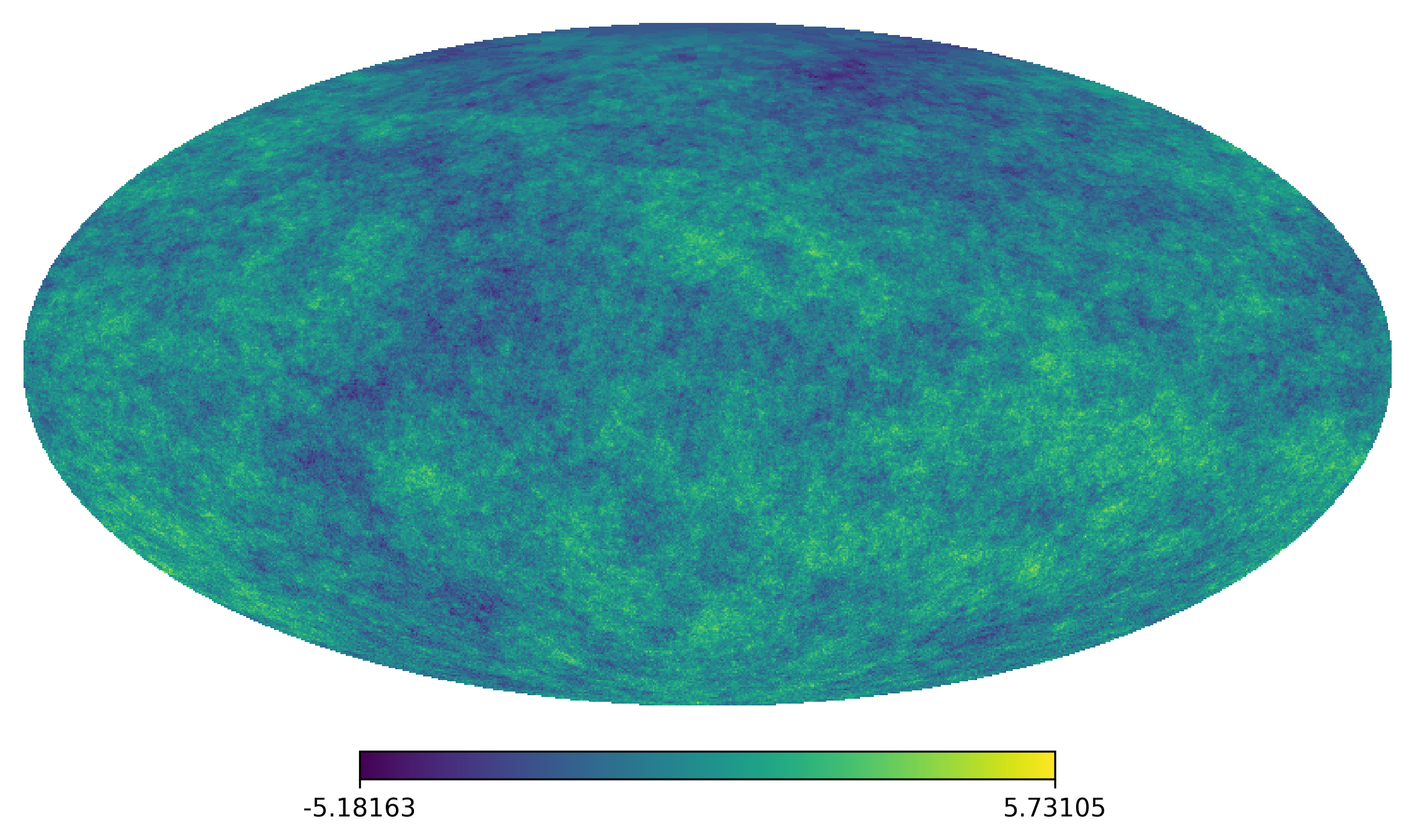}\\[0.2cm]
    \caption{\it Realization of the field}\label{fig4a}
  \end{subfigure}%
  \begin{subfigure}{0.61\textwidth}
    \centering
    \includegraphics[width=\textwidth, trim={4cm 0 0 5.5cm}, clip]{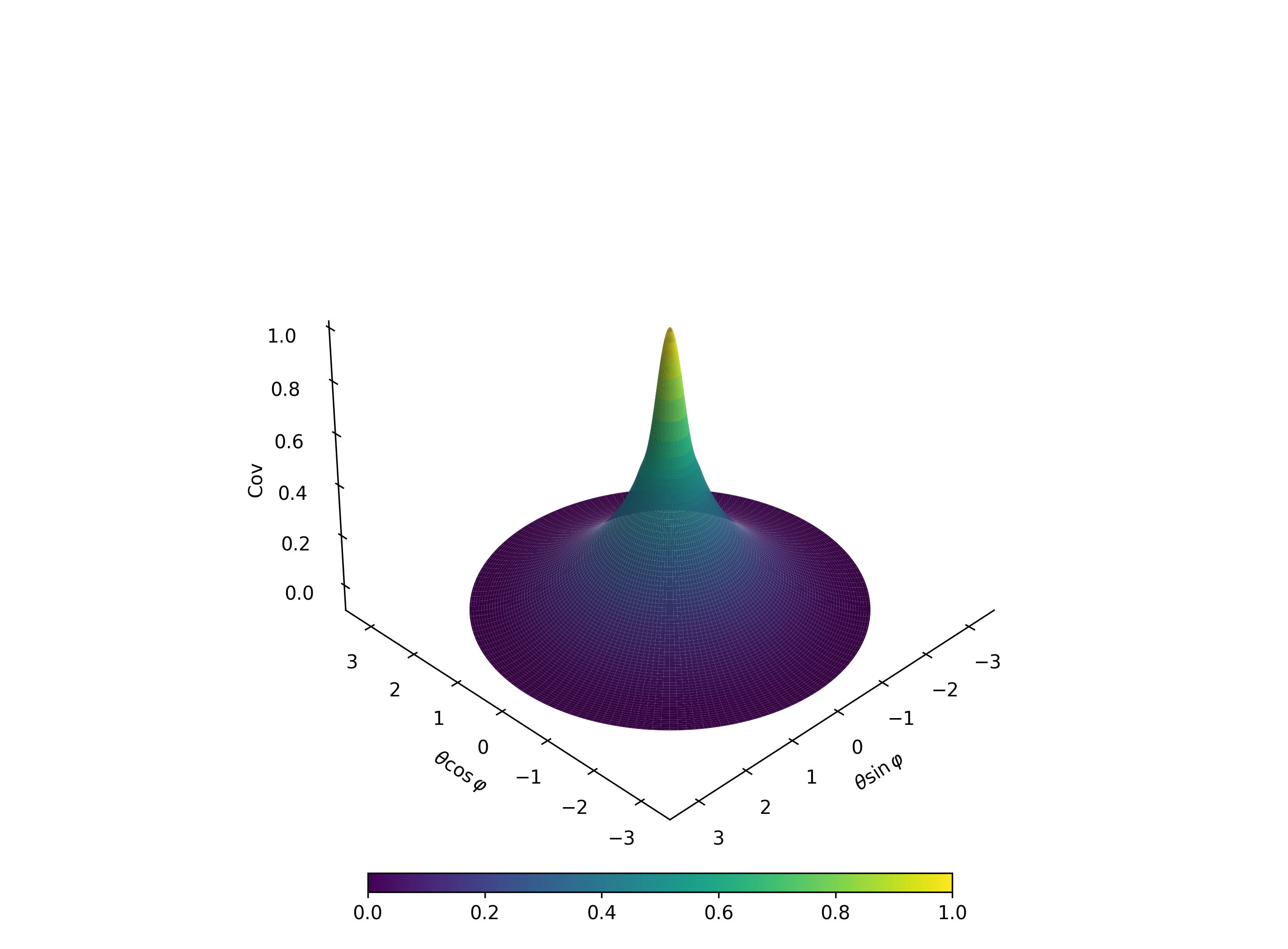}
        \captionsetup{justification=raggedright, singlelinecheck=false, margin=0cm}
    \caption{\it Covariance function  $\text{Cov}(T^U(\mathbf{0}),T^U(\mathbf{x}))$}
    \label{fig4b}
  \end{subfigure}
  \caption{\it Deformed field $T^U(\cdot)$
for $\lambda(\theta)=\pi\,(\theta/\pi)^{0.5}.$}
  \label{fig4}
\end{figure}

\end{example}

\section{Connections with the spectral theory}\label{sec_spec}
Spectral representations of random fields on the sphere provide a comprehensive analytical framework for analyzing their structure and properties. This section makes explicit the connection between the spectral properties of a spherical random field and its geometric isotropic or anisotropic behavior and the composite transformations introduced in Section~\ref{sec:composite}.

The spectral covariance matrix admits a block decomposition corresponding to irreducible representations of $SO(3)$, and anisotropy corresponds precisely to departures from irreducible block-diagonality.

For a general real-valued random field on $\mathbb{S}^2$ with spherical harmonic coefficients $\{a_{\ell m}\}$, the (possibly non-diagonal) angular power spectrum is defined as
\begin{equation*} \label{Cllmm} C_{\ell \ell^{\prime} m m^{\prime}}:=\mathbf{E}\left( a_{\ell m} a_{\ell^{\prime} m^{\prime}}^{\ast} \right),\quad \ell, \ell^{\prime} \in \mathbb{N}_0, \ m \in \{-\ell, \ldots, \ell\},\ m^{\prime} \in \{-\ell^{\prime}, \ldots, \ell^{\prime}\}.
\end{equation*}
The associated covariance function admits the expansion
\begin{equation}\label{cov}
\text{Cov}\left( T(\mathbf{x}_1), T (\mathbf{x}_2)\right)= \sum_{\ell \in \mathbb{N}_0} \sum_{ \ell^\prime \in \mathbb{N}_0} \sum_{m=-\ell}^{\ell}\sum_{m^\prime =-\ell^\prime}^{\ell^\prime} C_{\ell \ell^\prime m m^\prime} Y_{\ell m}(\mathbf{x}_1)Y^\ast_{\ell^\prime m^\prime}(\mathbf{x}_2).
\end{equation}
It follows from the spectral representation \eqref{spec} that a spherical random field is isotropic if and only if condition~\eqref{iso} holds, see \cite{Baldi}. Equivalently, a field is anisotropic if and only if its spectrum violates this condition. Thus, anisotropy occurs if and only if there exist indices $\ell,\ell^{\prime},m,m^{\prime}$ such that either
\begin{equation} \label{Ee}
C_{\ell \ell^{\prime} m m^{\prime}}  \neq 0, \quad \mbox{when} \quad |\ell - \ell^{\prime}| + |m - m^{\prime}| \neq 0,
\end{equation}
or
\begin{equation*}
C_{\ell \ell m m} \neq C_{\ell \ell m^{\prime} m^{\prime}},
\quad \text{for some } m\neq m^{\prime}.
\end{equation*}
The first condition reflects the presence of cross-correlations across different harmonic modes, while the second captures a non-uniform distribution of power across azimuthal indices at fixed angular frequency.

Finally, the spectral covariance matrix satisfies the Hermitian symmetry
\[
C_{\ell \ell^{\prime} m m^{\prime}} = C_{\ell^{\prime} \ell m^{\prime} m}^{\ast},
\quad \ell, \ell^{\prime} \in \mathbb{N}_0, \quad
m \in \{-\ell, \ldots, \ell\}, \quad m^{\prime} \in \{-\ell^{\prime}, \ldots, \ell^{\prime}\},
\]
which follows from the definition of covariance for complex-valued harmonic coefficients, and, since the field is real-valued,
\[
C_{\ell \ell^{\prime} m m^{\prime}} = (-1)^{m+m^{\prime}} C^*_{\ell\ell^{\prime} (-m) (-m^{\prime})}.
\]

We now classify different forms of spectral anisotropy according to the specific structure of the covariance matrix
$\{C_{\ell \ell^{\prime} m m^{\prime}}\}$, distinguishing whether anisotropy manifests across harmonic degrees, harmonic orders, or solely along the diagonal.

\begin{definition} \label{def20}
A random field is called harmonic degree-dependent anisotropic if condition~(\ref{Ee}) reduces to
\[
C_{\ell \ell^{\prime} m m^{\prime}}= \delta_{m}^{m^\prime} C_{\ell \ell^{\prime} m},
\]
with
\begin{equation*} \label{Cl}
C_{\ell \ell^{\prime} m}\neq 0 \quad \mbox{for some} \quad \ell \neq \ell^{\prime},\ \ell,\ell^{\prime}  \in \mathbb{N}_0, \  m\in\{-\min(\ell,\ell^\prime),\ldots, \min(\ell,\ell^\prime)\}.
\end{equation*}
\end{definition}
\begin{definition}\label{def21}
A random field is called harmonic order-dependent anisotropic if condition~(\ref{Ee}) takes the form
\[C_{\ell \ell^\prime m m^\prime} = \delta_{\ell}^{\ell^\prime} \, \tilde{C}_{\ell m m^\prime}
\]
with the coefficients satisfying
\begin{equation*} \label{clmm}
\tilde{C}_{\ell m m^\prime} \neq 0 \quad \text{for some } \ell \in \mathbb{N}_0,\ m \neq m^{\prime},\ m, m^\prime \in \{-\ell, \ldots, \ell\}.
\end{equation*}
\end{definition}

\begin{definition} \label{def22}
A random field is diagonal anisotropic if
\[
C_{\ell \ell^\prime m m^\prime} = \delta_{\ell}^{\ell^\prime} \delta_{m}^{m^\prime}\, {C}_{\ell m}
\]
with
\begin{equation*} \label{dan}
{C}_{\ell m} \neq {C}_{\ell m^{\prime}} \quad \text{for some } \ell \in \mathbb{N}_0, \ m \neq m^{\prime}, \ m, m^\prime \in \{-\ell, \ldots, \ell\}.
\end{equation*}
\end{definition}

The introduced conditions on $C_{\ell \ell^\prime m m^\prime}$ define different variations of the angular power spectrum and coupling modes. A second-order spherical random field belongs to one of the following disjoint spectral classes in (\ref{spect}), where the fully anisotropic class of fields has an angular power spectrum $C_{\ell \ell^\prime m m^\prime}$ that depends non-degenerately on all four indices $\ell, \ell^\prime, m,$ and $m^\prime.$
\begin{align}
\textsc{Second Order}&
\,= \,\textsc{Isotropic}
\,\cup\, \textsc{Degree-Dependent Anis.} \nonumber \\
&\hspace{-1cm}
\cup \,\textsc{Order-Dependent Anis.}
\, \cup\, \textsc{Diagonal Anis.}
\,\cup \,\textsc{Fully Anis.}\label{spect}
\end{align}
The subsequent results establish the relationships between these spectral classes, the anisotropic classes in~(\ref{isotr}), and the classes of composite random fields from Section~\ref{sec:composite}.

The early work~\cite{Jones}, and subsequently \cite{Alegría, Hitczenko} and~\cite{Buhmann}, investigated spectral representations of random fields that are axially symmetric with respect to a fixed axis passing through the pole $\mathbf{0}$. In this framework, it was shown that the condition
$C_{\ell \ell^{\prime} m m^{\prime}} = \delta_{m}^{m^{\prime}} \, C_{\ell \ell^{\prime} m}$
is necessary and sufficient for axial symmetry. Consequently, the following result holds.

\begin{proposition} \label{prop5}
The class of random fields that are axially symmetric with respect to the axis passing through the pole $\mathbf{0}$ is the union of the following mutually disjoint subclasses: isotropic fields, diagonal anisotropic fields, and harmonic degree-dependent anisotropic fields.
For all fields in this class, the poles corresponding to $\theta=0$ and $\theta=\pi$ are points of isotropy.

A random field is axially symmetric of this type if and only if its spectral coefficients satisfy the condition
\[C_{\ell \ell^{\prime} m m^{\prime}}= \delta^{m^{\prime}}_{m} C_{\ell \ell^{\prime} m}.
\]
Equivalently, its covariance function $C(\mathbf{x}_1, \mathbf{x}_2)$ admits the spectral representation
\begin{align}
C((\theta_1,\varphi_1),(\theta_2,\varphi_2))& = \sum_{\ell \in \mathbb{N}_0} \sum_{\ell^\prime \in \mathbb{N}_0} \sum_{m=-\min(\ell,\ell^{\prime})}^{\min(\ell,\ell^{\prime})}
    \left(\frac{(2\ell+1)(\ell-m)!(2\ell^{\prime}+1)(\ell^{\prime}-m)!}{16\pi^2(\ell+m)!(\ell^{\prime}+m)!}\right)^{1/2}\nonumber\\
 &\quad \times C_{\ell\ell^\prime m}\, e^{im(\varphi_1-\varphi_2)}  P_{\ell}^{m}(\cos(\theta_1))P_{\ell^{\prime}}^{m}(\cos(\theta_2)),\label{cov_ax}
\end{align}
where $\mathbf{x}_1=(\theta_1,\varphi_1)$ and $\mathbf{x}_2=(\theta_2,\varphi_2).$

Moreover, if the covariance function is differentiable, then it belongs to this axially symmetric class if and only if, for all $\theta_i,\varphi_i$, $i=1,2$, it satisfies
\begin{equation}\label{dif1}
\frac{\partial}{\partial \varphi_1} C((\theta_1,\varphi_1),(\theta_2,\varphi_2))=-\frac{\partial}{\partial \varphi_2} C((\theta_1,\varphi_1),(\theta_2,\varphi_2)).\end{equation}
\end{proposition}

Thus, the class of harmonic degree-dependent anisotropic fields consists of all axially symmetric fields around the axis through the pole $\mathbf{0}$ that are neither isotropic nor diagonal anisotropic. In particular, their covariance functions necessarily admit the representation~(\ref{axiall}). Additionally, by~\cite[Lemma 5]{Buhmann}, the condition $
C_{\ell \ell^{\prime} (- m)}= C_{\ell \ell^{\prime} m}$ is necessary and sufficient for longitudinal reversibility. Consequently, under this symmetry condition, the covariance function reduces to the form~(\ref{reversib}).

\begin{proposition} \label{prop5_1}
A spherical random field belongs to a disjoint union of the classes of isotropic, diagonal anisotropic, and harmonic order-dependent anisotropic fields if and only if its spectrum satisfies
\[
C_{\ell \ell^{\prime} m m^{\prime}}= \delta^{\ell^{\prime}}_{\ell} C_{\ell m m^{\prime}}.
\]
In this case, the corresponding covariance function $C(\mathbf{x}_1, \mathbf{x}_2)$ admits the spectral representation
\begin{align} C((\theta_1,\varphi_1),(\theta_2,\varphi_2))& = \sum_{\ell \in \mathbb{N}_0}  \sum_{m=-\ell}^{\ell} \sum_{m^{\prime}=-\ell}^{\ell}
   \frac{2\ell+1}{4\pi} \left(\frac{(\ell-m)!(\ell-m^{\prime})!}{(\ell+m)!(\ell+m^{\prime})!}\right)^{1/2}\nonumber\\
 &\quad \times \tilde{C}_{\ell mm^\prime}\, e^{i(m\varphi_1-m^{\prime}\varphi_2)}  P_{\ell}^{m}(\cos\theta_1)P_{\ell}^{m^{\prime}}(\cos\theta_2),\nonumber
\end{align}
where $\mathbf{x}_1=(\theta_1,\varphi_1)$ and $\mathbf{x}_2=(\theta_2,\varphi_2).$
If the covariance function is twice differentiable, then this class can be equivalently characterized by the condition
\begin{equation}\label{dif2}
\Delta
_{\mathbf{x}_1} C(\mathbf{x}_1, \mathbf{x}_2)=\Delta
_{\mathbf{x}_2}  C(\mathbf{x}_1, \mathbf{x}_2),
\end{equation}
holding for all $\mathbf{x}_1$ and $\mathbf{x}_2$, where $\Delta_{\mathbf{x}}$ is the Laplace–Beltrami operator on the sphere, given in spherical coordinates by
\[
\Delta _{(\theta ,\varphi )}=\frac{1}{\sin \theta } \; \frac{\partial }{
	\partial \theta }\left( \sin {\theta }\;\frac{\partial }{\partial \theta }
\right) +\frac{1}{\sin ^{2}{\theta }}\;\frac{\partial ^{2}}{\partial
	\varphi^{2}}.
\]
\end{proposition}

The next result follows immediately from Propositions~\ref{prop5} and~\ref{prop5_1}, showing that their simultaneous validity can be used to determine the diagonal anisotropic class via differential constraints.

\begin{corollary}
If an anisotropic spherical random field has a twice differentiable covariance function, conditions~(\ref{dif1}) and~(\ref{dif2}) are necessary and sufficient for it to be diagonally anisotropic.
\end{corollary}

The following proposition characterizes pointwise isotropy at the pole directly in terms of the spectral coefficients of the field.

\begin{proposition} \label{pisot}
The pole $\mathbf{0}$ is a point of isotropy of a spherical random field if and only if, for all $\ell^{\prime} \in \mathbb{N}$ and $m^{\prime} \ne 0,$ the field’s spectrum satisfies
\[
\sum_{\ell \in \mathbb{N}_0}\sqrt{2\ell+1}\, C_{\ell \ell^{\prime} 0 m^{\prime}}=0.
\]
\end{proposition}

The proposition highlights that pointwise isotropy is strictly weaker than axial symmetry or (global) isotropy and can be detected via linear constraints on the cross-degree spectral coefficients.

It follows directly from Definition~\ref{def21} that
\begin{corollary}\label{pisot1}
The pole $\mathbf{0}$ is a point of isotropy for a harmonic order-dependent anisotropic field if and only if
\[
\tilde{C}_{\ell^{\prime} 0 m^{\prime}}=0
\quad \text{for all } \ell^{\prime} \in \mathbb{N},\; m^{\prime} \ne 0 .
\]
\end{corollary}

\begin{remark}\label{rotation}
The results for an arbitrary point $\mathbf{x}_0 \in \mathbb{S}^2$ follow immediately by rotating the sphere so that $\mathbf{x}_0$ is mapped to the pole $\mathbf{0}$.
Let $\rho$ denote such a rotation, and let $C_{\ell \ell^{\prime}} = \{C_{\ell \ell^{\prime} m m^{\prime}}\}$ be the array of spectral coefficients of the field.
Then, by (\ref{AA}), the previous conditions can be applied to the rotated field by considering the transformed array of coefficients
\[
\tilde{C}_{\ell \ell^{\prime}} = \left\{ D^{\ell}(\rho)\, C_{\ell \ell^{\prime}}  \left(D^{\ell^{\prime}}(\rho)\right)^* \right\},
\]
where $D^{\ell}(\rho)$ are the Wigner $D$-matrices, see Appendix~\ref{appC}.
\end{remark}

The following subsections further detail the structure of the covariance functions and corresponding spectra of the composite fields introduced in Section \ref{sec:composite}. To construct these fields, a real-valued isotropic field $T(\mathbf{x})$ is considered, with the spectral representation~(\ref{spec}) and covariance function given by~(\ref{eq1}).

	\subsection{Spectral analysis of multiplicative fields}
Let $V \in L^2\left(\mathbb{S}^2, \sin \theta d\theta  d \varphi\right)$ be a deterministic function with the Laplace series expansion
	\begin{equation}\label{eqzz}
		V(\mathbf{x}) = \sum_{\ell \in \mathbb{N}_0} \sum _{m=-\ell}^{\ell} v_{\ell m}Y_{\ell m}(\mathbf{x}), \quad \mathbf{x} \in \mathbb{S}^2.
	\end{equation}
The coefficients $\{v_{\ell m}\}$ play the role of harmonic coefficients of $V(\cdot)$. For a real-valued function $V(\cdot)$, they satisfy the conjugate symmetry
$v_{\ell,-m} = (-1)^m v_{\ell m}^{\ast}.
$
\begin{remark}
If $V(\cdot)$ is constant over the sphere, standard arguments imply that the only non-zero coefficient is $v_{00}$. In this case, by Theorem~\ref{th4}, the multiplicative field reduces to
\[
T_V(\mathbf{x}) = \frac{v_{00}}{\sqrt{4\pi}}  T(\mathbf{x}),
\]
and its spectral representation and covariance function follow directly from those of the random field $T(\cdot)$.
\end{remark}

The following result provides the Laplace expansion of a multiplicative field $T_V$ and explicit expressions for its spectral coefficients in terms of those of the underlying isotropic field $T(\cdot)$ and the multiplicative function $V(\cdot)$.

\begin{theorem} \label{th5}
Let $V(\mathbf{x})$, $\mathbf{x} \in \mathbb{S}^2$, have the Laplace expansion (\ref{eqzz}). Then, the multiplicative field $T_V$ admits the spectral expansion
\[
T_V(\mathbf{x}) = \sum_{\ell \in \mathbb{N}_0} \sum_{m=-\ell}^{\ell} a_{\ell m,V}  Y_{\ell m} (\mathbf{x}), \quad \mathbf{x} \in \mathbb{S}^2,
\]
where, for any $\ell \in \mathbb{N}_0$ and $m \in \{-\ell,\ldots, \ell\}$, the harmonic coefficients $a_{\ell m,V}$ are centered random variables given by
\begin{align}\label{eq:mcoeff0}
a_{\ell m,V} &:= \int_{\mathbb{S}^2} T_V(\mathbf{x}) \, Y^{\ast}_{\ell m}(\mathbf{x}) \, d \mathbf{s}(\mathbf{x}) \\
&= (-1)^m \sum_{\ell_2 \in \mathbb{N}_0} \sum_{\ell_1 = |\ell_2 - \ell|}^{\ell_2 + \ell} \sum_{m_2=-\ell_2}^{\ell_2} v_{\ell_1\, m-m_2} \, a_{\ell_2 m_2} \nonumber \\
&\quad \times \sqrt{\frac{(2\ell_1+1)(2\ell_2+1)(2\ell+1)}{4\pi}}
\begin{pmatrix}
\ell_1 & \ell_2 & \ell \\
0 & 0 & 0
\end{pmatrix}
\begin{pmatrix}
\ell_1 & \ell_2 & \ell \\
m-m_2 & m_2 & -m
\end{pmatrix}, \label{eq:mcoeff}
\end{align}
where $\begin{pmatrix} \ell & \ell^{\prime} & L \\ m & m^{\prime} & M \end{pmatrix}$ denotes the Wigner $3j$-symbol, see Appendix~\ref{appB}.

The multiplicative field $T_V(\cdot)$ is centered, and its covariance function given by (\ref{cov}) can be expressed in terms of the spectral coefficients defined as
\begin{align}
C_{\ell \ell^{\prime} m m^{\prime}}
&= (-1)^{m+m'}\sum_{\ell_2 \in \mathbb{N}_0}
\sum_{\ell_1 = |\ell_2 - \ell|}^{\ell_2 + \ell}
\sum_{\ell_1^{\prime} = |\ell_2 - \ell^{\prime}|}^{\ell_2 + \ell^{\prime}}
\sum_{m_2 = -\ell_2}^{\ell_2}
v_{\ell_1\, m - m_2} \, v^*_{\ell_1^{\prime}\, m^{\prime} - m_2} \, C_{\ell_2} \, \frac{2\ell_2 + 1}{4\pi} \nonumber \\
&\quad \times \sqrt{(2\ell_1 + 1)(2\ell_1^{\prime} + 1)(2\ell + 1)(2\ell^{\prime} + 1)} \nonumber \\
&\quad \times
\begin{pmatrix} \ell_1 & \ell_2 & \ell \\ 0 & 0 & 0 \end{pmatrix}
\begin{pmatrix} \ell_1^{\prime} & \ell_2 & \ell^{\prime} \\ 0 & 0 & 0 \end{pmatrix}
\begin{pmatrix} \ell_1 & \ell_2 & \ell \\ m-m_2 & m_2 & - m \end{pmatrix}
\begin{pmatrix} \ell_1^{\prime} & \ell_2 & \ell^{\prime} \\ m^{\prime} - m_2 & m_2 & - m^{\prime} \end{pmatrix}, \label{spectral_mat}
\end{align}
for $\ell, \ell^{\prime} \in \mathbb{N}_0$ and $m, m^{\prime} \in \{-\ell,\ldots,\ell \}$.
\end{theorem}

\begin{remark}
Due to the selection rules of the Wigner $3j$-symbols,  see Appendix~\ref{appB}, many of the sums in formulas (\ref{eq:mcoeff}) and (\ref{spectral_mat}) are sparse. Only a subset of indices $\ell_1, \ell_2, m_2$ contribute non-zero terms. In particular, the triangle conditions and the requirement that the sum of the lower row indices equals zero significantly reduce the number of non-zero summands.
\end{remark}

Consider Case (3) of Theorem~\ref{th4} where $V(\mathbf{x})$ is constant on each circle $O(\mathbf{x_0}, {r}),$ with ${r} \in ({0}, \pi).$ Without loss of generality, choose a coordinate system such that $\mathbf{x_0}=\mathbf{0}$ is the pole, see Remark~\ref{rotation}. Due to the linear independence of spherical harmonics and the rotational symmetry of $V$ about the pole, only the modes with $m = 0$ contribute to the expansion. These modes are indeed invariant under azimuthal rotations, yielding the simplified harmonic expansion
		\begin{equation} \label{VVx}
			V(\mathbf{x}) = \sum_{\ell \in \mathbb{N}_0} v_{\ell0}Y_{\ell0}(\mathbf{x}), \quad \mathbf{x} \in \mathbb{S}^2.
		\end{equation}

The following result establishes that, in this setting, the multiplicative field $T_{V}(\mathbf{x}),$ $\mathbf{x} \in \mathbb{S}^2,$ is harmonic degree-dependent anisotropic. It also provides explicit expressions for its harmonic coefficients, power spectrum, and covariance structure in terms of the Laplace coefficients $v_{\ell 0}, \ \ell \in \mathbb{N}_0.$
\begin{theorem} \label{Ta}
Let $V(\mathbf{x}), \mathbf{x} \in \mathbb{S}^2,$ satisfy the zonal expansion (\ref{VVx}).  Then, for any $\ell \in \mathbb{N}_0,$ and $m \in \{-\ell,\ldots, \ell\}$, the harmonic coefficients of the multiplicative field $T_V(\cdot)$ defined by the integral \eqref{eq:mcoeff0} are equal to
		\begin{equation*}
			a_{\ell m,V}= (-1)^m  \sum_{\ell_2 \in \mathbb{N}_0}
    \sum_{\ell_1 = |\ell_2 - \ell|}^{\ell_2 + \ell}  v_{\ell_1 0} a_{\ell_2 m}\sqrt{\frac{(2\ell_1+1)(2\ell_2+1)(2\ell+1)}{4\pi}} \begin{pmatrix}
				\ell_1&\ell_2& \ell\\
				0& 0& 0
			\end{pmatrix}
			\begin{pmatrix}
				\ell_1&\ell_2& \ell\\
				0& m& -m
			\end{pmatrix}.
		\end{equation*}

Therefore, the associated power spectrum of the field $T_V(\cdot)$ has the harmonic degree-dependent form
\[	 C_{\ell \ell^\prime m m^\prime}	=\delta_{m}^{m^\prime}C_{\ell \ell^\prime m},
\]
with
		\begin{align}
			C_{\ell \ell^\prime m}	= &   \sum_{\ell_2 \in \mathbb{N}_0}
    \sum_{\ell_1 = |\ell_2 - \ell|}^{\ell_2 + \ell}
    \sum_{\ell_1^{\prime} = |\ell_2 - \ell^{\prime}|}^{\ell_2 + \ell^{\prime}}   v_{\ell_1 0}v_{\ell_1^\prime 0} \frac{2\ell_2+1}{4\pi}C_{\ell_2}
			\notag	\sqrt{(2\ell_1+1)(2\ell_1^\prime+1)(2\ell+1)(2\ell^{\prime}+1)}\\
			& \quad\times 	\begin{pmatrix}
				\ell_1&\ell_2& \ell\\
				0& 0& 0
			\end{pmatrix}
			\begin{pmatrix}
				\ell^\prime_1&\ell_2& \ell^\prime\\
				0& 0& 0
			\end{pmatrix}
			\begin{pmatrix}
				\ell_1&\ell_2& \ell\\
				0& m& -m
			\end{pmatrix}
			\begin{pmatrix}
				\ell^\prime_1&\ell_2& \ell^\prime\\
				0& m& -m
			\end{pmatrix}\label{eq:cllm}.
		\end{align}
		\end{theorem}

A direct implication of Theorem~\ref{Ta} concerns the covariance structure at the pole, showing the pointwise isotropy of the multiplicative field there.

\begin{corollary} \label{C2}
    Let the conditions of Theorem \ref{Ta} be satisfied. Then, the covariance between the field at the pole $\mathbf{0}$ and an arbitrary point $\mathbf{x} \in \mathbb{S}^2$ depends only on the angular separation $\gamma(\mathbf{0}, \mathbf{x})$ and is given by
    		\begin{align*}
			\text{Cov}\left(T_V (\mathbf{0}), T_V (\mathbf{x})\right)
			& = V(\mathbf{0})B(\gamma(\mathbf{0},\mathbf{x}))V{(\mathbf{x})}= \left( \sum_{\ell_1^\prime\in \mathbb{N}_0} \sqrt{2\ell_1^\prime +1 } v_{\ell_1^\prime 0}   \right)\sum_{\ell_1\in \mathbb{N}_0} \sqrt{2\ell_1 +1 } v_{\ell_1 0}  \\
			& \times\sum_{\ell_2\in \mathbb{N}_0} \frac{2\ell_2+1}{4\pi}C_{\ell_2}  \sum_{\ell = |\ell_2 - \ell_1|}^{\ell_2 + \ell_1} \frac{2\ell+1}{4\pi}	\begin{pmatrix}
				\ell_1 & \ell_2 & \ell\\
				0 & 0 & 0%
			\end{pmatrix}^2 P_\ell(\cos(\gamma(\mathbf{0},\mathbf{x}))).
		\end{align*}
        \end{corollary}

\subsection{Spectral analysis of convolutional fields}
This section provides the spectral representation of the three classes of convolutional fields introduced in Section~\ref{sub:conv}.

Consider the isotropic convolution model \eqref{eq:ic}, where the convolution kernel $H(\mathbf{x},\mathbf{x}^{\prime})$ depends only on the great-circle distance $\gamma(\mathbf{x},\mathbf{x}^{\prime})$.
This radial dependence ensures that the convolution preserves isotropy of the underlying field $T(\mathbf{x})$.
Its harmonic expansion can be written as
\[
H(\mathbf{x}, \mathbf{x}')
= H (\gamma(\mathbf{x}, \mathbf{x}'))= \sum_{\ell \in \mathbb{N}_0} h_{\ell} \,
\sum_{m=-\ell}^{\ell} Y_{\ell m}(\mathbf{x}) \, Y_{\ell m}^{*}(\mathbf{x}')\]
\begin{equation} \label{legen}
= \sum_{\ell \in \mathbb{N}_0} \frac{2\ell+1}{4\pi} \, h_{\ell} \,
P_\ell\left(\cos(\gamma(\mathbf{x},\mathbf{x}'))\right)
,
\end{equation}
where $\{h_{\ell}\}$ is an analog of the set of the angular power spectrum coefficients for the convolution kernel.

\begin{theorem} \label{thtt}
Let $H(\cdot,\cdot)$ have the representation (\ref{legen}).
Then, the convolutional random field $T_{\odot H}(\mathbf{x}), \, \mathbf{x} \in \mathbb{S}^{2}$, is zero-mean and isotropic with the spectral representation
\begin{equation} \label{tdot}
T_{\odot H}(\mathbf{x})
= \sum_{\ell \in \mathbb{N}_0} \sum_{m=-\ell}^{\ell}
a_{\ell m, \odot}\, Y_{\ell m}(\mathbf{x}),
\end{equation}
where
$a_{\ell m, \odot} = a_{\ell m}\, h_{\ell}.$

Its harmonic coefficients satisfy
\[
\mathbf{E}[a_{\ell m,\odot}] = 0, \quad
\mathbf{E}[a_{\ell m,\odot} a_{\ell' m',\odot}^*] = C_{\ell,\odot} \delta_\ell^{\ell'} \delta_m^{m'},
\]
where $C_{\ell,\odot} = C_\ell h_\ell^2$ is the angular power spectrum of the convolved field.
The corresponding covariance function is
\[
\text{Cov}(T_{\odot H}(\mathbf{x}_1),T_{\odot H}(\mathbf{x}_2)) = \sum_{\ell \in \mathbb{N}_0} \frac{2\ell+1}{4\pi} C_{\ell,\odot} P_\ell(\cos(\gamma(\mathbf{x}_1,\mathbf{x}_2))).
\]
\end{theorem}
Therefore, in the isotropic case, convolution acts diagonally on harmonic coefficients and preserves isotropy through a simple modulation of the angular power spectrum.

Consider now the commutative anisotropic convolution model given by \eqref{eq:cac}. Let the convolution kernel $H(\cdot)\in L_2(\mathbb{S}^2)$ in (\ref{eq:cac}) have the following Laplace series representation
\begin{equation} \label{Hxx}
H(\mathbf{x}) = \sum_{\ell \in \mathbb{N}_0} \sum_{m=-\ell}^{\ell} h_{\ell m}\, Y_{\ell m}(\mathbf{x}),
\quad \mathbf{x} \in \mathbb{S}^2,
\end{equation}
with the $L_2$ norm
\begin{equation*}\label{L2H}
\|H\|_{L^2(\mathbb{S}^2)}^2
= \sum_{\ell \in \mathbb{N}_0} \sum_{m=-\ell}^{\ell} | h_{\ell m}|^2.
\end{equation*}
To obtain a real-valued convolution, real-valued kernels are used. Hence, their harmonic coefficients satisfy the relation $h_{ \ell (-m)}=(-1)^{m}h^{\ast}_{\ell m}.$
To state the following results, we use the notations of Wigner azimuthal $d$ matrices, see Appendix~\ref{appC}.
Unlike isotropic convolution, commutative anisotropic convolution allows the kernel to have mode-dependent azimuthal weights $h_{\ell m}$, leading to coupling across $m$ indices.

The following results give the spectral and covariance structures of the convolutional random fields.

	\begin{theorem} \label{th8}
Let $T_{\oplus H} (\mathbf{x}), \mathbf{x} \in \mathbb{S}^2,$ be defined by (\ref{eq:cac}) with a real-valued convolution kernel $H(\cdot)$ given by (\ref{Hxx}). Then, it is a zero-mean field that has the spectral representation
	\begin{equation*}
		T_{\oplus H}(\mathbf{x})
		=\sum_{\ell\in \mathbb{N}_0} \sum_{m=-\ell}^{\ell} \sum_{M=-\ell}^{\ell}a_{\ell  m} h^{\ast}_{\ell M} (-1)^me^{i(M-m)\varphi}d^{\ell}_{Mm}(\theta),
	\end{equation*}
   where $d^{\ell}_{Mm}(\theta)$ is defined by (\ref{dmml}) in Appendix~\ref{appC}.

Its covariance function is given by
		\begin{align*}	\text{Cov}\left(T_{\oplus H} (\mathbf{x}_1),T_{\oplus H}(\mathbf{x}_2)\right) & =
		  \sum_{\ell\in \mathbb{N}_0}
            \sum_{m=-\ell}^{\ell}
          \sum_{m^\prime=-\ell}^{\ell}  C_{\ell } h^\ast_{\ell m}h_{\ell  m^\prime} e^{i\left(m\varphi_{1}-m^\prime\varphi_{2}\right)}
          e^{-im\Phi} e^{-i m^{\prime} \Psi} d^{\ell}_{m m^\prime} (\Theta),
		\end{align*}
		where $\mathbf{x},$ $\mathbf{x}_1,$ and $\mathbf{x}_2$ have spherical coordinates $(\theta, \varphi), (\theta_1,\varphi_1)$ and $(\theta_2, \varphi_2)$ respectively. The Euler angles $(\Phi,\Theta,\Psi)$ are obtained by composing the rotations corresponding to $\mathbf{x}_1$ and $\mathbf{x}_2$, with $\psi_1=\psi_2=0,$ see \cite[Eq. 6, Sec. 4.7.2]{vmk}.
	\end{theorem}
\begin{remark} The commutative anisotropic convolution retains linearity but induces coupling across azimuthal indices, yielding an anisotropic covariance structure naturally described in terms of Wigner $d$-matrices.  If $h_{\ell m} = h_\ell \delta_{m0}$, the commutative anisotropic convolution reduces to the isotropic convolution described in Theorem~\ref{thtt}.
\end{remark}

	Finally, we consider the sifting convolution model introduced by \eqref{eq:sc}.
	The definition of the translation $\mathcal{L}_{\mathbf{x}}$ for spherical harmonics is, see \cite{rm21},
	\begin{equation*}
		\left(\mathcal{L}_{\mathbf{x}} Y_{\ell m}\right)\left(\mathbf{x^\prime}\right) = Y_{\ell m}(\mathbf{x})Y_{\ell m}(\mathbf{x^\prime}),\quad \mathbf{x^\prime}\in \mathbb{S}^2.
	\end{equation*}

    \begin{remark}
The spherical translation operator $\mathcal{L}_{\mathbf{x}}$ is defined analogously to translations in $\mathbb{R},$ where they are represented as a product of basis functions. Specifically, for a complex exponential $e^{inx}$, the translation by $x'$ gives $e^{in(x+x')} = e^{inx} e^{inx'}$.
\end{remark}

Therefore, for a function $f(\cdot)$ on $\mathbb{S}^2$ with harmonic coefficients $\{f_{\ell m}\}$,
\begin{equation*}
		\left(\mathcal{L}_{\mathbf{x}} f\right)\left(\mathbf{x^\prime}\right) = \sum_{\ell\in \mathbb{N}_0} \sum_{m=-\ell}^{\ell} f_{\ell m} Y_{\ell m}(\mathbf{x})Y_{\ell m}(\mathbf{x^\prime}),\quad \mathbf{x^\prime}\in \mathbb{S}^2.
	\end{equation*}
Thus, the spherical Dirac delta has the harmonic expansion
\[
\delta_\mathbf{x}(\mathbf{x'}) = \sum_{\ell,m} Y_{\ell m}(\mathbf{x}) Y_{\ell m}^*(\mathbf{x'}).
\]

	The following result clarifies the spectral structure of $T_\circledast 	(\cdot).$ The sifting convolution acts diagonally in harmonic space, similarly to the isotropic convolution, but allows the weights $h_{\ell m}$ to vary with both $\ell$ and $m$, potentially introducing diagonal anisotropy.
	\begin{theorem}\label{prop:sift}
    For an isotropic random field $T(\mathbf{x})$ $\mathbf{x} \in \mathbb{S}^2,$ given by (\ref{spec}) and a real-valued convolution kernel $H(\mathbf{x}),$ $\mathbf{x} \in \mathbb{S}^2,$ defined by (\ref{Hxx}), their sifting convolution~(\ref{eq:sc}) has the spectral representation
		\begin{equation}\label{Tconv}
			T_\circledast 	(\mathbf{x}) =\sum_{\ell \in \mathbb{N}_0} \sum_{m=-\ell}^{\ell} a_{\ell m, \circledast} Y_{\ell m }(\mathbf{x}).
		\end{equation}
		where $a_{\ell m, \circledast} = a_{\ell m} h_{\ell m}^\ast$.

		Thus, it holds
		\begin{equation*}
		\mathbf{E} a_{\ell m, \circledast}  = 0, \quad \mathbf{E} \left(a_{\ell m, \circledast} a^\ast_{\ell^\prime m^\prime, \circledast}\right) = C_{\ell m,\circledast} \delta_{\ell}^{\ell^\prime}\delta_{m}^{m^\prime},
		\end{equation*}
		where the angular power spectrum satisfies $C_{\ell m,\circledast} = C_\ell \left \vert h_{\ell m} \right \vert^2.$

The sifting convolutional field is zero-mean, and its covariance function is equal to
		\begin{equation*}
			\text{Cov}\left( T_\circledast(\mathbf{x}_1) , T_\circledast(\mathbf{x}_2)\right)= \sum_{\ell \in \mathbb{N}_0} C_{\ell} \sum_{m=-\ell}^{\ell} \left \vert h_{\ell m} \right \vert ^2 Y_{\ell m} (\mathbf{x}_1)Y^\ast_{\ell m} (\mathbf{x}_2), \quad \mathbf{x}_1, \mathbf{x}_2 \in \mathbb{S}^2.
		\end{equation*}
\end{theorem}

Finally, we clarify the relations between diagonal anisotropic and sifting-convolution random fields.

\begin{proposition}\label{prop8}
A sifting convolution random field is diagonal anisotropic if and only if there exists some $\ell \in \mathbb{N}_0$ and distinct $m, m' \in \{-\ell, \ldots, \ell\}$ such that $C_\ell>0$ and $|h_{\ell m}| \neq |h_{\ell m'}|$.
If $|h_{\ell m}| = |h_{\ell m'}|$ for all $m, m'$, the field is isotropic.

A diagonal anisotropic random field can be represented in law as a sifting convolution if and only if its angular power spectrum satisfies the condition
\begin{equation}\label{sqrtsum}
\sum_{\ell \in \mathbb{N}_0} \sum_{m=-\ell}^{\ell} \sqrt{C_{\ell m}} < +\infty.
\end{equation}
\end{proposition}

The condition (\ref{sqrtsum}) ensures that the sifting convolution series converges in $L^2(\mathbb{S}^2)$.

\begin{example} \label{ex6}
Let us use a convolution kernel $H(\mathbf{x}),$ $\mathbf{x} \in \mathbb{S}^2,$ defined by (\ref{Hxx}) with the coefficients \begin{equation}\label{hlm}
h_{\ell m} =
\begin{cases}
(\ell+|m|)^{-1}, & 1\le \ell \le  \ell_{max}, \; m = -\ell, \dots, \ell,\\
0, & else.
\end{cases}
\end{equation}
\begin{figure}[!b]
  \centering
  \begin{subfigure}{0.48\textwidth}
    \centering
    \includegraphics[width=\textwidth, trim={1 0 0 0}, clip]{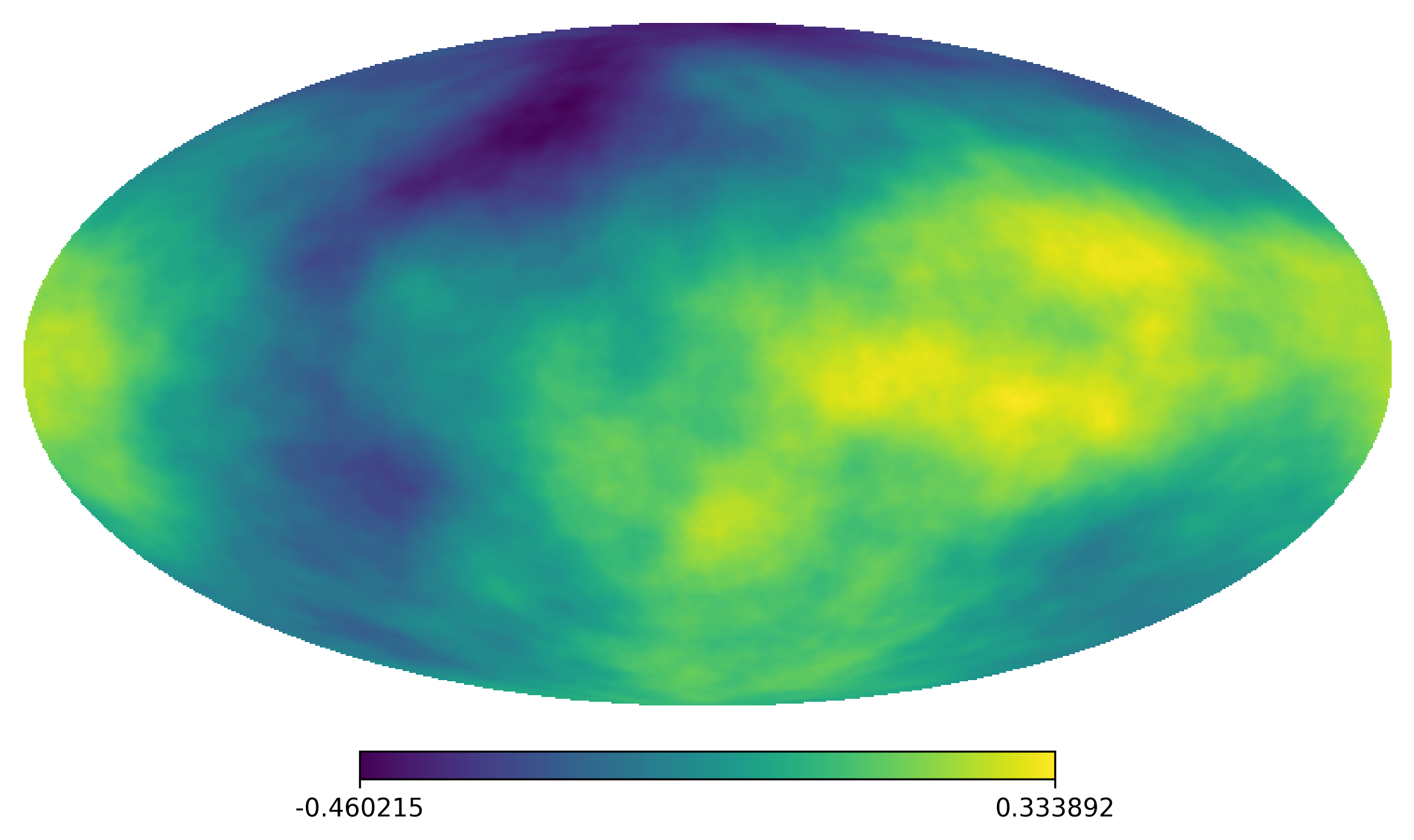}\\[0.2cm]
    \caption{\it Realization of the field}\label{fig5a}
  \end{subfigure}%
  \begin{subfigure}{0.68\textwidth}
    \centering
    \includegraphics[width=\textwidth, trim={4cm 0 0 5.5cm}, clip]{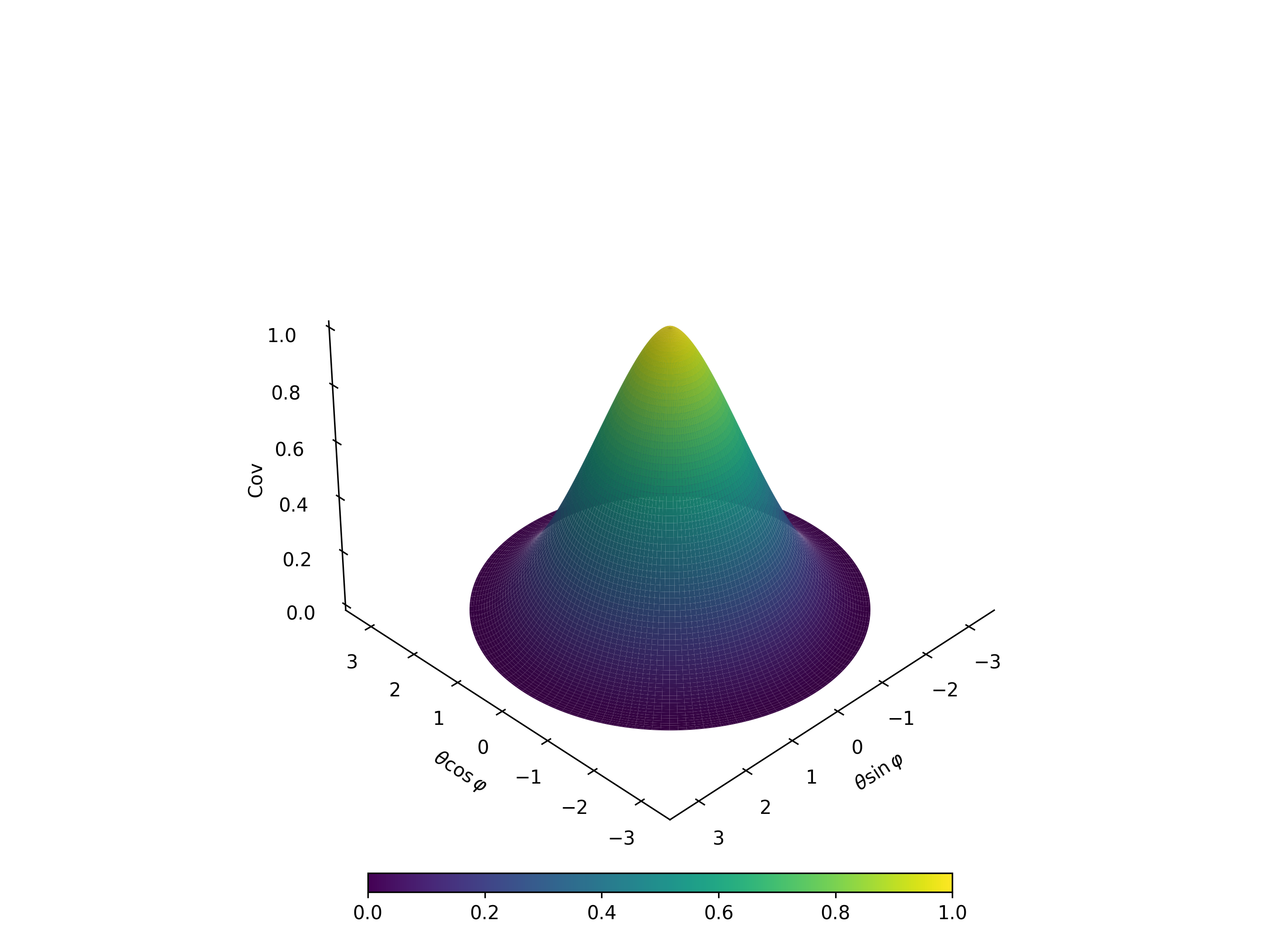}
        \captionsetup{justification=raggedright, singlelinecheck=false, margin=0cm}
    \caption{\it Covariance function  $\text{Cov}\left( T_\circledast(\mathbf{0}) , T_\circledast(\mathbf{x})\right)$}
    \label{fig5b}
  \end{subfigure}
\caption{\it Convolutional field $T_{\circledast}(\mathbf{x})$ with kernel $h_{\ell m}$ given by {\rm (\ref{hlm})}.}
  \label{fig5}
\end{figure}
The realization of the convolutional field $T_{\circledast}(\mathbf{x})$ in Figure~\ref{fig5a} is obtained by applying the spectral representation (\ref{Tconv}) with the harmonic  coefficients $a_{\ell m,\circledast}$ derived from $h_{\ell m}$ and the coefficients $a_{\ell m}$ in Example~\ref{ex1}. The realization exhibits the same pattern of "cold" and "hot" areas, but with a significantly smoother local structure compared to the field shown in Figure~\ref{fig1a}.

By Theorem~\ref{prop:sift} and the properties~(\ref{ylm}), (\ref{Ylll}) and~(\ref{Yl0}) of the spherical harmonics, the corresponding covariance function $\text{Cov}\left( T_\circledast(\mathbf{0}) , T_\circledast(\mathbf{x})\right)$ can be computed as
\begin{align*}
			\text{Cov}\left( T_\circledast(\mathbf{0}) , T_\circledast(\mathbf{x})\right)&= \sum_{\ell \in \mathbb{N}_0} C_{\ell} \sum_{m=-\ell}^{\ell} \left \vert h_{\ell m} \right \vert ^2 Y_{\ell m} (\mathbf{0})Y^\ast_{\ell m} (\mathbf{x})=
            \sum_{\ell \in \mathbb{N}_0} \sqrt{\frac{2\ell+1}{4\pi}}  C_{\ell} \left \vert h_{\ell 0} \right \vert ^2 Y^\ast_{\ell 0} (\mathbf{x})\\
            & =\sum_{\ell \in \mathbb{N}_0} {\frac{2\ell+1}{4\pi}} C_{\ell} \left \vert h_{\ell 0} \right \vert ^2
    P_{\ell}(\cos\theta).
		\end{align*}
Therefore, it is isotropic at the pole $\mathbf{0},$ which is in agreement with Propositions~\ref{prop5} and~\ref{prop8}.
Figure~\ref{fig5b} plots the normalized covariance function, which confirms that the field is isotropic at the pole $\mathbf{0}$.  The strong correlation in Figure~\ref{fig5b} compared to Figure~\ref{fig1b} in Example~\ref{ex1} is due to the convolution smoothing.
\end{example}

\subsection{Spectral representation of deformed random fields}
Finally, we provide some spectral results about $U$-deformed fields $T^U(\cdot).$ A bounded continuous covariance function on the unit sphere guarantees square integrability: $\mathbb E\int|T(U\mathbf{x})|^2d\mathbf{s}(\mathbf{x})\le 4\pi\sup_{\mathbf x}C(\mathbf x,\mathbf x)<\infty$. Therefore, the deformed field $T^U(\cdot)$ admits an $L^2(\mathbb{S}^2)$ expansion in spherical harmonics.

\begin{theorem}\label{defor_spec}
   For an isotropic random field $T(\mathbf{x})$ $\mathbf{x} \in \mathbb{S}^2,$ given by (\ref{spec}), the random field $T^U(\cdot)$ admits the following spectral representation
\begin{equation*}
T^U(\mathbf{x})
=\sum_{\ell' \in \mathbb{N}_0}
\sum_{m'=-\ell'}^{\ell'}
b_{\ell' m'}\, Y_{\ell' m'}(\mathbf{x}),
\end{equation*}
where
\begin{equation*}
b_{\ell' m'}
=
\sum_{\ell \in \mathbb{N}_0}
\sum_{m=-\ell}^{\ell}
Q^{U}_{\ell m,\ell' m'}\, a_{\ell m}
\end{equation*}
and
\begin{equation*}
Q^{U}_{\ell m,\ell' m'}
=
\int_{\mathbb{S}^2}
Y_{\ell m}(U(\mathbf{x}))\,
{Y_{\ell' m'}^*(\mathbf{x})} \,
d\mathbf{s}(\mathbf{x}).
\end{equation*}

The spherical harmonic coefficients $\{b_{\ell m}\}$ are centered random variables with the corresponding angular power spectrum given by
\begin{equation*}
\mathbb{E}\!\left[
b_{\ell m}\, {b^*_{\ell' m'}}
\right]
=
\sum_{\tilde{\ell}  \in \mathbb{N}_0}
C_{\tilde{\ell} }
\sum_{\tilde{m} =-\tilde{\ell} }^{\tilde{\ell} }
Q^{U}_{\tilde{\ell}  \tilde{m},\ell m}\,
{\left(Q^{U}_{\tilde{\ell}  \tilde{m},\ell' m'}\right)^*}.
\end{equation*}
\end{theorem}
In particular, when $U(\cdot)$ is not an isometry, the operator
$Q^U_{\ell m,\ell' m'}$ generally induces coupling across different
harmonic modes, yielding a non-diagonal spectral covariance structure,
which is consistent with the geometric anisotropy established in
Corollary~\ref{cor2}.
The results of Theorem~\ref{defor_spec} can be further specified for particular deformations
$U(\cdot).$ For example, the following result directly follows  from~\eqref{ylm} and~\eqref{Ylms}.

\begin{corollary}
Let the deformations $U(\cdot)$ be a circular scaling transformation of the sphere at the pole $\mathbf{0}$ defined by a circular scaling function $\lambda(\cdot).$ Then, the random field $T^U(\cdot)$ admits the spectral representations from Theorem~\ref{defor_spec} with the coefficients
\[Q^{U}_{\ell m,\ell' m'}
= \delta _{m}^{m^{\prime }}\tilde{Q}^{U}_{\ell \ell',m}\,,\]
where
\[\tilde{Q}^{U}_{\ell \ell',m} = \left[\frac{(2\ell+1)(2\ell'+1)(\ell-m)!(\ell'-m)!}{4(\ell+m)!(\ell'+m)!}\right] ^{1/2}
 \int_{0}^{\pi }P_{\ell}^{m}(\cos(\lambda(\theta)))
    P_{\ell'}^{m}(\cos\theta)\sin \theta d\theta.
\]
\end{corollary}

\section{Conclusion}
The results obtained in this paper suggest several directions for future research:
\begin{enumerate}
\item Extend the analysis to random fields on $n$-dimensional spheres and balls for $n>3$, or on other compact, regular 2-dimensional manifolds, such as tori \cite{Leonenko_ball, Malyarenko1}.
\item Investigate which results remain valid for fields with non-continuous covariance functions. For instance, some properties hold without continuity assumptions, while others, such as Theorem~\ref{t1}, rely on continuity through the hairy ball theorem.
\item Explore analogous questions for complex-valued random fields or fields with other invariance structures, such as covariate-dependent random fields, see \cite{Malyarenko1, Malyarenko}.
\item Develop statistical procedures to distinguish isotropic fields from the anisotropic scenarios studied here \cite{Caponera, Duque}.
\item Generalize the current second-order analysis to higher-order moments or $k$-point covariance functions.
\end{enumerate}










\section{Proofs}
\begin{proof}[Proof of Theorem \ref{t1}]
Since $T(\boldsymbol{\cdot})$ is anisotropic, there are two alternatives:
\begin{enumerate}
    \item  The field is isotropic at each point of $\mathbb{S}^2$, but there exist at least two circles, $O(\mathbf{x}_0, r)$ and $O(\mathbf{x}_0^{\prime}, r)$, of the same radius and with different centers $\mathbf{x}_0 \neq \mathbf{x}_0^{\prime}$, on which the corresponding values of the covariance functions $C(\mathbf{x}_0, \mathbf{x})$ and $C(\mathbf{x}_0^{\prime}, \mathbf{x})$ are distinct constants.
    Obviously, the covariance function is not rotationally invariant.
    \item There exist a point $\mathbf{x}_0$ and a circle $O(\mathbf{x}_0,r),$ where the covariance function $C(\mathbf{x}_0, \mathbf{x}),$ $\mathbf{x}\in O(\mathbf{x}_0,r),$ is not constant.
Consider a tangent plane $T_{\mathbf{x}_0}\mathbb{S}^2$ at $\mathbf{x}_0$ and the orthogonal projection of $O(\mathbf{x}_0,r)$ on it.  Fix an arbitrary direction in $T_{\mathbf{x}_0}\mathbb{S}^2$, and define a tangent vector at $\mathbf{x}_0$ in this direction whose length is $\max_{\mathbf{x} \in O(\mathbf{x}_0, r)} C(\mathbf{x}_0, \mathbf{x}) - \min_{\mathbf{x} \in O(\mathbf{x}_0, r)} C(\mathbf{x}_0, \mathbf{x}) > 0.$ The length is strictly positive, since the covariance $C(\mathbf{x}_0, \mathbf{x})$ is not constant on $O(\mathbf{x}_0,r).$
By the same argument, there exists an angle $\alpha_0>0$ such that, for every nonzero
$\alpha\in(-\alpha_0,\alpha_0)$, the rotation $\rho_{\mathbf{x}_0,\alpha}$ about
$\mathbf{x}_0$ does not preserve the covariance function on
$O(\mathbf{x}_0,r)$. Hence, for every nonzero
$\alpha\in(-\alpha_0,\alpha_0)$, there exists a point
$\mathbf{x}\in O(\mathbf{x}_0,r)$ such that
$C\bigl(\mathbf{x}_0,\rho_{\mathbf{x}_0,\alpha}\mathbf{x}\bigr)
\neq
C(\mathbf{x}_0,\mathbf{x}).
$

Now suppose that the covariance function is rotationally invariant. Then, for every point $\mathbf{x}\in\mathbb{S}^2$, the restriction of the covariance function to the circle $O(\mathbf{x},r)$ is obtained from its restriction to $O(\mathbf{x}_0,r)$ by a rotation of the sphere.  Applying the same rotation to the tangent vector at $\mathbf{x}_0$ defines a tangent vector at $\mathbf{x}$. Since rotations act smoothly on $\mathbb{S}^2$, the covariance function is continuous, and small rotations $\rho_{\mathbf{x},\alpha}$ do not preserve the covariance function, this construction yields a continuous tangent vector field on $\mathbb{S}^2$. Every vector in this field has the same positive length and, hence, the vector field is nowhere vanishing. This contradicts the hairy ball theorem \cite[p.253]{Renteln}, and concludes the proof demonstrating that the covariance function can't be rotationally invariant in this case.
\end{enumerate}\vspace{-6mm}
\end{proof}

\begin{proof} [Proof of Corollary \ref{CC1}]
Part (2) of the proof of Theorem~\ref{t1} immediately implies the results.
\end{proof}

\begin{proof} [Proof of Theorem \ref{th2}]
Let us assume that there are two circles $O(\mathbf{x}_0,r)$ and $O(\mathbf{x}_0^{\prime},r),$ $r \in (0, \pi),$ such that $T(\mathbf{x})$ is isotropic on them, but has different values of $C(\mathbf{x}_0,\mathbf{x}), \: \mathbf{x} \in O(\mathbf{x}_0,r),$ and $C(\mathbf{x}_0^{\prime},\mathbf{x}^{\prime}), \: \mathbf{x}^{\prime} \in O(\mathbf{x}_0^{\prime},r) .$ A geodesic arc connecting $\mathbf{x}_0$ and~$\mathbf{x}_0^{\prime}$ has the length $\gamma(\mathbf{x}_0, \mathbf{x}_0^{\prime}).$ Let us consider $\left\lceil {\gamma(\mathbf{x}_0, \mathbf{x}_0^{\prime})}/{r} \right\rceil+1$ equidistant points (where $\left\lceil \mathbf{\cdot} \right\rceil$ stands for the ceiling functions) along this geodesic arc, where the first and last points coincide with $\mathbf{x}_0$ and  $\mathbf{x}_0^{\prime}$ respectively.  For each of these equidistant points, consider circles $O(\mathbf{x},r)$ with a center located at the respective point.
As every two consecutive circles intersect and the points of intersection are equidistant from their centers, the values of the covariance functions are identical on all consecutive circles. We used isotropy at the points of intersection. Hence, the covariance function must be the same across all these circles, in particular, on $O(\mathbf{x}_0, r)$ and $O(\mathbf{x}_0^{\prime},r),$ which contradicts the assumption.
\end{proof}

\begin{proof} [Proof of Theorem \ref{th3}]
By Theorem~\ref{th2}, the field cannot be isotropic at each point. Therefore, the proof follows the argument presented in the proof of part~2 of Theorem~\ref{t1}. Given that the covariance function $C(\mathbf{x}_1,\mathbf{x}_2)$ is continuous, the circular scaling transformation of the field also has a continuous covariance function on~$\mathbb{S}^2.$ Since $\lambda(\cdot)$ is continuous, a non-zero vector field on $\mathbb{S}^2$ can be constructed analogously to Theorem~\ref{t1}, but using the continuously mapped sets $O(\mathbf{x}_0,\lambda(r))$ instead of $O(\mathbf{x}_0,r)$ for each location on $\mathbb{S}^2.$ It again leads to a contradiction with the hairy ball theorem.
\end{proof}

\begin{proof} [Proof of Proposition \ref{prop1}]
Consider an arbitrary point $\mathbf{x}_0\in \mathbb{S}^2.$ By Definition~\ref{TV}, the covariance function of the field $T_V(\mathbf{x})$ is given by
\begin{equation}\label{mulcov}
C_{T_V}(\mathbf{x}_0,\mathbf{x}):=\mathbf{E}(T_V(\mathbf{x}_0) {T_V(\mathbf{x})})= 
V(\mathbf{x}_0)V(\mathbf{x})B(\gamma (\mathbf{x}_0,\mathbf{x})).\end{equation}
By isotropy of $T(\mathbf{x}),$ it follows that $B(\gamma (\mathbf{x}_0,\mathbf{x}))\equiv const\neq 0,$ if $\mathbf{x} \in O(\mathbf{x}_0,r).$  Since $V(\mathbf{x})$ takes different values on that circle and $V(\mathbf{x}_0)\neq 0,$ therefore, $C_{T_V}(\mathbf{x}_0,\mathbf{x})\not \equiv const,$ if $\mathbf{x} \in O(\mathbf{x}_0,r).$ Hence, the field $T_V(\mathbf{x})$ is not isotropic at every point of $\mathbb{S}^2,$
which, by Definition~\ref{def_sani},  concludes the proof.
\end{proof}

\begin{proof} [Proof of Example \ref{ex2}]
Note that this function is correctly defined as it takes the same value 1 for all $\varphi\in[0, 2\pi)$ if $\theta =0$ or~$\pi$ and when $\varphi =0$ and $2\pi.$ Since $  {V}(\theta, \varphi) \ge 1,$ it is a non-vanishing function.
For each circle of latitude, $\theta$ is the same non-zero constant from the interval $(0, \pi).$ Therefore, as ${V}(\cdot)$ varies with~$\varphi,$ it is not constant on all circles of latitude.

Consider any other circle. It is symmetric with respect to exactly two different meridians. It intersects with some circle of latitude at two points. Let $\varphi$-coordinates of these two intersection points be denoted by $\varphi_1$ and $\varphi_2,$ where  $\varphi_1\not=\varphi_2.$ These points are symmetric with respect to the mentioned two meridians.

{\rm (i)} In the case when  $\varphi_1\not=2\pi-\varphi_2,$  we obtain that  $\varphi_1(2 \pi -\varphi_1)\neq\varphi_2(2 \pi -\varphi_2),$ and, thus, the function ${V}(\cdot)$ takes on at least two different values on that circle.

{\rm (ii)} In the case, when $\varphi_1=2\pi-\varphi_2,$ the considered circle is symmetric with respect to the meridians at $\varphi=0$ and $\pi.$

Then, if the center of the circle is not at the intersection of those meridians and the equator (i.e., the antipodal points $(\pi/2,0)$ or $(\pi/2,\pi)$), there is a meridian at some $\varphi_3$ different from~$0$ and $\pi,$ that intersects the circle at two distinct points. At these points, the values of $\varphi$ are the same, but the values of $\theta(\pi-\theta)$ differ.  Hence, ${V}(\cdot)$ takes two different values at those intersection points.

Finally, if the center of the circle is at $(\pi/2,0)$ and its angular radius is $r\in (0,\pi/2],$ then the points $(\pi/2, r)$ and $(\pi/2+r ,0)$ belong to this circle. If the angular radius is $r\in (\pi/2,\pi),$ then the points $(\pi/2,r )$ and $(3\pi/2-r ,\pi)$ belong to the circle. The corresponding values of the function $  {V}(\cdot)$ are different as
\[  {V}(\pi/2, r )= \pi^2 r (2 \pi -r )/4+1 > 1= \begin{cases}
{V}(\pi/2+r, 0), & \mbox{if} \ r\in (0,\pi/2], \\
  {V}(3\pi/2-r ,\pi), & \mbox{if} \ r\in (\pi/2, \pi) .
\end{cases} \]
The same proof is also applies to the case of the center at $(\pi/2,\pi).$
\end{proof}

\begin{proof} [Proof of Proposition \ref{prop3}]
The statement will be proved by explicitly constructing an example of such fields. Let us consider two independent isotropic random fields $T_{1}(\mathbf{x})$ and $T_{2}(\mathbf{x}),$ $\mathbf{x} \in \mathbb{S}^2,$ with the corresponding covariance functions $B_1(\cdot)$ and $B_2(\cdot),$ see Remark~\ref{rem2}. Let $B_1(\cdot)$ and $B_2(\cdot)$ be Askey covariance functions~\cite{Gneiting}, with the parameters $c_1$ and $c_2,$ $0<c_2<r_0<c_1,$ respectively. Thus, they vanish if the angular power distance exceeds $c_1$ and $c_2,$ respectively, and remain positive otherwise.
Consider $  {T}_{2,V} (\theta, \varphi)=  {V}(\theta, \varphi)\cdot  {T_2}(\theta, \varphi),$ where $  {V}(\theta, \varphi)$ is given by (\ref{V}). By the computations in Example \ref{ex2}, the covariance function of $  {T}_{2, V}$ is not constant on all $O(\mathbf{x},r)$ with $r<c_2.$

Let us define a new field $ {T} (\theta, \varphi):=  {T_1}(\theta, \varphi)+  {T}_{2,V}(\theta, \varphi).$ Then, by the independence of ${T_1}(\cdot)$ and ${T}_{2}(\cdot),$ its covariance is
\begin{align} \label{abcd}
C\left((\theta_1, \varphi_1),(\theta_2, \varphi_2)\right)
&= C_1\left((\theta_1, \varphi_1),(\theta_2, \varphi_2)\right)+
C_{2,V}\left((\theta_1, \varphi_1), (\theta_2, \varphi_2)\right).
\end{align}
As ${T_1}$ is isotropic and $C_{2, V}\left((\theta_1, \varphi_1), (\theta_2, \varphi_2)\right)=0$ for all angular distances greater than $c_2,$ the covariance defined by (\ref{abcd}) equals $C_1\left((\theta_1, \varphi_1),(\theta_2, \varphi_2)\right),$ and therefore isotropic on each circle with an angular radius $r_0>c_2.$ At the same time, as the sum of independent isotropic and strictly anisotropic fields, $  {T}$ is not isotropic at each point of $\mathbb{S}^2,$ it completes the proof.
\end{proof}

\begin{proof}  [Proof of Lemma \ref{op}]
Let us consider the field given by (\ref{sumT}).
By independence of $T_1$ and $T_2$, the covariance function of the field $T$ is equal
\begin{equation} \label{Cc}
C(\mathbf{x}_1,\mathbf{x}_2)=C_{T_1}(\mathbf{x}_1,\mathbf{x}_2)+d(\mathbf{x}_1,J)d(\mathbf{x}_2,J) C_{T_2}(\mathbf{x}_1,\mathbf{x}_2).
\end{equation}

Notice that by continuity of all functions in (\ref{Cc}), $C(\mathbf{x}_1,\mathbf{x}_2)$ is also continuous.
As $d(\mathbf{x},J) = 0,$ if $\mathbf{x} \in J,$  then $C(\mathbf{x}_1,\mathbf{x}_2)= C_{T_1}(\mathbf{x}_1,\mathbf{x}_2),$ if $\mathbf{x}_1 \in J,$ which means that the field $T(\mathbf{x}),$ is isotropic at all points of $J$.

The condition of rotational invariance of the set $\mathbb{S}^2 \setminus J$ is sufficient. Indeed, let us substitute $\mathbf{x}_1=\mathbf{x}_0$ and $\mathbf{x}_2=\mathbf{x}$ in (\ref{sumT}):
\begin{equation} \label{Dd}
C(\mathbf{x}_0,\mathbf{x})=C_{T_1}(\mathbf{x}_0,\mathbf{x})+d(\mathbf{x}_0,J)d(\mathbf{x},J)C_{T_2}(\mathbf{x}_0,\mathbf{x}).
\end{equation}
 As $J$ is rotationally invariant around the axis  passing through the points  $\mathbf{x}_0$ and $\mathbf{x}^{\prime}_0,$ then $d(\mathbf{x},J)$ and all other functions in (\ref{Dd}) take constant values for all $\mathbf{x} \in O(\mathbf{x}_0,r)$ and each angular distance $r \in (0,\pi).$ Hence,  $\mathbf{x}_0$ is a point of isotropy of $T(\mathbf{x}).$

To demonstrate that the condition is necessary, notice that for all $\mathbf{x} \in J$  the distance $d(\mathbf{x}, J) = 0$ and, therefore, by (\ref{Dd}), one gets $C(\mathbf{x}_0, \mathbf{x}) = C_{T_1}(\mathbf{x}_0,\mathbf{x}).$ At the same time, for all $\mathbf{x} \notin J,$ it holds that $d(\mathbf{x}, J) > 0$ and
 $C(\mathbf{x}_0, \mathbf{x}) > C_{T_1}(\mathbf{x}_0 ,\mathbf{x})$ as it was assumed that $C_{T_2}(\mathbf{x}_0 ,\mathbf{x}) > 0.$
Hence, if the set $\mathbb{S}^2 \setminus J$ is not rotationally invariant there is a circle $O(\mathbf{x}_0,r)$ on which $C(\mathbf{x}_0, \mathbf{x})$ takes two different values.
\end{proof}

\begin{proof} [Proof of Theorem \ref{prop4}]
The statement will be proven by explicitly constructing such fields.
Let $T_{1}(\mathbf{x})$ and $T_{2}(\mathbf{x})$ be two independent isotropic random fields, such that $T_2$ has a positive-valued covariance function.

By Lemma~\ref{op}, if $J$ is not rotationally invariant, $T(\mathbf{x})$ defined by (\ref{sumT}) gives a random field that is anisotropic at each point of $\mathbb{S}^2 \setminus J.$

Let us provide a random field that is anisotropic at each point of the rotationally invariant set $\mathbb{S}^2 \setminus J.$ Assume that $J\neq \mathbb{S}^2$ and $J\neq\emptyset$. Then there exists a unique point $\mathbf x_0\in \mathbb S^2\setminus J$ such that $J$ is rotationally invariant about the axis determined by the antipodal pair $\mathbf x_0$ and $\mathbf x_0'$. Without loss of generality, we may take $\mathbf x_0$ to be the north pole $\mathbf 0$.

Let us modify the field $T(\mathbf{x})$ given by (\ref{sumT}) as

\[\widetilde T(\mathbf x)=T_1(\mathbf x)+a(\mathbf x)\,T_2(\mathbf x),\qquad \mathbf x\in\mathbb S^2,\]
where
\[a(\mathbf x,J):=\begin{cases}
    d(\mathbf x,J)\left(1+\frac{(\pi-\theta)\theta}{\pi^2}\cos\varphi\right), & \mbox{if} \ \theta\in (0,\pi),\\
    d(\mathbf x,J),& \mbox{if} \ \theta = 0\ \mbox{or}\  \pi,
\end{cases}\]
$(\theta, \varphi)$ are spherical coordinates of $\mathbf x.$ Note that $a(\cdot,J)$ is a continuous function on $\mathbb S^2$ that vanishes on $J$ and is strictly positive on $\mathbb S^2\setminus J$.

By independence of $T_1$ and $T_2$, the covariance function of the field $\widetilde T$ is equal
\begin{equation} \label{Cc1}
\widetilde C(\mathbf{x}_1,\mathbf{x}_2)=C_{T_1}(\mathbf{x}_1,\mathbf{x}_2)+a(\mathbf{x}_1,J)a(\mathbf{x}_2,J) C_{T_2}(\mathbf{x}_1,\mathbf{x}_2).
\end{equation}

Since the set $J$ is not rotationally invariant about any point $\mathbf x\in \mathbb S^2\setminus J$ that is not a pole, none of these points can be a point of isotropy by analogous arguments to Lemma~\ref{op}.
If $\mathbf x_1$ is a pole and $(\theta_2,\varphi_2)$ are the spherical coordinates of a point $\mathbf x_2\in \mathbb S^2\setminus J$, then, by \eqref{Cc1},
\[
\widetilde C(\mathbf{x}_1,\mathbf{x}_2)=C_{T_1}(\mathbf{x}_1,\mathbf{x}_2)+d(\mathbf{x}_1,J)d(\mathbf{x_2},J)\left(1+\frac{(\pi-\theta_2)\theta_2}{\pi^2}\cos\varphi_2\right) C_{T_2}(\mathbf{x}_1,\mathbf{x}_2).
\]
As the covariance function $C_{T_2}(\mathbf{x}_1,\mathbf{x}_2)$ is continuous, there is a circle of latitude $\theta_2$, which belongs to $\mathbb S^2\setminus J$ and where  $C_{T_2}(\mathbf{x}_1,\mathbf{x}_2)>0.$ Therefore, $\widetilde C(\mathbf{x}_1,\mathbf{x}_2)$ varies with $\varphi_2$.  This implies that $\widetilde T(\mathbf x)$ is anisotropic at the poles, completing the proof.
\end{proof}

\begin{proof} [Proof of Theorem \ref{yu}]
Let us assume that an axially symmetric random field has two different axes of symmetry passing through points
$\mathbf{x}_0$ and $\tilde{\mathbf{x}}_0$, $\mathbf{x}_0 \neq \pm\tilde{\mathbf{x}}_0$.
Then, for all points $\mathbf{x}_1, \mathbf{x}_2 \in \mathbb{S}^2,$ and angles $\alpha, \alpha^{\prime} \in [0,2\pi),$ its covariance function satisfies
\[
C(\mathbf{x}_1, \mathbf{x}_2)
= C(\rho_{\mathbf{x}_0,\alpha}\mathbf{x}_1,\, \rho_{\mathbf{x}_0,\alpha}\mathbf{x}_2)
= C(\rho_{\tilde{\mathbf{x}}_0,\alpha^{\prime}} \rho_{\mathbf{x}_0,\alpha} \mathbf{x}_1,\, \rho_{\tilde{\mathbf{x}}_0,\alpha^{\prime}}\rho_{\mathbf{x}_0,\alpha}\mathbf{x}_2),
\]
Notice, that compositions of rotations $\rho_{\mathbf{x}_0,\alpha}$ and $\rho_{\tilde{\mathbf{x}}_0,\alpha^{\prime}}$ from two different rotational subgroups $SO(2)_{\mathbf{x}_0}$  and $ SO(2)_{\tilde{\mathbf{x}}_0}$ generate the entire group $SO(3).$  Therefore, the considered field is isotropic.
\end{proof}

\begin{proof} [Proof of Theorem \ref{th4}]
Case (1) and (2) are trivial.

In case (3), by (\ref{mulcov}), it is obvious that the constant multipliers on each $O(\mathbf{x}_0,r)$  preserve isotropy at the antipodal points $\mathbf{x}_0$ and $\mathbf{x}_0^{\prime}$.

The assumption that the function $V(\mathbf{x})$ takes different values on two circles of distinct radii implies that, for every $\varepsilon>0$, there exist radii $r_0^{\prime}$ and $r_0^{\prime\prime}$ such that
$|r_0^{\prime}-r_0^{\prime\prime}|<\varepsilon
$
and $V(\mathbf{x})$ takes different values on these corresponding circles. Indeed, suppose the contrary. Then, there would exist some $\varepsilon>0$ such that $V(\mathbf{x})$ is constant on every set of circles whose radii lie within an interval of length $\varepsilon$. The sphere could then be covered by finitely many overlapping spherical segments of a spherical height $\varepsilon$, with $V(\mathbf{x})$ constant on each segment. It implies that $V(\mathbf{x})$ is constant on the entire sphere, contradicting the assumption.
For any point $\mathbf{x}_1$ that is different from $\mathbf{x}_0$ and $\mathbf{x}_0^{\prime},$ there exists a circle $O(\mathbf{x}_1, r_1)$ that intersects both $O(\mathbf{x}_0, r_0^{\prime})$ and $O(\mathbf{x}_0, r_0^{\prime\prime}).$ Therefore, the covariance $C(\mathbf{x}_1,\mathbf{x})$ is not constant on $O(\mathbf{x}_1, r_1)$ and the field $T_V(\mathbf{x})$ is not isotropic at the point $\mathbf{x}_1.$

Case (4) follows directly from the above reasons, as there is no isotropy at any $\mathbf{x}_0.$
\end{proof}

\begin{proof} [Proof of Theorem \ref{thU}]
Let $C^U(\mathbf{x}_1,\mathbf{x}_2)=B(\gamma(U(\mathbf{x}_1),U(\mathbf{x}_2)))$ be the covariance function of the $U$-deformed field $T^U(\cdot)$.

If $U(\cdot)$ is not an isometry, there exist pairs of points with the same geodesic distance whose images under $U(\cdot)$ have different geodesic distances. Any isotropic covariance function $C(\cdot)$ that takes different values for such distances will yield a covariance $C^U(\cdot)$ for $T^U(\cdot)$ that does not depend only on the original geodesic distance.

First, suppose that a random field $T^U(\cdot)$ is isotropic with the corresponding $B^U(\cdot)$ defined by~(\ref{eq1}), but the corresponding bijection $U(\cdot)$ is not an isometry of the sphere.   Thus, there exist points $\mathbf{x}_0$, $\mathbf{x}_1$, and $\mathbf{x}_2$ such that
$
\gamma(\mathbf{x}_0,\mathbf{x}_1)=\gamma(\mathbf{x}_0,\mathbf{x}_2),
$
but
$
\gamma\bigl(U^{(-1)}(\mathbf{x}_0), U^{(-1)}(\mathbf{x}_1)\bigr)\neq
\gamma\bigl(U^{(-1)}(\mathbf{x}_0), U^{(-1)}(\mathbf{x}_2)\bigr),$
where $U^{(-1)}(\cdot)$ denotes the inverse mapping.

Consider an isotropic random field $T(\cdot)$ with its covariance function given by $B\bigl(\gamma(\mathbf{x}_1,\mathbf{x}_2)\bigr)$,  due to~(\ref{eq1}).
Let $B(\cdot)$ be a strictly decreasing function on its domain $[0,\pi]$. For example, the power exponential covariance function~\cite{Gneiting} can serve as $B(\cdot)$. Thus, it holds that
\begin{align*}
    &
C\bigl(U^{(-1)}(\mathbf{x}_0),U^{(-1)}(\mathbf{x}_1)\bigr)
= B\bigl(\gamma(U^{(-1)}(\mathbf{x}_0), U^{(-1)}(\mathbf{x}_1))\bigr)\\
& \quad \neq
B\bigl(\gamma(U^{(-1)}(\mathbf{x}_0), U^{(-1)}(\mathbf{x}_2))\bigr)
= C\bigl(U^{(-1)}(\mathbf{x}_0),U^{(-1)}(\mathbf{x}_2)\bigr),
\end{align*}
which, by construction, implies that
\[
B^U\bigl(\gamma(\mathbf{x}_0,\mathbf{x}_1)\bigr)
= C^U\bigl(\mathbf{x}_0,\mathbf{x}_1\bigr)
\neq
C^U\bigl(\mathbf{x}_0,\mathbf{x}_2\bigr)
= B^U\bigl(\gamma(\mathbf{x}_0,\mathbf{x}_2)\bigr).
\]

As $
\gamma(\mathbf{x}_0,\mathbf{x}_1)=\gamma(\mathbf{x}_0,\mathbf{x}_2),
$ it gives a contradiction with the isotropy of $T^U(\cdot)$.

Conversely, suppose that $U(\cdot)$ is an isometry of the sphere.   By Remark~\ref{isom}, since rotations and reflections about great circles map all spherical circles and their centers to circles and their corresponding centers, the isometry mapping preserves the stationarity of the field. This completes the proof.
    \end{proof}

\begin{proof} [Proof of Corollary~\ref{cor2}]
The isotropy at the point $\mathbf{x}_0$ immediately follows as the transformation $U(\cdot)$ preserves the point $\mathbf{x}_0$ and map the set of all circles $O(\mathbf{x}_0, r),$ $r\in (0,\pi)$ in itself.

Consider an isotropic random field $T(\cdot)$ with covariance function $B\bigl(\gamma(\mathbf{x}_1,\mathbf{x}_2)\bigr)$ defined by~\eqref{eq1}. Notice that for each $\varepsilon>0$ there exist at least two values $r_1$ and $r_2$ such that $0< r_1 < r_2 < \varepsilon$ and $B(r_1)\neq B(r_2)$. Indeed, if this were not the case, then the isotropic covariance function $B(r)$ would be constant on some interval $[0,\varepsilon]$. Consequently, for any $\mathbf{x}'\in \mathbb{S}^2$, the corresponding covariance function $C(\mathbf{x}',\mathbf{x})$ would be constant on the disc $D(\mathbf{x}',\varepsilon)$ centered at $\mathbf{x}'$. By the properties of positive-definite functions, this would imply that it is constant on $\mathbb{S}^2$, which contradicts Condition~\ref{cond2}. Moreover, the values $r_1$ and $r_2$ can be chosen arbitrarily close to each other.

Without loss of generality, we may assume that the point $\mathbf{x}_0$ is the pole $\mathbf{0}$. A circular scaling transformation at $\mathbf{0}$ is determined by a scaling function $\lambda(\cdot)$ and maps every spherical circle into a (possibly different) circle.
Since $\lambda(\cdot)$ is continuous, satisfies $\lambda(0)=0$ and $\lambda(\pi)=\pi$, and is not the identity, there exist $\theta_0 \in (0,\pi)$ and $\delta_0>0$ such that either
$2\lambda(\theta_0)$ is greater or smaller than $\lambda(\theta_0+\delta)+\lambda(\theta_0-\delta)$
for all $\delta\in(0,\delta_0).$

Indeed, consider the function
$g(\theta):=\lambda(\theta)-\theta,$ $\theta\in[0,\pi].$
Then $g(\cdot)$ is continuous on $[0,\pi]$ and satisfies $g(0)=g(\pi)=0$. Since $g$ is not identically zero, it attains either a positive maximum or a negative minimum at some interior point $\theta\in(0,\pi)$.

Suppose that $g(\cdot)$ attains a local maximum. Then there exists the rightmost point $\theta_0$ at which this maximum is attained, and there is some $\delta_0>0$ such that
$g(\theta_0)>g(\theta_0+\delta),$ $ g(\theta_0)\ge g(\theta_0-\delta)$
for all $\delta\in(0,\delta_0)$. Consequently,
\[
2\lambda(\theta_0)>\lambda(\theta_0+\delta)+\lambda(\theta_0-\delta),
\qquad \delta\in(0,\delta_0),
\]
and therefore
\[
\lambda(\theta_0)-\lambda(\theta_0-\delta)>\lambda(\theta_0+\delta)-\lambda(\theta_0),
\qquad \delta\in(0,\delta_0).
\]

It follows that all circles centered at latitude $\theta_0$ with radii not exceeding $\delta_0$ are mapped by $\lambda(\cdot)$ into non-circular curves.

By the properties of the covariance function of the field $T(\cdot)$, choose
$0<r_1<r_2<\varepsilon<\delta_0
$
and a circle on the transformed sphere with the center at $\lambda(\theta_0)$ that intersects both non-circular curves obtained as the images of the circles of radii $r_1$ and $r_2$. This is possible because $r_1$ and $r_2$ may be chosen arbitrarily close. Then the covariance function of the random field $T^U(\cdot)$ takes two distinct values on this circle, which contradicts the global isotropy of the field.

The same argument applies when $g(\cdot)$ attains a local minimum.
\end{proof}

\begin{proof} [Proof of Proposition \ref{prop5}]
It follows from (\ref{iso}) and Definitions \ref{def20} and \ref{def22} that only isotropic, diagonal anisotropic, and harmonic degree-dependent anisotropic fields have spectra of the form $C_{\ell \ell^{\prime} m},$ which is sufficient and necessary for axial symmetry around the axis through $\mathbf{0}$. It follows from their spectral representations that these three classes are disjoint.  By the properties of axially symmetric fields around the axis through pole $\mathbf{0},$ the poles with $\theta = 0$ and $\theta = \pi$ are their points of isotropy.

Then, the expression (\ref{cov_ax}) for the covariance function directly follows from (\ref{spec}), (\ref{ylm}) and (\ref{cov}).
By~(\ref{Ylms}) and the fact that spherical harmonics form a linearly independent system on $\mathbb{S}^2$, the application of the differentiations in (\ref{dif1}) to (\ref{cov}) gives that for all $m$ and $m^{\prime}$ it holds
\[(m-m^{\prime}) C_{\ell \ell^{\prime} m m^{\prime}} =  0.
\]
 This is possible if and only if the coefficients $C_{\ell \ell^{\prime} m m^{\prime}} = 0$ for $m \ne m^{\prime}$.
\end{proof}

\begin{proof} [Proof of Proposition \ref{prop5_1}]
It follows from (\ref{iso}) and Definitions \ref{def21} and \ref{def22} that only isotropic, diagonal anisotropic, and harmonic order-dependent fields have spectra of the form $C_{\ell m m^{\prime}}.$

The expression for the covariance function directly follows from (\ref{ylm}), (\ref{cov}), and $C_{\ell \ell^{\prime} m m^{\prime}} = \delta_{\ell}^{\ell^{\prime}} C_{\ell m m^{\prime}}.$ Since the spherical harmonics are eigenfunctions of the Laplace–Bel\-trami operator, it holds
\[
\Delta_{(\theta, \phi)} Y_{\ell}^{m}(\theta, \phi) = -\ell(\ell + 1)\, Y_{\ell}^{m}(\theta, \phi),
\quad
\ell \in \mathbb{N}_0, \quad m = -\ell, \dots, \ell.
\]
Hence, applying (\ref{dif2}) to (\ref{cov}) and using linear independence of spherical harmonics on $\mathbb{S}^2$ yield
\[
(\ell(\ell+1)- \ell^{\prime}(\ell^{\prime}+1))C_{\ell \ell^{\prime} m m^{\prime}} =0,
\]
which proves that the coefficients $C_{\ell \ell^{\prime} m m^{\prime}} = 0$ if $\ell \neq \ell^{\prime}$.
\end{proof}

\begin{proof} [Proof of Proposition \ref{pisot}]
Consider the covariance function (\ref{cov}) for two points, $\mathbf{x}_1 = \mathbf{0}$ and $\mathbf{x}_2$ with the spherical coordinates $(\theta, \varphi)$.
If the pole $\mathbf{0}$ is a point of isotropy, then, by (\ref{Yl0}) and (\ref{ylm}), the covariance
\begin{align*}
\operatorname{Cov}(T(\mathbf{0}), T(\mathbf{x}_2))
&= \sum_{\ell, \ell^{\prime} \in \mathbb{N}_0}
\sum_{m=-\ell}^{\ell} \sum_{m^{\prime}=-\ell^{\prime}}^{\ell^{\prime}}
C_{\ell \ell^{\prime} m m^{\prime}} Y_{\ell m}(\mathbf{0}) Y_{\ell^{\prime} m^{\prime}}^*(\mathbf{x}_2) \notag\\
&= \sum_{\ell, \ell^{\prime} \in \mathbb{N}_0} \sum_{m^{\prime}=-\ell^{\prime}}^{\ell^{\prime}}
C_{\ell \ell^{\prime} 0 m^{\prime}} \sqrt{\frac{(2\ell+1)}{4\pi}}
Y_{\ell^{\prime} m^{\prime}}^*(\mathbf{x}_2) \notag\\
&= \sum_{m^{\prime}=-\infty}^{+\infty} e^{-i m^{\prime} \varphi}
\sum_{\ell^{\prime} \geq |m^{\prime}|}
\sum_{\ell \in \mathbb{N}_0} C_{\ell \ell^{\prime} 0 m^{\prime}}
\frac{\sqrt{(2\ell+1)(2\ell^{\prime}+1)}}{4\pi}
P_{\ell^{\prime} m^{\prime}}(\cos(\theta)),
\end{align*}
depends only on $\theta$. Due to the linear independence properties of the sets $\{ e^{i m^{\prime} \varphi} \}$ and $\{ P_{\ell^{\prime} m^{\prime}}(\cos(\theta)) \}$, this is possible only when, for any $\ell^{\prime} \in \mathbb{N}$ and $m^{\prime} \neq 0$, it holds that
\[
\sum_{\ell \in \mathbb{N}_0} C_{\ell \ell^{\prime} 0 m^{\prime}} \sqrt{2\ell+1} = 0.
\]
\end{proof}

\begin{proof} [Proof of Theorem \ref{th5}]
Using (\ref{spec}), (\ref{Ylll}), (\ref{eqzz}) and the definition of Gaunt integral, given by~(\ref{eq:gaunt}) in Appendix~\ref{appB}, yields
\begin{align*}
	a_{\ell m,V} & = \int_{\mathbb{S}^2} \left(\sum_{\ell _1 \in \mathbb{N}_0} \sum_{m_1=-\ell _1}^{\ell _1} v_{\ell _1 m_1} Y_{\ell _1 m_1} (\mathbf{x}) \right)\left(\sum_{\ell _2\in \mathbb{N}_0}\sum_{m_2=-\ell _2}^{\ell _2} a_{\ell _2 m_2} Y_{\ell _2 m_2} (\mathbf{x}) \right) Y^{\ast}_{\ell m}(\mathbf{x}) d \mathbf{s} (\mathbf{x})\\
			\notag & =  \sum_{\ell_1 \in \mathbb{N}_0}\sum_{\ell_2\in \mathbb{N}_0} \sum_{m_1=-\ell_1}^{\ell_1} \sum_{m_2=-\ell_2}^{\ell_2} v_{\ell_1 m_1}a_{\ell_2 m_2}  \int_{\mathbb{S}^2}  Y_{\ell_1 m_1}  (\mathbf{x}) Y_{\ell_2 m_2} (\mathbf{x})Y^{\ast}_{\ell m}(\mathbf{x}) d \mathbf{s} (\mathbf{x})\\
			& = (-1)^m  \sum_{\ell_1 \in \mathbb{N}_0}\sum_{\ell_2\in \mathbb{N}_0} \sum_{m_1=-\ell_1}^{\ell_1} \sum_{m_2=-\ell_2}^{\ell_2} v_{\ell_1 m_1}a_{\ell_2 m_2}\sqrt{\frac{(2\ell_1+1)(2\ell_2+1)(2\ell+1)}{4\pi}} \\
			& \quad\times \begin{pmatrix}
				\ell_1&\ell_2& \ell\\
				0& 0& 0
			\end{pmatrix}
			\begin{pmatrix}
				\ell_1&\ell_2& \ell\\
				m_1& m_2& -m
			\end{pmatrix}.
		\end{align*}
		 Since for all $\ell \in \mathbb{N}_0$ and $m \in \{-\ell, \ldots, \ell \}$ it holds that $\mathbf{E}a_{\ell m} = 0,$ the random variables $a_{\ell m,V} $ and the field $T_V$ are centered.

		Then, by  (\ref{Yl0})
		\begin{align}
			\mathbf{E}(a_{\ell m,V}a^\ast_{\ell^\prime m^\prime,V}) =& (-1)^{m+m^\prime}   \sum_{\ell_1 \in \mathbb{N}_0} \sum_{\ell_1^\prime\in \mathbb{N}_0} \sum_{\ell_2 \in \mathbb{N}_0} \sum_{ \ell_2^\prime\in \mathbb{N}_0} \sum_{m_1=-\ell_1}^{\ell_1} \sum_{m_1^\prime=-\ell^\prime_1}^{\ell^\prime_1}\sum_{m_2=-\ell_2}^{\ell_2} \sum_{m_2^\prime=-\ell_2^\prime}^{\ell_2^\prime}v_{\ell_1 m_1}v_{\ell_1^\prime m_1^\prime}^*\nonumber\\
			\quad\times & \sqrt{\frac{(2\ell_1+1)(2\ell_2+1)(2\ell+1)}{4\pi}} \sqrt{\frac{(2\ell^\prime_1+1)(2\ell^\prime_2+1)(2\ell^\prime+1)}{4\pi}} \mathbf{E}\left(a_{\ell_2 m_2}a^{\ast}_{\ell^\prime_2 m^\prime_2}\right)\nonumber\\
			\quad\times&	\begin{pmatrix}
				\ell_1&\ell_2& \ell\\
				0& 0& 0
			\end{pmatrix}
			\begin{pmatrix}
				\ell^\prime_1&\ell^\prime_2& \ell^\prime\\
				0& 0& 0
			\end{pmatrix}
			\begin{pmatrix}
				\ell_1&\ell_2& \ell\\
				m_1& m_2& -m
			\end{pmatrix}
			\begin{pmatrix}
				\ell^\prime_1&\ell^\prime_2& \ell^\prime\\
				m^\prime_1& m^\prime_2& -m^\prime
			\end{pmatrix}\nonumber\\
			= & (-1)^{m+m^\prime} \sum_{\ell_1\in \mathbb{N}_0}  \sum_{\ell_1^\prime\in \mathbb{N}_0} \sum_{\ell_2\in \mathbb{N}_0} \sum_{m_1=-\ell_1}^{\ell_1} \sum_{m_1^\prime=-\ell^\prime_1}^{\ell^\prime_1}\sum_{m_2=-\ell_2}^{\ell_2} v_{\ell_1 m_1}v_{\ell_1^\prime m_1^\prime}^*C_{\ell_2}\nonumber\\
			\quad\times&\frac{2\ell_2+1}{4\pi}\sqrt{(2\ell_1+1)(2\ell_1^\prime+1)(2\ell+1)(2\ell^{\prime}+1)}\nonumber\\
			\quad\times&	\begin{pmatrix}
				\ell_1&\ell_2& \ell\\
				0& 0& 0
			\end{pmatrix}
			\begin{pmatrix}
				\ell^\prime_1&\ell_2& \ell^\prime\\
				0& 0& 0
			\end{pmatrix}
			\begin{pmatrix}
				\ell_1&\ell_2& \ell\\
				m_1& m_2& -m
			\end{pmatrix}
			\begin{pmatrix}
				\ell^\prime_1&\ell_2& \ell^\prime\\
				m^\prime_1& m_2& -m^\prime
			\end{pmatrix}. \label{eq:cov}
		\end{align}
        It follows from the selection rules $|\ell^\prime_1-\ell_2| \leq \ell^\prime \leq\ \ell^\prime_1+\ell_2$  for the Wigner $3j$-symbols in Appendix~\ref{appB}, that  $|\ell_2 -\ell^\prime| \leq \ell^\prime_1 \leq \ell_2 +\ell^\prime$ and analogously $|\ell_2 -\ell| \leq \ell_1 \leq \ell_2 +\ell.$ Also, because for non-zero $3j$-symbols in (\ref{eq:cov}) it holds $m_1=m-m_2$ and $m^\prime_1=m^\prime-m_2$, the expressions for the coefficients $a_{\ell m,V}$ and $C_{\ell \ell^{\prime} m m^{\prime}}$ above can be rewritten as in~(\ref{eq:mcoeff}) and (\ref{spectral_mat}), respectively.
      Finally, noting that
		\begin{align*}
			\mathbf{E}\left( T_V (\mathbf{x}_1)T_V (\mathbf{x}_2)\right)& = \sum_{\ell\in \mathbb{N}_0} \sum_{ \ell^\prime \in \mathbb{N}_0} \sum_{m=-\ell}^{\ell}\sum_{m^\prime =-\ell^\prime}^{\ell^\prime} \mathbf{E}\left(a_{\ell m,V}a^\ast_{\ell^\prime m^\prime,V}\right) Y_{\ell m}(\mathbf{x}_1)Y^\ast_{\ell^\prime m^\prime}(\mathbf{x}_2)
		\end{align*}
        completes the proof of the theorem.
\end{proof}

\begin{proof} [Proof of Theorem \ref{Ta}]
		Choosing the coordinate system which aligns the isotropy center with one of the two poles, the selection rules for the Wigner $3j$-symbols, see Appendix~\ref{appB}, yield
		\begin{align*}
			a_{\ell m,V}& = (-1)^m\sum_{\ell_1\in \mathbb{N}_0} \sum_{\ell_2\in \mathbb{N}_0}  \sum_{m_2=-\ell_2}^{\ell_2} v_{\ell_1 0} a_{\ell_2 m_2}\sqrt{\frac{(2\ell_1+1)(2\ell_2+1)(2\ell+1)}{4\pi}}\\
            & \times \begin{pmatrix}
				\ell_1&\ell_2& \ell\\
				0& 0& 0
			\end{pmatrix}
			\begin{pmatrix}
				\ell_1&\ell_2& \ell\\
				0& m_2& -m
			\end{pmatrix}\\
			&= (-1)^m  \sum_{\ell_1\in \mathbb{N}_0} \sum_{\ell_2\in \mathbb{N}_0}  v_{\ell_1 0} a_{\ell_2 m}\sqrt{\frac{(2\ell_1+1)(2\ell_2+1)(2\ell+1)}{4\pi}} \begin{pmatrix}
				\ell_1&\ell_2& \ell\\
				0& 0& 0
			\end{pmatrix}
			\begin{pmatrix}
				\ell_1&\ell_2& \ell\\
				0& m& -m
			\end{pmatrix}.
		\end{align*}
		Thus, \eqref{eq:cov} becomes
		\begin{align*}
			\mathbf{E}(a_{\ell m,V}a^\ast_{\ell^\prime m^\prime,V})	& = \delta^{m^{\prime}}_{m}  \sum_{\ell_1\in \mathbb{N}_0} \sum_{\ell_1^\prime\in \mathbb{N}_0} \sum_{\ell_2\in \mathbb{N}_0}  v_{\ell_1 0}v_{\ell_1^\prime 0} \frac{2\ell_2+1}{4\pi}C_{\ell_2}\\
			\notag	&\quad \times \sqrt{(2\ell_1+1)(2\ell_1^\prime+1)(2\ell+1)(2\ell^{\prime}+1)}\\
			&\quad\times \begin{pmatrix}
				\ell_1&\ell_2& \ell\\
				0& 0& 0
			\end{pmatrix}
			\begin{pmatrix}
				\ell^\prime_1&\ell_2& \ell^\prime\\
				0& 0& 0
			\end{pmatrix}
			\begin{pmatrix}
				\ell_1&\ell_2& \ell\\
				0& m& -m
			\end{pmatrix}
			\begin{pmatrix}
				\ell^\prime_1&\ell_2& \ell^\prime\\
				0& m& -m
			\end{pmatrix}\\
			& =  \delta_{m}^{m^\prime}C_{\ell \ell^\prime m}.
		\end{align*}
Analogously to the proof of Theorem~\ref{th5}, the application of the selection rules for the Wigner $3j$-symbols gives the expressions for the coefficients $a_{\ell m,V}$ and $C_{\ell \ell^{\prime} m}.$
\end{proof}

\begin{proof} [Proof of Corollary \ref{C2}]

        By substituting $\mathbf{x}_2 = \mathbf{0},$ (\ref{Yl0}) and (\ref{eq:cllm}) in
        \[\mathbf{E}\left( T_V (\mathbf{x}_1)T_V (\mathbf{x}_2)\right) = \sum_{\ell\in \mathbb{N}_0} \sum_{ \ell^\prime \in \mathbb{N}_0} \sum_{m=-\min(\ell,\ell^{\prime})}^{\min(\ell,\ell^{\prime})} C_{\ell \ell^{\prime} m} Y_{\ell m}(\mathbf{x}_1)Y^\ast_{\ell^\prime m}(\mathbf{x}_2)
		\]
        one obtains
		\begin{align*}
			\mathbf{E}\left( T_V (\mathbf{x})T_V (\mathbf{0})\right) & =\sum_{\ell \in \mathbb{N}_0}  \sum_{\ell^\prime \in \mathbb{N}_0}\sum_{\ell_1,\ell_1^\prime\in \mathbb{N}_0}\sum_{\ell_2\in \mathbb{N}_0}  v_{\ell_1 0}v_{\ell_1^\prime 0} \frac{2\ell_2+1}{4\pi}C_{\ell_2}	\begin{pmatrix}
				\ell_1&\ell_2& \ell\\
				0& 0& 0
			\end{pmatrix}^2
			\begin{pmatrix}
				\ell^\prime_1&\ell_2& \ell^\prime\\
				0& 0& 0
			\end{pmatrix}^2\\
			\notag	&\quad \times \sqrt{(2\ell_1+1)(2\ell_1^\prime+1)(2\ell+1)(2\ell^{\prime}+1)}  Y_{\ell 0}(\mathbf{x})\sqrt{\frac{2\ell^\prime+1}{4\pi}}\\
			&= \sum_{\ell \in \mathbb{N}_0}\sum_{\ell_1\in \mathbb{N}_0} \sum_{\ell_1^\prime\in \mathbb{N}_0} \sqrt{(2\ell_1+1)(2\ell_1^\prime+1)} v_{\ell_1 0}v_{\ell_1^\prime 0} \sum_{\ell_2\in \mathbb{N}_0} \frac{2\ell_2+1}{4\pi}C_{\ell_2}\\
			&\quad \times 	\begin{pmatrix}
				\ell_1&\ell_2& \ell\\
				0& 0& 0
			\end{pmatrix}^2
			\left(\sum_{\ell^\prime \in \mathbb{N}_0}	(2\ell^\prime+1)\begin{pmatrix}
				\ell^\prime_1&\ell_2& \ell^\prime\\
				0& 0& 0
			\end{pmatrix}^2\right)\sqrt{\frac{2\ell+1}{4\pi}}  Y_{\ell 0}(\mathbf{x}).
		\end{align*}
		The application of the formula (\ref{ylm}) and the orthogonality property of the Wigner \mbox{$3j$-symbols,} consult Appendix~\ref{appB}, yields
		\begin{align*}
			\mathbf{E}\left( T_V (\mathbf{x})T_V (\mathbf{0})\right)
			& = \left( \sum_{\ell_1^\prime\in \mathbb{N}_0} \sqrt{2\ell_1^\prime +1 } v_{\ell_1^\prime 0}   \right)\sum_{\ell_1\in \mathbb{N}_0} \sqrt{2\ell_1 +1 } v_{\ell_1 0} \sum_{\ell_2\in \mathbb{N}_0} \frac{2\ell_2+1}{4\pi}C_{\ell_2}  \\
			& \times \sum_{\ell\in  \mathbb{N}_0} \frac{2\ell+1}{4\pi}P_{\ell}(\cos \theta)	\begin{pmatrix}
				\ell_1 & \ell_2 & \ell\\
				0 & 0 & 0%
			\end{pmatrix}^2.
		\end{align*}
Finally, by applying the selection rules for the Wigner
$3j$-symbols, one obtains the statement of the corollary.
	\end{proof}

\begin{proof} [Proof of Theorem \ref{thtt}]

It follows from (\ref{Ylms}), (\ref{eq:ic})  and (\ref{legen}) that
\begin{align*}
T_{\odot H}(\mathbf{x})
&= \int_{\mathbb{S}^2}
\sum_{\ell \in \mathbb{N}_0} \sum_{m=-\ell}^{\ell} \sum_{\ell^\prime\in \mathbb{N}_0}
a_{\ell m} h_{\ell^\prime}\sum_{m^{\prime}=-\ell^\prime}^{\ell^\prime} Y_{\ell m}(\mathbf{x}^{\prime})\, Y_{\ell^\prime m^\prime}(\mathbf{x})\,
Y_{\ell^{\prime} m^{\prime}}^{*}(\mathbf{x^\prime}) \, d\mathbf{s}(\mathbf{x^\prime}) \\
&= \sum_{\ell \in \mathbb{N}_0} \sum_{m=-\ell}^{\ell}
a_{\ell m}\, h_{\ell}\,
Y_{\ell m}(\mathbf{x}),
\end{align*}
which gives the representation (\ref{tdot}).

Then, the representation for the covariance function of $T_{\odot H}(\cdot)$ immediately follows from (\ref{iso})  and the Legendre addition theorem for spherical harmonics
\begin{equation*}
\sum_{m=-\ell}^{\ell} Y_{\ell m}(\mathbf{x}) \, Y_{\ell m}^{*}(\mathbf{x}^{\prime})
=\frac{2\ell+1}{4\pi}P_\ell(\cos\gamma(\mathbf{x},\mathbf{x}')). 
\end{equation*}
\end{proof}

\begin{proof} [Proof of Theorem \ref{th8}]

By (\ref{Ylms}), (\ref{eq:cac}) and the relation (\ref{AA}) in Appendix~\ref{appC}, taking into account that for a real-valued convolution kernel it holds $H(\mathbf{x})=H^*(\mathbf{x}),$ one obtains that
\begin{align*}
T_{\oplus H}(\mathbf{x})
&= \int_{\mathbb{S}^2} \sum_{\ell \in \mathbb{N}_0} \sum_{m=-\ell}^{\ell}
   \sum_{L \in \mathbb{N}_0} \sum_{M=-L}^{L} \sum_{m^{\prime}=-L}^{L}
   a_{\ell m}\, h_{LM}^{\ast}\, D^{L{\ast}}_{M m^{\prime}}(\rho) \,
   Y_{\ell m}(\mathbf{x}^{\prime}) \, Y^{\ast}_{L m^{\prime}}(\mathbf{x}^{\prime}) \, d \mathbf{s}(\mathbf{x}^{\prime}) \\
&= \sum_{\ell \in \mathbb{N}_0} \sum_{m=-\ell}^{\ell}
   a_{\ell m}  \sum_{M=-\ell}^{\ell}
   h_{\ell M}^{\ast}\, D^{\ell \ast}_{M m}(\rho)\\
   &= \sum_{\ell \in \mathbb{N}_0} \sum_{m=-\ell}^{\ell} \sum_{M=-\ell}^{\ell} (-1)^m
   a_{\ell m} h^{\ast}_{\ell M} e^{i(M-m)\varphi} d^{\ell}_{M m}(\theta),
\end{align*}
where the last equality follows by applying (\ref{BB}) in Appendix~\ref{appC} to $D_{Mm}^{\ell}\left(\varphi,\theta,\pi-\theta\right).$

As $T_\oplus$ is centered and by (\ref{iso}), its covariance function takes the form
		\begin{align*}
			\mathbf{E}&\left(T_{\oplus H} (\mathbf{x}_1)T_{\oplus H}(\mathbf{x}_2)\right)  = \sum_{\ell\in \mathbb{N}_0} \sum_{\ell ^\prime\in \mathbb{N}_0}
            \sum_{m=-\ell}^{\ell}
            \sum_{m^\prime=-\ell^\prime}^{\ell^\prime}
            \sum_{M=-\ell}^{\ell} \sum_{M^\prime=-\ell^\prime}^{\ell^\prime} \mathbf{E}\left(a_{\ell m}a_{\ell^\prime m^\prime}^\ast\right) h^\ast_{\ell M}h_{\ell^\prime M^\prime} \\
			 & \times\,(-1)^{m+m^\prime} e^{i\left[\left(M-m\right)\varphi_1-\left(M^\prime-m^\prime\right)\varphi_2\right]}d^{\ell}_{Mm}(\theta_1)
			d^{\ell^\prime}_{M^\prime m^\prime}(\theta_2)\\
   			& = \sum_{\ell\in \mathbb{N}_0}
            \sum_{m=-\ell}^{\ell}
          \sum_{M=-\ell}^{\ell}\sum_{M^\prime=-\ell}^{\ell} C_{\ell } h^\ast_{\ell M}h_{\ell  M^\prime}  e^{-im(\varphi_1-\varphi_2)}e^{i(M\varphi_1-M^{\prime}\varphi_2)}
			d^{\ell}_{M m}(\theta_1) d^{\ell}_{M^{\prime} m}(\theta_2)\\
			&= \sum_{\ell \in \mathbb{N}_0 } \sum_{M=-\ell}^{\ell}
          \sum_{M^\prime=-\ell}^{\ell}  C_{\ell } h^\ast_{\ell M}h_{\ell  M^\prime} e^{i\left(M\varphi_1-M^\prime\varphi_2 \right)} e^{-iM \Phi} e^{-i M^{\prime} \Psi} d^{\ell}_{M M^\prime} (\Theta),
		\end{align*}
where the last equality follows from the symmetry property and the addition theorem for elements of the Wigner $D$-matrices in Appendix~\ref{appC}.
\end{proof}

\begin{proof} [Proof of Theorem \ref{prop:sift}]
		By (\ref{eq:sc}) and (\ref{oper}), applying $\mathcal{L}_{\mathbf{x}}$ to the isotropic random field $T$ yields
		\begin{align*}
			T_{\circledast} (\mathbf{x})& = \int_{\mathbb{S}^2} \left(\sum_{\ell \in \mathbb{N}_0}   \sum_{m=-\ell}^{\ell} a_{\ell m} Y_{\ell m} \left(\mathbf{x}\right)  Y_{\ell m} \left(\mathbf{x^\prime}\right)\right)
			\left(\sum_{\ell^\prime \in \mathbb{N}_0} \sum_{m^\prime = -\ell^{\prime}}^{\ell^{\prime}} h_{\ell^\prime m^\prime}^\ast Y_{\ell^{\prime} m^{\prime}}^\ast \left(\mathbf{x^\prime}\right)  \right) d \mathbf{s}(\mathbf{x^\prime})\\
			&=\sum_{\ell \in \mathbb{N}_0}  \sum_{m=-\ell}^{\ell} a_{\ell m} h_{\ell m}^\ast Y_{\ell m} \left(\mathbf{x}\right).
		\end{align*}
		Since $T$ is centered, also $\mathbf{E}\left(T_{\circledast}(\mathbf{x})\right)=0$ for any $\mathbf{x} \in \mathbb{S}^2$. Finally, by isotropy of $T,$ we have that
		\begin{equation*}
			\mathbf{E}\left(a_{\ell m,{\circledast}}a_{\ell^\prime m^\prime ,{\circledast}}^\ast\right)= C_{\ell } \left \vert h_{\ell m} \right \vert^2 \delta_{\ell}^{\ell^\prime}\delta_{m}^{m^\prime},
		\end{equation*}
which immediately implies other results of Theorem \ref{prop:sift}.
\end{proof}

\begin{proof}[Proof of Proposition~\ref{prop8}]
It follows from Theorem~\ref{prop:sift}, condition~(\ref{iso}), and Definition~\ref{def22} that a sifting convolutional random field is either isotropic or diagonal anisotropic.

The angular power spectrum
$C_{\ell m, \circledast} = C_{\ell}\, |h_{\ell m}|^2$
varies with $m$ for some fixed $\ell$ if and only if there exist
$m \neq m^{\prime}$ such that $|h_{\ell m}| \neq |h_{\ell m^{\prime}}|$. To represent in law a diagonal anisotropic field as a sifting convolution, one needs two sets of non–negative coefficients $\{C_{\ell}\}$ and
$\{h_{\ell m}\}$, such that its angular power spectrum can be given as $C_{\ell m} = C_{\ell}\, h_{\ell m}^2,$ where $h_{\ell m} = h^*_{\ell m}.$

Since these coefficients correspond to harmonic representations of
isotropic random fields and $L_2(\mathbb{S}^2)$-integrable convolution
kernel, they must satisfy the conditions
\[
\sum_{\ell \in \mathbb{N}_0} (2\ell+1)C_{\ell} =\sum_{\ell \in \mathbb{N}_0} \sum_{m=-\ell}^{\ell}C_{\ell} < \infty
\quad \text{and} \quad
\sum_{\ell \in \mathbb{N}_0} \sum_{m=-\ell}^{\ell} h_{\ell m}^2 < \infty.
\]

Hence, the problem is equivalent to factorisation $a_n = a_n^{(1)} a_n^{(2)},$
where $\{a_n\},$ $\{a_n^{(1)}\}$ and $\{a_n^{(2)}\}$ are non–negative sequences,
and the series formed by adding elements of each sequence must converge.

Let us show that
\begin{equation} \label{suman}
\sum_{n} \sqrt{a_n} < \infty
\end{equation}
is a sufficient and necessary condition for it.

If (\ref{suman}) holds true, then selecting
$a_n^{(1)} = a_n^{(2)} = \sqrt{a_n}$
gives the required result, as
$\sum_{n} a_n <  \max_{n}(\sqrt{a_n})\cdot\sum_{n}\sqrt{a_n} < +\infty,$ because the maximum is finite by convergence condition (\ref{suman}). Now let us assume that $\{a_n\}$ is factorisable into two sequences with
convergent series, but (\ref{suman}) does not hold. Then, by Hölder’s inequality,
\[
\sum_n \sqrt{a_n}
= \sum_n \sqrt{a_n^{(1)} a_n^{(2)}}
\le \left( \sum_n a_n^{(1)} \sum_n a_n^{(2)} \right)^{1/2}
< +\infty,
\]
and we obtain a contradiction.

The application to the sequence $\{C_{\ell m}\}$ completes the proof.
  \end{proof}

\begin{proof}[Proof of Theorem~\ref{defor_spec}]
As the deformation $U(\cdot)$ does not alter the mean-square integrability of the spherical harmonics $\{Y_{\ell m}(\cdot)\}$, their $U$-deformed versions can be represented as
\begin{equation}\label{spherU}
Y_{\ell m}(U(\mathbf{x}))
=
\sum_{\ell' \in \mathbb{N}_0}
\sum_{m'=-\ell'}^{\ell'}
Q^{U}_{\ell m,\ell' m'} \,
Y_{\ell' m'}(\mathbf{x}),
\end{equation}
where the corresponding coefficients are defined by
\begin{equation*}
Q^{U}_{\ell m,\ell' m'}
=
\int_{\mathbb{S}^2}
Y_{\ell m}(U(\mathbf{x}))\,
{Y_{\ell' m'}^*(\mathbf{x})} \,
d\mathbf{s}(\mathbf{x}).
\end{equation*}
Then, the results of the theorem follow straightforwardly from substituting~\eqref{spherU} into
\[T^U(\mathbf{x})=T(U(\mathbf{x}))
=
\sum_{\ell \in \mathbb{N}_0}
\sum_{m=-\ell}^{\ell}
a_{\ell m}\, Y_{\ell m}(U(\mathbf{x}))\]
and using the uncorrelatedness identity~(\ref{iso}).   \end{proof}

\appendix

\section{Simulation of spherical random fields}\label{appA}
The realizations of the random fields used in the examples were generated using the \texttt{healpy} package in Python. The field values were simulated at HEALPix pixels according to the chosen angular power spectrum. The results were visualized using the Mollweide projection of the sphere \cite{Snyder1987}.

The HEALPix scheme divides the sphere into equal-area pixels. The resolution is controlled by the parameter $N_{\text{side}} = 2^j$, $j \in \mathbb{N}_0$, which determines the fineness of the pixelization. The total number of pixels on the sphere is
$N_{\text{pix}} = 12 \, N_{\text{side}}^2.
$
The maximum multipole moment resolved in the spherical harmonic expansion is
$\ell_{\text{max}} = 3 \, N_{\text{side}} - 1,
$
which sets the smallest angular scale captured by the simulation~\cite{gorski2005healpix}.

For all examples in this paper, we used the parameter $N_{\text{side}} = 512$, corresponding to $N_{\text{pix}} = 3,145,728$ pixels and the maximum multipole $\ell_{\text{max}} = 1535$.

\section{Wigner symbols and Clebsch-Gordan coefficients}\label{appB}

The Wigner $3j$-symbols and Clebsch-Gordan coefficients have been introduced in quantum mechanics and angular momentum theory. They are used in the representation theory of $SO(3)$ and encode the coupling of angular momenta $\ell$ and $\ell^\prime$, to form a total angular momentum $L$.
We briefly recall their definitions and main properties. Further details can be found in \cite{vmk}.

In the context of random fields on the sphere, Clebsch-Gordan coefficients are used to analyse the coupling between spherical harmonics. The product of two spherical harmonics admits an expansion in the spherical harmonic basis and can be thus represented as a weighted sum of spherical harmonics, with the Clebsch-Gordan coefficients $C_{\ell m,\ell^\prime m^\prime}^{LM}$ serving as the weights in the coupling relation
\begin{equation*}
Y_{\ell m}(\theta, \phi) Y_{\ell^\prime m^\prime}(\theta, \phi) =\sum_{L\in \mathbb{N}_0} \sum_{M=-L}^{L} \sqrt{\frac{(2\ell+1)(2\ell'+1)}{4\pi(2L+1)}}\,C_{\ell0,\ell'0}^{L0}C_{\ell m,\ell'm'}^{LM}Y_{LM }(\theta, \phi),
\end{equation*}
where $\ell, \ell^\prime,L \in \mathbb{N}_0,$ $ m \in \{-\ell,\ldots,\ell \},$ $ m^\prime \in \{-\ell^\prime,\ldots,\ell^\prime \}$ and $M \in \{-L,\ldots, L\},$ see \cite[Section~5.6.2, (9)]{vmk}.

For $\ell, \ell^\prime,  L \in \mathbb{N}_0 $, the Wigner $3j$-symbol is defined by \begin{equation} \label{Ell}
\begin{pmatrix} \ell & \ell^\prime & L \\ m & m^\prime & M \end{pmatrix} := \frac{(-1)^{\ell - \ell^\prime - M}}{\sqrt{2L + 1}} C_{\ell m,\ell^\prime m^\prime}^{L(-M)},
\end{equation}
see \cite[Section~8.1.2, (12)]{vmk}.

The $3j$-symbol  vanishes unless the following selection rules are satisfied, see \cite[Section~8.1.1, (1) and (2)]{vmk},
\begin{itemize}
        \item $|m|\leq \ell, \: |m^{\prime}|\leq \ell^\prime, \:|M|\leq L, $
	\item $ \left \vert \ell - \ell^\prime \right \vert  \leq L \leq \ell + \ell^\prime $,
	\item $ m + m^\prime +M=0,$
\end{itemize}

By \eqref{Ell} and \cite[Section~8.1.1, (8)]{vmk}, the following unitarity relation holds
\begin{equation*}
\sum_{m=-\ell}^{\ell} \sum_{m^\prime=-\ell^\prime}^{\ell^\prime} \begin{pmatrix} \ell & \ell^\prime & L \\ m & m^\prime & M \end{pmatrix} \begin{pmatrix} \ell & \ell^\prime & L^\prime \\ m & m^\prime & M^\prime \end{pmatrix} = \frac{\delta_{L}^{L^\prime}\delta_{M}^{M^\prime} }{2 L + 1},
\end{equation*}
while the orthogonality property states that
\begin{equation*}
\sum_{L \in \mathbb{N}_0} \sum_{M=-L}^{L} (2L + 1) \begin{pmatrix} \ell & \ell^\prime & L \\ m & m^\prime & M \end{pmatrix} \begin{pmatrix} \ell & \ell^\prime & L \\ \tilde{m} & \tilde{m}^\prime & M \end{pmatrix} = \delta_{m}^{\tilde{m}} \, \delta_{m^\prime}^{\tilde{m}^\prime}.
\end{equation*}

An important relation between spherical harmonic products and Wigner $3j$-symbols is provided by the Gaunt integral \cite[Remark~3.46]{marpecbook}
\begin{equation}\label{eq:gaunt}
\int_{\mathbb{S}^2} Y_{\ell m}(\mathbf{x}) Y_{\ell^\prime m^\prime}(\mathbf{x}) Y_{L M}(\mathbf{x}) d s(\mathbf{x}) = \sqrt{\frac{(2\ell + 1)(2\ell^\prime+ 1)(2L + 1)}{4\pi}} \begin{pmatrix} \ell & \ell^\prime & L \\ 0 & 0 & 0 \end{pmatrix} \begin{pmatrix} \ell & \ell^\prime & L \\ m & m^\prime & M \end{pmatrix}.
\end{equation}

\section{Wigner matrices}\label{appC}
Each element of $SO(3)$ can be represented as a $3\times 3$ real orthogonal matrix with determinant equal to $1$. Wigner $D$-matrices $\left\{D^\ell\right\}_{\ell\in \mathbb{N}_0}$ form a unitary representation of the rotation group and describe the action of rotations on spherical harmonics. If
$f\in L^2(\mathbb{S}^2)$ has harmonics coefficients $f_{\ell m}$, the harmonic coefficients of the rotated function are given by
    \begin{equation}\label{AA}
    (f(\rho^{-1}))_{\ell M} = \sum_{m=-\ell}^{\ell} D_{Mm}^{\ell}(\rho) f_{\ell m},
    \end{equation}
where $D_{Mm}^{\ell}\left(\rho\right)$ are entries of the Wigner $D$-matrix $D^\ell$, see \cite[(6.10)]{marpecbook}.

The Wigner $D$-matrices can also be represented in terms of the Wigner (azimuthal) $d$-matrices as follows
\begin{equation}\label{BB}
D_{Mm}^{\ell}\left(\alpha,\beta,\psi\right) = e^{-iM\alpha}d^{\ell}_{Mm}(\beta)e^{-im\psi},
\end{equation}
where $d^{\ell}_{M m}(\beta)$ is given by the following formula  from \cite[(3.20)]{marpecbook}
\begin{align}\label{dmml}
d^{\ell}_{M m}(\beta) & =  2^{-M}
(1 - \cos \beta)^{\frac{M + m}{2}} \,(1 + \cos \beta)^{\frac{M - m}{2}}\notag\\
&\quad \times \left( \frac{(\ell - M)!\,(\ell + M)!}{(\ell - m)!\,(\ell + m)!} \right)^{1/2}
P_{\ell - m}^{(M - m, \, M + m)}(\cos \beta),
\end{align}
$P_{n}^{(a,b)}(\cdot)$ are Jacobi polynomials with parameters $a$ and $b$.

The $d$-matrices encode rotations about the $y$-axis, while rotations about the $z$-axis are represented by the complex exponential factors involving the remaining Euler angles.

The paper uses the symmetry property, see~ \cite[Section~4.4]{vmk},
$$d_{mM}^{\ell}(\beta)=d_{Mm}^{\ell}(-\beta),$$
and the addition theorem for Wigner $D$-matrices
\begin{equation*}
	\sum_{M=-\ell}^{\ell}e^{-i M \psi^{\prime}}d^{\ell}_{mM}(\beta_1)d^{\ell}_{M,m^\prime}(\beta_2)= e^{-i m \alpha}d_{m m ^\prime}^{\ell}(\beta)e^{-i m ^\prime\psi},
\end{equation*}
where the Euler angles $(\alpha,\beta,\psi)$ are obtained from $\psi^{\prime},\beta_1,$ and $\beta_2$ by combining two rotations, see \cite[Section~4.7.2, (5) and (6)]{vmk}.



\acks  A. Olenko would like to express his gratitude for the support and hospitality provided by Sapienza Universit\`a di Roma during his sabbatical in~2024, that contributed to the initiation and development of this research.
 The authors also thank Prof. D. Marinucci for discussions on the current cosmological motivations for anisotropic spherical models.

\fund  A.Olenko and S.Khan were supported by the Australian Research Council's Discovery Projects funding scheme (project  DP220101680). C.Durastanti has been partially funded by Progetti di Ateneo Sapienza RG1221815C353275 (2022), RM12117A6212F538 (2021) and PRIN 2022 - GRAFIA - 202284Z9E4. 

\competing There were no competing interests to declare which arose during the preparation or publication process of this article.




\end{document}